\documentclass[12pt,twoside]{article}
\usepackage{amsmath,amsfonts,amssymb,amsthm}
\usepackage{graphicx}
\usepackage{hyperref}
\usepackage{color}
\def\XXint#1#2#3{{\setbox0=\hbox{$#1{#2#3}{\int}$}
\vcenter{\hbox{$#2#3$}}\kern-.5\wd0}}

\newcommand{\const}{\mbox{\em const}}

\newcommand{\grad}{\nabla}
\renewcommand{\P}{{\mathcal P}}
\newcommand{\laplace}{\Delta}

\renewcommand{\div}{\grad\cdot}
\newcommand{\N}{\mathbf{N}}
\renewcommand{\S}{\mathbf{S}}

\newcommand{\B}{{\mathcal B}}

\newcommand{\A}{{\mathcal A}}
\renewcommand{\and}{\quad\mbox{and}\quad}
\newcommand{\R}{\mathbf{R}}
\newcommand{\I}{\mathbf{1}}

\newcommand{\Z}{\mathbf{Z}}
\newcommand{\g}{ \mathsf{g}}
\newcommand{\M}{\mathcal M}
\newcommand{\cS}{\mathcal S}
\newcommand{\epsgap}{\eps_{\mathrm{gap}}}
\DeclareMathOperator{\tr}{tr}

\DeclareMathOperator{\Lip}{Lip}
\DeclareMathOperator{\Span}{span}

\DeclareMathOperator{\spt}{spt}

\def\loc{{\mathrm{loc}}}
\renewcommand{\L}{\ensuremath{\mathcal{L}}}
\newcommand{\Ha}{\ensuremath{\mathcal{H}}}

\newcommand{\mres}{\mathbin{\vrule height 1.6ex depth 0pt width
0.13ex\vrule height 0.13ex depth 0pt width 1.3ex}}

\newcommand{\eps}{\varepsilon}
\newcommand{\la}{\langle}
\newcommand{\ra}{\rangle}

\numberwithin{equation}{section}
\newtheorem{prop}{Proposition}[section]
\newtheorem{theorem}{Theorem}[section]
\newtheorem{lemma}{Lemma}[section]
\newtheorem{remark}{Remark}[section]
\newtheorem{definition}{Definition}[section]

\newcommand{\tacka}{\, \cdot\,}

\begin{document}

\title{Invariant manifolds for the porous medium equation}

\author{Christian Seis\footnote{Institut f\"ur Angewandte Mathematik, Universit\"at Bonn, Endenicher Allee 60, 53115 Bonn, Germany}}

\maketitle

\begin{abstract}
In this paper, we investigate the speed of convergence and higher-order asymptotics of solutions to the porous medium equation posed in $\R^N$. Applying a nonlinear change of variables, we rewrite the equation as a diffusion on a  fixed domain with quadratic nonlinearity. The degeneracy is cured by viewing the dynamics on a hypocycloidic manifold. It is in this framework that we can prove a differentiable dependency of solutions on the initial data, and thus, dynamical systems methods are applicable. Our main result is the construction of  invariant manifolds in the phase space of solutions which are tangent at the origin to the eigenspaces of the linearized equation. We show how these invariant manifolds can be used to extract information on the higher-order long-time asymptotic expansions of solutions  to the porous medium equation.
\end{abstract}

\section{Introduction}

\subsection{Motivation}

The long-time asymptotic behavior of solutions is of fundamental importance in the study of nonlinear diffusion processes. Besides giving insight into characteristic qualitative properties, the long-time limit often sets a benchmark for the solutions' expected regularity.
The prototype for nonlinear diffusions is the porous medium equation
\begin{equation}\label{1}
\partial_t u - \laplace u^m = 0
\end{equation}
posed in $\R^N$. Here $m$ is a constant that we assume to be positive for a moment. As the name suggests, this equation in well-known for modelling the flow of gas through a porous medium, but also other applications are of relevance, e.g., as models for groundwater infiltration, population dynamics or heat radiation in plasmas, cf.\ \cite[Chapter2]{Vazquez07}. Solutions of \eqref{1} are assumed to be nonnegative.

The qualitative behavior of solutions in general and the long-time behavior in particular differs for different values of $m$. For instance, for $m>1$, 
the diffusion flux $m u^{m-1}$ vanishes where $u=0$, and thus, if the initial configuration is compactly supported, the solution retains a compact support for all times. Hence, in such a situation, the porous medium equation features a free boundary value problem and the mass spreads slowly over the full space. We speak of ``slow diffusion''. The situation is completely different if $m<1$. Sticking to the mass-preserving range $\frac{N-2}N<m<1$, solutions to \eqref{1} become positive instantaneously and remain positive. In this range, \eqref{1} is called  the  ``fast diffusion equation''. 

The characteristic long-time behavior of solutions to the porous medium equation \eqref{1} was first identified by Kamin in one dimension \cite{Kamin73,Kamin75} and by Friedman, Kamin, and V\'azquez in several dimensions \cite{FriedmanKamin80,KaminVazquez88}: Solutions approach the self-similar Barenblatt profiles \cite{ZeldovicKompaneec50,Barenblatt52,Pattle59} as they dwindle away to nothing.
The optimal rate of convergence crucially depends on the choice of the initial datum. There are no uniform rates available without imposing a decay estimate on the initial datum at infinity, see \cite[Theorem 1.3]{Vazquez03}. In the framework of entropy solutions, sharp rates were first found by Carrillo and  Toscani \cite{CarrilloToscani00}, Otto \cite{Otto01} and del Pino and  Dolbeault \cite{DelPinoDolbeault02}.

In the fast diffusion range,  optimal rates of convergence and higher order asymptotic corrections have been the subject of recent interest. In \cite{DenzlerKochMcCann15}, Denzler, Koch and McCann compute a number of higher-order modes of the long-time asymptotic expansions building up on a spectral calculation of Denzler and  McCann \cite{DenzlerMcCann05} for the linearized dynamics. An overview on preliminary results is also provided in \cite{DenzlerKochMcCann15}. In the slow diffusion range, $m>1$, the only known results on higher order asymptotics are a full asymptotic expansion of Angenent \cite{Angenent88} in one space dimension, using the spectral calculation of Zel'dovich  and Barenblatt \cite{ZeldovicBarenblatt58}, and an improved rate of convergence after centering the data by V\'azquez \cite{Vazquez83}, which is also proved in the one-dimensional setting only. Regarding the multidimensional dynamics, a preliminary study was conducted by the author, who diagonalized the operator associated to linearized porous medium equation \cite{Seis14}.

In the present paper, we successfully relate the information from \cite{Seis14} to the nonlinear slow diffusion dynamics.
This is achieved by exploiting ideas that were originally developed for finite-dimensional dynamical systems. We prove that there exist finite-dimensional invariant manifolds in the phase space of the solutions, which are tangent at the origin to the eigenspaces of the linear operator, and every solution approaches these manifolds at a precise rate. In other words, we prove to what extent the porous medium dynamics can be characterized by the dynamics of a system of ordinary differential equations.

Invariant manifold theorems can be used to study the higher-order asymptotics for the porous medium equation: Once an order of decay is specified, our result ensures the existence of a finite-dimensional invariant manifold such that all solutions that are sufficiently close to the self-similar Barenblatt solutions approach that manifold with at least this rate. Thanks to the discreteness of the spectrum,  every eigenmode is accessible, and as a consequence, one can in principle compute the long-time asymptotics to any order. We will give four explicit applications in Theorems \ref{T13}--\ref{T12b} below.

The difficulty in making the dynamical systems approach work  lies in the incompatibility between the space in which the spectral calculation was performed and the space in which solutions depend differentiably on their initial data. A framework in which the latter can be achieved was first developed by Koch in his habilitation thesis \cite{Koch99}. The main obstacle for establishing this differentiable dependency on the initial data is the variable support of solutions. In general, the domains on which solutions live differ from the one of the Barenblatt solution --- the latter can be chosen stationary upon going over to self-similar variables. To overcome this, Koch suggests a nonlinear change of dependent and independent variables which results in a reparametrization of the graphs of the solutions. In the new coordinates, solutions live on the unit ball and Barenblatts become constants. The parabolicity of the resulting equation can be restored  by trading the Euclidean metric for the hypocycloid metric, that is, the metric on the unit ball which is obtained when measuring distances along hypocycloids.
It turns out that the solutions of the new equation depend even analytically on the initial data. To retrieve the information from \cite{Seis14}, we have to modify the problem slightly for solutions which are far away from the constant attractor.

The situation considered in the present paper differs from the one studied in \cite{DenzlerKochMcCann15} in many regards. Because solutions to the fast diffusion equation are positive everywhere, it suffices to consider the evolution in relative variables to establish differentiable dependency on the initial data. Parabolicity is restored by carrying out the analysis on the asymptotically cylindrical cigar manifold, and the well-posedness theory for the rescaled equation is fairly standard. The main difficulty in \cite{DenzlerKochMcCann15} is the occurrence of continuous spectrum found in \cite{DenzlerMcCann05}, above which only a finite number of eigenvalues are accessible. The authors need to work in weighted spaces to penetrate into the spectrum. The choice of norms has a severe drawback: The resulting Banach spaces are not Hilbert spaces and the linearized operators are no longer self-adjoint.

The present work  can be considered as a model study for  higher order asymptotic expansions for slow diffusions, such as the thin film equation with linear mobility 
\[
\partial_t u + \div\left(u\grad\laplace u\right) = 0.
\]
The global attractor is again a Barenblatt solution \cite{BernisPeletierWilliams92,FerreiraBernis97,CarrilloToscani02,MatthesMcCannSavare09}. A formal long time asymptotic expansion was obtained by Bernoff and Witelski \cite{BernoffWitelski02} in one space dimension, but no rigorous results on the higher order asymptotics are known. As a first step towards a rigorous understanding, building up on the spectral calculation in \cite{Seis14}, McCann and the author diagonalized the linear operator by exploiting the simple polynomial relation between the linear operators of the thin film equation and of the porous medium equation \cite{McCannSeis15}. In a slightly simpler context, John constructed solutions to the thin film equation that depend differentiably on the initial data  \cite{John15}.

\subsection{Examples for higher-order asymptotics}

The applications of the invariant manifold approach, Theorems \ref{T13}--\ref{T12b} below, are most conveniently stated for the   pressure variable in self-similar coordinates.  By virtue of a  standard rescaling argument, see, e.g., \cite[Section 1.1]{Seis14}, the porous medium equation can be equivalently written as a nonlinear Fokker--Planck equation
\begin{equation}
\label{1z}
\partial_t u - \laplace u^m -\div(xu)=0,
\end{equation}
which is often referred to as the ``confined'' porous medium equation. Under the same rescaling, the Barenblatt solutions $u_*(R,\tacka)$ become stationary,
\[
\frac{m}{m-1} u_{*}(R,x)^{m-1} = \frac12(R-|x|^2)_+,
\]
where $(f)_+ = \max\{f,0\}$, and the radius $R$ of its support is determined by the mass of the solution. For studying the regularity of solutions, it is convenient to write the equation for the so-called pressure variable $v = \frac{m}{m-1}u$. Upon changing the time scale from $t$ to $(m-1)t$ and setting $\sigma=\frac{2-m}{m-1}$, the equation for $v$ reads
\begin{equation}\label{1a}
\partial_t v - v\laplace v - (\sigma+1) \left(|\grad v|^2 + x\cdot \grad v\right) - Nv=0.
\end{equation}
Existence and uniqueness of solutions to the Cauchy problem for the porous medium equation is established in \cite{BenilanCrandallPierre84}. In \cite{CaffarelliVazquezWolanski87}, Caffarelli, V\'azquez and Wolanski show that nonnegative solutions with compactly supported initial value have Lipschitz pressures after a certain time. In view of the Barenblatt pressure, this result is optimal. However, due to the existence of focusing solutions \cite{AronsonGraveleau93,AngenentAronson95}, one cannot hope to have Lipschitz regularity at any small time.

%

To simplify the following discussion, we additionally normalize the ``mass'' of the initial configuration $v_0$ such that
\begin{equation}
\label{1az}
\int v_0^{\sigma+1} \, dx = \int v_*^{\sigma+1}\, dx,
\end{equation}
where $v_* = v_*(1,\tacka)$. Then the long-time asymptotics of the porous medium equation is governed by $v_*$. More precisely,
\begin{equation}
\label{1b}
\lim_{t\uparrow\infty} \|v(t)- v_*\|_{L^{\infty}} = 0.
\end{equation}
We refer to V\'azquez' survey \cite{Vazquez03} for a proof of \eqref{1b}.

In this paper, we are interested in the characteristic behavior of solutions for large times, or equivalently, close to the Barenblatt solution. We thus select our initial data in a neighborhood of $v_*$. More precisely, we shall assume that $v_0$ is Lipschitz with
\begin{equation}
\label{1c}
\|v_0 - \rho\|_{L^{\infty}(\P(v_0))}\le \delta_0\quad\mbox{and}\quad\|\grad(v_0-\rho)\|_{L^{\infty}(\P(v_0))}\le \eps_0
\end{equation}
for some $\delta_0,\, \eps_0>0$, and where, $\rho(x) = \frac12(1-|x|^2)$. Regarding the validity of our main theorems, we could have equally chosen $v_*$ over $\rho$ in \eqref{1c}. The actual formulation, however, slightly simplifies the proofs. Notice that \eqref{1c} entails that $\spt(v_0)\subset B_{1+\delta}(0)$. 

Restricting ourself to the multidimensional setting (the one-dimensional case is considered in \cite{Angenent88}), the spectrum of the linearized (and suitably transformed) equation consists purely of the eigenvalues
\[
\lambda_{\ell k} = (\sigma+1)(\ell + 2k) + k(2k+2\ell+N-2),
\]
where $(\ell,k)\in \N_0\times\N_0$, see Proposition \ref{P2} below. We relabel the eigenvalues in increasing order, that is, $\{\lambda_{\ell k}\}_{(\ell,k)\in \N_0\times \N_0} = \{\lambda_j\}_{j\in\N_0}$ with $\lambda_j<\lambda_{j+1}$.

The knowledge of the spectrum  (and the corresponding eigenvalues) of the linearized equation yields information on the higher-order asymptotics for \eqref{1a} via the Invariant Manifold Theorem \ref{T5} below: 
If $j\in \N_0$ and $\lambda\in (\lambda_j,\lambda_{j+1})$ are arbitrarily fixed, then the long-time asymptotic behavior of solutions up to terms of order $e^{-\lambda_j t}$ is given by a finite system of ODEs, which result from restricting a transformed version of \eqref{1a} to a finite-dimensional invariant manifold. This manifold is tangent at the origin to the eigenspaces of the first $j$ eigenvalues. In principle, we can compute the long-time asymptotics to any order. In the following, we make four explicit applications: We derive the exact rate of convergence to the Barenblatt solution  and compute the corrections in the long-time asymptotic expansions up to third order by restricting the dynamics to the  invariant (sub-)manifolds.

Our first result is a stability result for the Barenblatt pressure. 
\begin{theorem}[Stability of the Barenblatt pressure]
\label{T13}
There exists $\eps_0,\, \delta_0>0$ with the following property. If $v_0\in C^{0,1}$ satisfies \eqref{1az} and \eqref{1c} and $v$ is the solution to \eqref{1a} with initial value $v_0$, then
\[
\| v(t) - v_*\|_{L^{\infty}(\R^N)}\lesssim e^{-(\sigma+1)t}\quad\mbox{for all }t\ge 0.
\]
Moreover, for any $k\in \N_0$ and $\beta\in \N_0^N$, it holds
\[
\| \partial_t^k\partial_x^{\beta} \left(v(t) - \rho\right)\|_{L^{\infty}(\P(v(t)))}\lesssim e^{-(\sigma+1)t}\quad\mbox{for all }t\ge 1.
\]
\end{theorem}
The same decay behavior in any spatial $C^k$ norm was already proved in \cite{LeeVazquez03} building up on Koch's smoothness results \cite{Koch99}. 

The rate of converge in Theorem \ref{T13} is sharp. It is saturated by spatial translations of the rescaled Barenblatt pressure $v_*$, that is, for any $b\in\R^N$, $v(t,x) = v_*(x- e^{-(\sigma+1)t}b)$ defines a solution to \eqref{1a} which approaches $v_*$ with rate $\sigma+1$.

The value $\lambda_{10}= \sigma+1$ is the smallest nonzero eigenvalue of the linearized operator, see  Proposition \ref{P2} below. The corresponding eigenfunctions generate  spatial translations, which leave the porous medium equation \eqref{1} invariant. The eigenvalue $\lambda_{00}=0$ corresponds to rescaling of mass, which leaves \eqref{1} invariant. Once the mass is fixed, see \eqref{1az}, this eigenvalues drops out of the spectrum \cite{Seis14}. 

The exact rate in Theorem \ref{T13} does not follow immediately from our invariant manifold theorem, Theorem \ref{T5} below, but requires a careful analysis of the nonlinear terms in the transformed equation, \eqref{13} below.
 
We obtain better rates of convergence when taking into account appropriate translations of the Barenblatt solution. In this way, we get rid of the eigenvalue $\sigma+1$ in the spectrum of the linearized operator, and the   decay rate is given by the next eigenvalue $2(\sigma+1)$.

\begin{theorem}[First order corrections --- spatial translations]
\label{T12}
Let $\lambda\in (\sigma+1,2(\sigma+1))$ be given. There exist $\eps_0,\,\delta_0>0$ with the following property: If $v_0\in C^{0,1}$ satisfies \eqref{1az} and \eqref{1c} and $v$ is the solution to \eqref{1a} with initial value $v_0$, then there exists a vector $b\in \R^N$ such that
\[
\|v(t) - v_*(\tacka - e^{-(\sigma +1)t}b)\|_{L^{\infty}(\R^N)}\lesssim e^{-\lambda t}\quad\mbox{for all }t\ge 0.
\]
Moreover, there exists a constant $C>0$ such that
\[
B_{1-r(t)}(0)\subset \P(v(t))\subset B_{1+r(t)}(0)\quad\mbox{for all }t\ge 0,
\]
where $r(t) = e^{-(\sigma+1)t}|b| + Ce^{-\lambda t}$.
\end{theorem}
This result is obtained by a straightforward application of the invariant manifolds constructed in Theorem \ref{T5} below, see Theorem \ref{T11} below, and a change of variables. For an analogous earlier result in one space dimension, we refer to \cite{Vazquez83}.

The rate $2(\sigma+1)$, which it is not exactly achieved within this framework, corresponds to affine transformations of the Barenblatt solution.  This is the first nontrivial eigenvalue as it does not reflect a symmetry of \eqref{1}. Again, choosing appropriate affine transformations and centering the solution at the origin, we improve the rate of convergence once more. 

\begin{theorem}[Second order corrections --- affine transformations]
\label{T12a}
Let $\lambda\in (2(\sigma+1),\min\{2(\sigma+1)+N, 3(\sigma+1)\})$ be given. Suppose that $\sigma >-\frac12$. There exist $\eps_0,\, \delta_0>0$ with the following property: If $v_0\in C^{0,1}$ satisfies \eqref{1az}, \eqref{1c} and 
\begin{equation}
\label{1d}
\int v_0(x)^{\sigma+1}x\, dx=0,
\end{equation}
and $v$ is the solution to \eqref{1a} with initial value $v_0$,
then there exists  a symmetric and trace-free matrix $A\in \R^N\times \R^N$ such that
\[
\|v(t) - v_*(x - e^{-2(\sigma+1)t} Ax)\|_{L^{\infty}}\lesssim e^{-\lambda t}\quad\mbox{for all }t\ge 0.
\]
Moreover, there exists a constant $C>0$ such that
\[
B_{1-r(t)}(0)\subset \P(v(t))\subset B_{1+r(t)}(0)\quad\mbox{for all }t\ge 0,
\]
where $r(t) = e^{-2(\sigma+1)t}|A| + Ce^{-\lambda t}$.
\end{theorem}

Hence, the rate of convergence is given by the next eigenvalue in line, namely $\min\{2(\sigma+1)+N, 3(\sigma+1)\}$. That is, we arrive at the first crossing of eigenvalues: $2(\sigma+1)+N$ corresponds to dilations or spatial rescaling of the Barenblatt solutions, and $3(\sigma+1)$ comes from pear-shaped transformations, see Figure 1 in \cite{Seis14}.

 Our approach allows for a geometric interpretation of \eqref{1az} and \eqref{1d}: By fixing the mass and centering the solution at the origin, we achieve that the first two eigenmodes are suppressed. In other words, the dynamics live on a submanifold of the  invariant manifold which is parallel to the third eigenspace.

The restriction to $\sigma>-\frac12$ is an artifact of the formulation of the invariant manifold theorem, and we will not make an attempt at getting rid of it in the present paper.

In our final application of the Invariant Manifold Theorem \ref{T5},  affine transformations are modded out and we focus on dilations.

\begin{theorem}[Third order corrections --- dilations]
\label{T12b}
Suppose that $N<\sigma+1$Let $\lambda\in (2(\sigma+1)+N, 3(\sigma+1))$ be given. There exist $\eps_0,\, \delta_0>$ with the following property: If $v_0\in C^{0,1}$ satisfies \eqref{1az}, \eqref{1c} and \eqref{1d} and
\begin{equation}
\label{1e}
\int v_0(x)^{\sigma+1} x\cdot Mx\, dx = 0
\end{equation}
for every trace-free matrix $M\in \R^{N\times N}$, and $v$ is the solution to \eqref{1a} with initial value $v_0$, then there exists a constant $c\in \R$ such that
\[
\|v(t) - R(t)^{2-\frac1{\alpha}} v_*\left(\frac{x}{R(t)}\right)\|_{L^{\infty}}\lesssim e^{-\lambda t}\quad\mbox{for all }t\ge T,
\]
where  $\alpha =\frac{\sigma+1}{2(\sigma+1)+N}$ and $R(t) = \left(1-c e^{-(2(\sigma+1) +N)t}\right)^{\alpha}$.
Moreover, there exists a constant $C>0$ such that
\[
B_{1-r(t)}(0)\subset \P(v(t))\subset B_{1+r(t)}(0)\quad\mbox{for all }t\ge 0,
\]
where $r(t) = e^{-(2(\sigma+1)+N)t}\alpha |c| + Ce^{-\lambda t}$.
\end{theorem}

In principle, similar statements on corrections of even higher order could be derived via the Invariant Manifold Theorem \ref{T5}, but the computations would get more and more involved the larger the eigenvalues were chosen.

The remainder of the paper is organized as follows: In the following section, we express the rescaled pressure equation 
\eqref{1a} in new variables that are suitable for our dynamical systems approach to work. We present a  theorem on well-posedness of the resulting equation and state our main theorem about invariant manifolds. Finally, we reformulate  Theorems \ref{T13}--\ref{T12b} in terms of the new variables. Section \ref{S:2} contains the proofs of Theorems \ref{T13}--\ref{T12b} as a consequence of the corresponding statements for the new variables. The rest of the paper is entirely devoted to the analysis of the transformed equation. In Section \ref{S:5}, we study its linear version and establish a well-posedness theory that is appropriate for addressing the nonlinear equation with a fixed point argument. This is the content of Section \ref{S:6}. The invariant manifold theorem is proved in Section \ref{S7}, and the applications are studied in Section \ref{S:8}.

\section{New variables and main theorem}\label{S:main}

The main difficulty of applying dynamical systems arguments to the porous medium equation \eqref{1} (or \eqref{1a}) consists of finding a setting in which solutions depend differentiably on the initial data. One problem appears to be the free moving boundary at which the regularity of  solutions breaks down and at which perturbations around the (stationary) Barenblatt solutions are hard to define. The  elegant mass-transport approach chosen in \cite{DenzlerMcCann05,Seis14} and which builds up on Otto's gradient flow interpretation \cite{Otto01}  overcomes this obstacle via Brenier's theorem  \cite{Brenier87,Brenier91}, but a differentiable dependency on the initial data using mass transport techniques seems out of reach.

A framework in which the latter can be achieved was introduced by Koch in his habilitation thesis \cite{Koch99}. Koch suggests a nonlinear change of dependent and independent variables which transforms the free boundary problem \eqref{1a} into a partial differential equation on a fixed domain. Moreover, in the new variables, Barenblatt solutions  $u_*(R,\tacka)$ are mapped onto the constant $R$.

With total mass fixed as in \eqref{1az} (i.e., $R=1$), at any time $t\ge0$, the transformation of the spatial coordinates reads
\begin{equation}\label{100}
z = \frac{x}{\sqrt{2 v(t,x) +|x|^2}}.
\end{equation}
It is clear that the new coordinate $z$ lies in the unit ball, and the transformation reduces to $z=x$ if $v$ is the Barenblatt solution $v_*$. Notice that $(z, \sqrt{2v_*(z)})$ is obtained by projecting $(x,\sqrt{2v(t,x)})$ orthogonally onto the graph of $\sqrt{2v_*}$, which is a half-sphere.   

We introduce the new dependent variables as perturbations around the Barenblatt solution. More precisely, if $w(t,z)$ denotes the distance between $(x,\sqrt{2v(t,x)})$ and the half-sphere, it holds:
\begin{equation}\label{101}
 1 + w(t,z) = \sqrt{2v(t,x) + |x|^2}.
 \end{equation}
As announced above, this change of variables transforms the Barenblatt solutions $v_*$ into the constant function $1$, or, in other words, perturbations are constantly zero, i.e., $w\equiv0$. The change of variables will be made rigorous in Lemmas \ref{L20} and \ref{L20a} below. These two lemmas in particular show that there is a diffeomorphism between the graph of  $v$ and the graph of $w$. In  Figure \ref{fig}
\begin{figure}[t]
   \centering
\scalebox{0.55}{\includegraphics{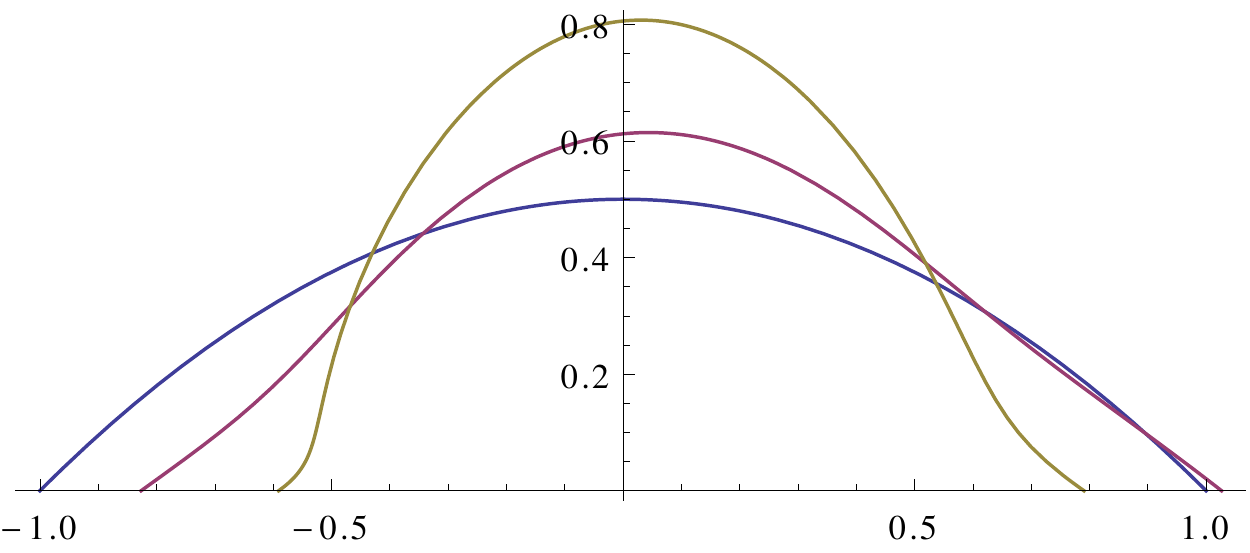}}
	\scalebox{0.55}{\includegraphics{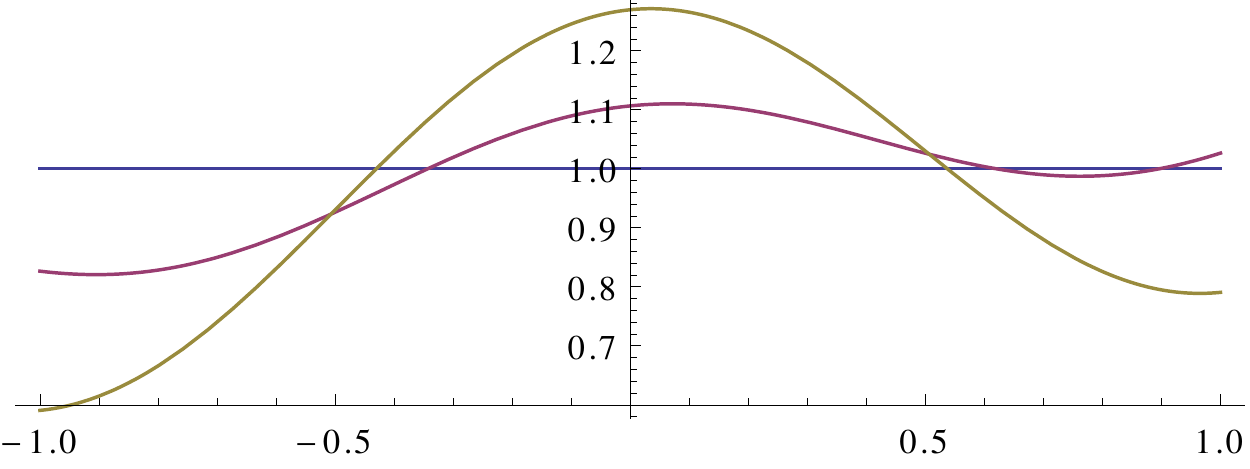}}
\caption{The plot on the left shows the graph of the Barenblatt pressure (blue curve) and two perturbations in the $(x,v)$-plane. The plot on the right shows the corresponding graphs in the $(z,1+w)$-plane.\label{fig}}
\end{figure}
above, we illustrate this change  of variables around the Barenblatt solution. 

For further references, we note that the inverse change of variables reads
\begin{eqnarray}
x &=& (w(t,z) +1)z,\label{45a}\\
v(t,x)- \rho(x)  &=& w(t,z) + \frac12w(t,z)^2.\label{45b}
\end{eqnarray}

The transformed equation for the $w$-variables is
\begin{eqnarray}
\lefteqn{\partial_t w - \rho^{-\sigma}\div\left(\rho^{\sigma+1}\grad w\right) }\nonumber \\
&=& (N+2\sigma+1) \rho \frac{|\grad w|^2}{1+w+z\cdot \grad w}  -\rho^{-\sigma}\div\left(\rho^{\sigma+1} z\frac{|\grad w^2|}{1+w+z\cdot \grad w}\right),\label{13}
\end{eqnarray}
see also Lemma \ref{L20a} below.
Since $w$ represents the perturbation around the constant (former Barenblatt) solution, we will often call \eqref{13} the (nonlinear) {\em perturbation equation}. 

In \cite{Koch99}, Koch develops a regularity theory for equations of this type in H\"older spaces. For a slightly simpler version of \eqref{13} in which the nonlinearity is independent of the function itself (and thus only depends on derivatives), Kienzler \cite{Kienzler14} improves on Koch's results and establishes a theory for small  Lipschitz initial data. Building up on Kienzler's techniques, we will construct smooth solutions to \eqref{13} for small initial data in $C^{0,1}$. These solutions are uniformly controlled in any $C^k$ norm by the Lipschitz norm of the initial datum, which prevents the nonlinear terms from degenerating. In addition, solutions depend analytically on the initial datum. More precisely, we prove

\begin{theorem}[Well-posedness for the perturbation equation]\label{T1}
There exists $\eps>0$ and $\delta>0$ such that for every $g\in \B_{\eps,\delta}$, there exists a unique solution $w$ to \eqref{13}. This solution  is smooth and depends analytically on $g$. Moreover,  $\|w\|_{L^{\infty}} + \|w\|_{\Lip}<1$ and
\begin{equation}
\label{7az}
t^{k + |\beta|} \left|\partial_t^k \partial_z^{\beta} \grad w(t,z)\right| \lesssim \|g\|_{\Lip},
\end{equation}
for any $(t,z)\in (0,\infty)\times B_1(0)$ and $k\in\N_0$ and $\beta\in \N_0^N$.
\end{theorem}

Here and in the following, we use the notation
\[
\B_{\eps,\delta}: = \left\{ w\in C^{0,1}:\: \|w\|_{L^{\infty}}\le \delta,\, \|w\|_{\Lip}\le \eps\right\}.
\]

The analyticity of the solution is a byproduct of the fixed point argument applied in the construction of solutions of the nonlinear equation. The underlying idea has been first used by Angenent \cite{Angenent90b,Angenent90a} and has been further exploited by Koch and  Lamm \cite{KochLamm12}.

Before stating our results on invariant manifolds and long-time asymptotics, we need some preparations. Let $\Sigma(\L) = \{\lambda_{k}\}_{k\in\N_0}$ be the (discrete) spectrum of the linear operator $\L = -\rho\laplace +(\sigma+1)z\cdot \grad$ on the Hilbert space $L^2_{\sigma} = L^2(d\mu_{\sigma})$, where $d\mu_{\sigma} = \rho^{\sigma}dx$. See Proposition \ref{P2} below for more details on the spectrum of $\L$. Suppose that the eigenvalues are in increasing order and not repeated, so that $\lambda_k < \lambda_{k+1}$. 
For $K\in\N_0$ arbitrary, consider the subspace $E_c$ of $L^2_{\sigma}$ that is spanned by the eigenfunctions corresponding to the first $K$ eigenvalues, and denote by $E_s$ its orthogonal complement, i.e., $L^2_{\sigma} = E_c\oplus E_s$. Let $\Lambda$ and $\epsgap$ be positive constants so that $\Lambda\in \left(e^{-(K+1)} +\epsgap, e^{-k}-\epsgap\right)$. The semi-flow associated to \eqref{13} will be denoted by $S^t$, that is $S^t(g) = w(t)$ if $w$ denotes the solution of \eqref{13} with initial value $g$ given by Theorem \ref{T1} above.

We are now in the position to state our main result.

\begin{theorem}[Invariant manifolds]
\label{T5}
There exist constants $\eps,\, \delta>0$ and $\tilde \eps,\, \tilde \delta>0$, a function $\theta: E_c\to E_s$ and manifolds $W_c = \left\{ g_c + \theta(g_c):\: g_c\in E_c\right\}$ and $W^{\loc}_c  := W_c \cap \B_{\eps,\delta}$ with the following properties:
\begin{enumerate}
\item The function $\theta$ is a differentiable map from $\B_{\eps,\delta}$ to $C^{0,1}$ satisfying $\theta(0) = 0$, $D \theta(0) = 0$ and $\|\theta(g_c)\|_{C^{0,1}}\lesssim \|g_c\|_{C^{0,1}}^2$ for all $g_c\in E_c\cap \B_{\eps,\delta}$.
\item If $g\in W_c\cap \B_{\tilde \eps,\tilde \delta}$, then $S^t(g)\in W_c^{\loc}$ for all $t\ge0$.
Moreover, there exists a unique $\tilde g\in W_c^{\loc}$ such that $S^t(\tilde g)\in W_c^{\loc}$ for all $t\ge 0$ and 
\[
\|S^t(g) - S^t(\tilde g)\|_{C^1}\lesssim \Lambda^t \quad\mbox{for all }t\ge 0.
\]
\item If $\{w(t)\}_{t\le 0}$ is a negative semi-orbit of \eqref{13} with  $g\in W_c$ and $w(t)\in\B_{\eps,\delta}$ for all $t\le 0$, then $w(t)\in W_c^{\loc}$ for all $t\le 0$. 
\end{enumerate}
\end{theorem}
The first property shows that the local center manifold $W_c^{\loc}$ is tangent at the origin to the eigenspace $E_c$.
The second property shows that  $W_c^{\loc}$ is locally flow invariant under the semi-flow $\{S(t)\}_{t\ge0}$ for the nonlinear perturbation equation \eqref{42a} and shows to what degree solutions to the perturbation equation \eqref{13} are characterized by a system of ODEs.

We finally give four applications of the invariant manifold theorem. In the first example, we study the long-time limit of small data solutions of the perturbation equation.

\begin{theorem}[Stability]
\label{T10}
There exists positive constants $\tilde \eps$ and $\tilde \delta$ with the following properties: If $g\in \B_{\tilde \eps,\tilde \delta}$ and  $w$ solves \eqref{13} with initial datum $g$, then there exists a constant $a\in\R$ such that for all $k\in\N_0$ and $\beta\in \N_0^N$ it holds
\begin{equation}\label{7aza}
\|\partial_t^k\partial_z^{\beta}\left(w(t)-a\right)\|_{L^{\infty}} \lesssim e^{-(\sigma+1)t}\quad\mbox{for all }t\ge 1.
\end{equation}
\end{theorem}

The rate of convergence in the above theorem is sharp. It is saturated by spatial translations. The exact rate does not follow automatically from Theorem \ref{T5}, but requires a detailed analysis of the nonlinear term.

Notice that any constant (except $1$) is a solution to the perturbation equation \eqref{13}. As a consequence, the constant long-time limit is not necessarily zero. On the level of the porous medium equation \eqref{1}, this corresponds to the invariance of solutions under the rescaling of mass. Once the mass is fixed as in \eqref{1az}, this constant drops out in \eqref{7aza}. In this way, we obtain the first order corrections in the long-time limit of \eqref{13}.

\begin{theorem}[First order corrections --- spatial translations]
\label{T11}Let $\lambda\in (\sigma+1,2(\sigma+1))$ be given.
There exists positive constants $\tilde \eps$ and $\tilde \delta$ with the following properties: If $g\in B_{\tilde \eps,\tilde \delta}$ and $w$ solves \eqref{13} with initial datum $g$ and satisfies
\[
\lim_{t\uparrow\infty}\int  w(t)\, d\mu_{\sigma}=0,
\]
then there exists a vector  $b\in \R^N$   such that
\[
\|w(t)-e^{-(\sigma+1)t}b\cdot z \|_{C^{0,1}} \lesssim e^{-\lambda t}\quad\mbox{for all }t\ge0.
\]
\end{theorem}

The rate of convergence can be further improved when the center of mass is asymptotically fixed at the origin. In this case, spatial translations are suppressed and the dynamics are governed by affine transformations. Thus, we climb the eigenvalue ladder one rung up:

\begin{theorem}[Second order corrections --- affine transformations]\label{T17}
Let  $\lambda\in (2(\sigma+1),\min\{2(\sigma+1)+N), 3(\sigma+1)\})$ be given.
There exists positive constants $\tilde \eps$ and $\tilde \delta$ with the following properties: If $g\in B_{\tilde \eps,\tilde \delta}$ and $w$ solves \eqref{13} with initial datum $g$ and satisfies
\[
\lim_{t\uparrow\infty} \int w(t)\, d\mu_{\sigma}=0\quad\mbox{and}\quad \lim_{t\uparrow\infty} e^{(\sigma+1)t}\int w(t) z\, d\mu_{\sigma}=0.
\]
Then there exists a symmetric and trace-free matrix $A\in \R^{\N\times N}$ such that
\[
\|w(t)-e^{-2(\sigma+1)t}z\cdot Az \|_{C^{0,1}} \lesssim e^{-\lambda t}\quad\mbox{for all }t\gtrsim 1.
\]
\end{theorem}

We have a fist eigenvalue crossing when $N=\sigma+1$: 
It holds that $\lambda_{3}= \min\{2(\sigma+1)+N,3(\sigma+1)\}$. On the level of the porous medium equation, the eigenvalue $2(\sigma+1)+N$ corresponds to dilations,  which is, by the invariance of the porous medium equation under time shifts, an artifact of the dynamics. We can access the corresponding eigenvalue by moding out certain second order moments that are preserved under \eqref{1}. On the level of the perturbation equation \eqref{13}, this is achieved in the following

\begin{theorem}[Third order corrections --- dilations]
\label{T14}
Suppose that $N<\sigma +1$. Let $\lambda\in (2(\sigma+1)+N, 3(\sigma+1))$ be given.
There exists positive constants $\tilde \eps$ and $\tilde \delta$ with the following properties: If $g\in B_{\tilde \eps,\tilde \delta}$ and $w$ solves \eqref{13} with initial datum $g$ and satisfies
\begin{gather*}
\lim_{t\uparrow\infty} \int w(t)\, d\mu_{\sigma} =0,\\
 \lim_{t\uparrow\infty} e^{(\sigma+1)t}\int w(t) z\, d\mu_{\sigma}=0,\\
   \lim_{t\uparrow\infty}e^{2(\sigma+1)t } \int w(t) M\!:\!z\otimes z  \, d\mu_{\sigma} =0
\end{gather*}
for every trace-free matrix $M\in \R^{N\times N}$. Then there exists $c\in \R$ such that
\[
\|w(t)-e^{-2((\sigma+1)+N)t}c(1-\gamma|\tacka|^2) \|_{C^{0,1}} \lesssim e^{-\lambda t}\quad\mbox{for all }t\gtrsim 1,
\]
where $\gamma = N^{-1}\left(2(\sigma+1)+N\right)$.
\end{theorem}

In principle, every eigenvalue is accessible through the Invariant Manifold Theorem \ref{T5}. However, as the complexity of the corresponding eigenfunctions increases with every eigenvalue, we refrain from studying further examples.

\newpage

\section{Higher order asymptotics for the pressure variable}\label{S:2}

In the present section, we show how the statements on the long-time asymptotics for the perturbation equation \eqref{13}, Theorems \ref{T10}--\ref{T14}, imply the corresponding statements on the long-time asymptotics for the confined pressure equation \eqref{1a}, Theorems \ref{T13}--\ref{T12b}.

It is convenient to recall and introduce some notation. 
For any function $v:[0,\infty)\times \R^N\to \R$ we denote by $\P(v)$ its positivity set and by $\P(v(t))$ its positivity set at time $t$, that is, $\P(v) = \left\{(t,x)\in [0,\infty)\times \R^N:\: v(t,x)>0\right\}$ and $\P(v(t)) = \{x\in\R^N:\: v(t,x)>0\}$. 

In a first step, we show that the transformation \eqref{100}, \eqref{101} is well-defined and show how smallness conditions for $v_0$ in \eqref{1c} turn into smallness conditions for the new variables.

\begin{lemma}[Change of coordinates I]\label{L20}
Suppose that $v_0\in C^{0,1}$ satisfies \eqref{1c} for some $\eps_0,\, \delta_0>0$ with $\eps_0+2\delta_0<1$. 
Let $\Phi_0: \P(v_0)\to  B_1(0)$ be given by 
\[
\Phi_0(x) = \frac{x}{\sqrt{2v_0(x) + |x|^2}},
\]
for all $x\in\P(v_0)$. Then the following holds:
\begin{enumerate}
\item The mapping $\Phi_0$ is a diffeomorphism.
\item Let $g: B_1(0)\to \R$ be given by
\[
g(\Phi_0(x))  = \sqrt{2v_0(x) +|x|^2} - 1,
\]
for all $x\in \P(v_0)$. Then 
\[
\|g\|_{L^{\infty}}\le 2\delta_0\quad\mbox{and}\quad
\|g\|_{\Lip} \le \frac{1+2\delta_0}{1-2\delta_0 - \eps_0}\eps_0.
\]
\end{enumerate}
\end{lemma}

\begin{proof}
1. The Jacobi matrix of $\Phi_0$ is given by
\[
\grad \Phi_0(x) = \frac{\I}{\sqrt{2v_0(x) +|x|^2}}  - \frac{x\otimes\left(\grad v_0(x) + x\right)}{\left(2 v_0(x)+|x|^2\right)^{3/2}},
\]
where $\I$ denotes the identity matrix in $\R^N\times \R^N$. We compute its Jacobian determinant:
\begin{eqnarray*}
\det \grad\Phi_0(x) &= & \left(2 v_0(x) +|x|^2\right)^{-\frac{N}2}\det \left(\I - \frac{x\otimes\left(\grad v_0(x) + x\right)}{2v_0(x) + |x|^2}\right) \\
&=&\left(2v_0(x) + |x|^2\right)^{-\frac{N}2-1} \left(2v_0(x) - x\cdot \grad v_0(x)\right),
\end{eqnarray*}
where the second identity follows from the auxiliary formula\footnote{
This formula immediately follows from the calculation
\[
\left(\begin{array}{cc} \I & 0\\ b^T &1 \end{array}\right)\left(\begin{array}{cc} \I + a\otimes b & a\\0 &1 \end{array}\right)\left(\begin{array}{cc} \I & 0\\-b^T &1 \end{array}\right) = \left(\begin{array}{cc} \I & a\\0 &1 + a\cdot b \end{array}\right),
\]
}
$\det \left(\I + a\otimes b\right) = 1 + a\cdot b$.
Because $2v_0 + |x|^2 = 2(v_0-\rho) +1$ and $2v_0 +x\cdot \grad v_0 = 2(v_0-\rho) -x\cdot\grad(v_0-\rho) +1$, we deduce from \eqref{1c} that the Jacobian  is positive on the positivity set of $v_0$. Hence $\Phi_0$ is a diffeomorphism.

2. Notice that $g$ is well-defined because $\Phi_0$ is a diffeomorphism. The uniform bound on $g$ is a consequence of the first assumption on $v_0$ in \eqref{1c}. For the Lipschitz bound, notice that 
\[
\grad v_0(x) + x = \frac{1+g(z)}{1+g(z)+z\cdot\grad g(z)} \grad g(z)
\]
if $z = \Phi_0(x)$. Therefore, multiplying by $x = (1+g)z$ gives
\[
\frac{1+g}{1+g+z\cdot \grad g}= 1-\frac{x\cdot \grad v_0 + |x|^2}{2v_0 +|x|^2}
 = \frac{2v_0 - x\cdot \grad v_0}{2v_0 +|x|^2},
\]
and the statement follows via \eqref{1c}.

%

\end{proof}

The following lemma justifies the inverse transformation \eqref{45a}, \eqref{45b}. Moreover, we show that under this transformation, solutions to the perturbation equation \eqref{13} become solutions to the confined pressure equation \eqref{1a}. Thanks to the uniqueness of the associated initial value problems, both equations are thus equivalent.

\begin{lemma}[Change of variables II]\label{L20a}
Let $\eps,\, \delta$ be as in Theorem \ref{T1}. Suppose  $w$ is the solution to \eqref{13} with initial value $g\in \B_{\eps,\delta}$. Consider $\Psi :(0,\infty)\times B_1(0)\to (0,\infty)\times \R^N$ given by
\[
\Psi(t,z) = (t, (1+w(t,z))z),\quad\mbox{for }(t,z)\in (0,\infty)\times B_1(0).
\]
Then the following holds:
\begin{enumerate}
\item $\Psi$ is a diffeomorphism onto its image.
\item Let $v$ be defined by
\[
v(t,x) = \left\{\begin{array}{ll}\rho(z)(1+w(t,z))^2 & \mbox{if }(t,x) = \Psi(t,z) \mbox{ for some } z\in B_1(0),\\ 0 &\mbox{otherwise}.\end{array}\right.
\]
Then $v$ is $ C^{0,1}$ and solves the confined pressure equation \eqref{1a}.
\end{enumerate}
\end{lemma}

\begin{proof}
1. We start with the computation of the Jacobian matrix of $\Psi$, 
\[
(\partial_t\Psi(t,z),\grad\Psi(t,z)) = \left(\begin{array}{cc} 1 & 0\\ \partial_t(t,z) z& z\otimes \grad w(t,z) + (1+w(t,z))\I\end{array}\right).
\]
Arguing similarly as in the proof of the previous lemma, we deduce that
\[
\det  (\partial_t\Psi(t,z),\grad\Psi(t,z)) = (1+w(t,z))^{N-1}\left(1 + w(t,z) + z\cdot \grad w(t,z)\right),
\]
and the expression on the right is positive by Theorem \ref{T1}. Thus $\Psi$ is a diffeomorphism onto is image.

2. It is clear that $v$ is well-defined, bounded and  Lipschitz on $\R^N$ because $w$ is bounded and Lipschitz. To verify that $v$ indeed solves \eqref{1a}, we give a quite formal argument. Starting point is the identity
\[
\sqrt{2v(t,y) +  |x|^2} = 1+  w\left(t, \frac{x}{\sqrt{2v(t,x) +|x|^2}}\right).
\]
A multiple use of the elementary differentiation rules yields 
\begin{eqnarray*}
\left(1+ w  + z\cdot \grad  w\right)\partial_t v  &=& (1+ w)^2\partial_t  w,\\
\left(1+ w + z\cdot \grad  w\right)\left(\grad v +x\right)& =& (1+ w)\grad  w,
\end{eqnarray*}
and
\begin{eqnarray*}
\lefteqn{\left(1+ w + z\cdot \grad  w\right)\left(\laplace v + N\right) }\\
&=& \laplace  w - \frac{2 z\cdot \left(D^2  w\grad  w\right)+ |\grad  w|^2}{1+ w+z\cdot \grad  w} + \frac{|\grad  w|^2\left(  2z\cdot \grad  w + z\otimes z: D^2 w\right)}{\left(1+ w+z\cdot \grad  w\right)^2}.
\end{eqnarray*}
Noting that
\[
z\cdot \grad\left( \frac{|\grad  w|^2}{1+ w+z\cdot \grad  w}\right) = \frac{2z\cdot\left(D^2 w\grad  w\right)}{1+ w+z\cdot \grad  w} - \frac{|\grad  w|^2}{1+ w+z\cdot \grad  w}\left(2z\cdot \grad  w + z\otimes z:D^2 w\right),
\]
the previous formula simplifies to
\[
 \left(1+ w + z\cdot \grad  w\right)\left(\laplace v + N\right)
=\laplace  w + (N-1) \frac{|\grad  w|^2}{1+ w+ z\cdot \grad  w} - \div\left(z \frac{|\grad  w|^2}{1+ w+ z\cdot \grad  w}\right).
\]	
Multiplying both sides of the equation by $(1+ w)^{-2} v = \rho(z)$, a short computation leads us to
\begin{eqnarray*}
\lefteqn{(1+ w)^{-2}\left(1+ w + z\cdot \grad  w\right)\left(v\laplace v + Nv\right)}\\
&=& \div\left(\rho\left(\grad  w - z\frac{|\grad  w|^2}{1+ w+z\cdot \grad  w}\right)\right) \\
&&\mbox{} + z\cdot \grad  w - |z|^2 \frac{|\grad  w|^2}{1+ w+z\cdot \grad  w} +(N-1) \rho \frac{|\grad  w|^2}{1+ w+z\cdot \grad  w}.
\end{eqnarray*}
Without much effort we can now derive \eqref{1a}.
\end{proof}

By the virtue of the preceding two lemmas, we are now in the position to deduce Theorems \ref{T13}--\ref{T12b} directly from  Theorems \ref{T10}--\ref{T14}.

\begin{proof}[Proof of Theorem \ref{T13}]
Let $\tilde \eps$ and $\tilde \delta$ be as in Theorem \ref{T10}. Let $g$ be defined as in Lemma \ref{L20}. Then $g\in\B_{\tilde \eps,\tilde \delta}$ is $\eps_0$ and $\delta_0$ are chosen sufficiently small. Let $w$ be the solution to \eqref{13} with initial value $g$. By Theorem \ref{T10}, there exists a constant $a\in \R$ such that
\begin{eqnarray}
\| w(t) - a\|_{L^{\infty}}&\lesssim & e^{-(\sigma+1)t}\quad\mbox{for all }t\ge0,\label{45cz},\\
\|\partial_t^k \partial_z^{\beta} (w(t) - a)\|_{L^{\infty}}&\lesssim & e^{-(\sigma+1)t}\quad\mbox{for all }t\ge1,\label{45c},
\end{eqnarray}
for all $k\in \N_0$ and $\beta\in\N_0^N$. In particular, $w(t) \to a$ for large times. Since on the other hand
\begin{equation}
\label{45d}
v-\rho = w + \frac12 w^2 \quad\mbox{in }\P(v),
\end{equation}
by Lemma \ref{L20a}, and $v\to v_* = (\rho)_+$ by \eqref{1b}, we must have that $a^2+a=0$, and thus $a=0$ since $\|w\|_{L^{\infty}}\le \tilde \delta <1$.
It immediately follows from \eqref{45cz} and \eqref{45d} that
\begin{equation}
\label{45ez}
|v(t,x) - \rho(x)| \lesssim e^{-(\sigma+1)t}\quad\mbox{for all }x\in \P(v(t)), \, t\ge 0.
\end{equation}
This estimate in particular implies that
\[
1 - Ce^{-(\sigma+1)t} \le |x|\le 1+C e^{-(\sigma+1)t}
\]
for all $x\in \partial \P(v(t))$ and $t\ge0$ and some $C>0$. Therefore $B_{1 - Ce^{-(\sigma+1)t}}(0)\subset \P(v(t))\subset B_{1 + Ce^{-(\sigma+1)t}}(0)$. 
Moreover, in \eqref{45ez}, we
 can easily replace $\rho$ by $v_*=(\rho)_+$ because $v-v_* = v\le v-\rho$ outside of $B_1(0)$. On the other hand, since $B_1(0)\setminus \P(v(t)) \subset B_1(0)\setminus B_{1 - Ce^{-(\sigma+1)t}}(0)$, it holds that $|v_*(x)|\lesssim e^{-(\sigma+1)t}$ outside of $\P(v(t))$. In any case
 \[
 |v(t,x) - v_*(x)|\lesssim e^{-(\sigma +1)t} \quad\mbox{for all }t\ge0.
 \]
To deduce the corresponding statement for higher order derivatives, we argue as follows. From the proof of Lemma \ref{L20a} we recall that
\[
\partial_i(v-\rho) = \frac{1+w}{1+w+z\cdot \grad w}\partial_i w.
\]
The expression on the right-hand side and all of its derivatives with respect to $z$ are bounded by $e^{-(\sigma+1)t}$ thanks to \eqref{45c}. This implies the desired statement for $(k,|\beta|) = (0,1)$. Moreover, derivatives of the diffeomorphism
\[
z= \Phi_t(x)  = \frac{x}{\sqrt{2v(t,x) + |x|^2}}
\]
are uniformly  bounded, for instance,
\[
\grad_x \Phi_t = \frac{\I}{\sqrt{2v + |x|^2}} - \frac{x\otimes (\grad v+ x)}{\sqrt{2v + |x|^2}^3} = \frac{\I - z\otimes \grad w}{1+w} ,
\]
and the control of higher order derivatives follows via the chain rule and \eqref{45c}. Hence, for any $\beta\in \N_0^N$, it holds that 
\[
\|\partial_x^{\beta} (v-\rho)\|_{L^{\infty}(\P(v(t))}\lesssim e^{-(\sigma+1)t}\quad\mbox{for all }t\ge0,
\]
and the control of the temporal derivatives follows upon using the equation \eqref{1a} in the form
\[
\partial_t v = v\laplace(v-\rho) +(\sigma+1) \grad v\cdot \grad(v-\rho).
\]
This concludes the proof of Theorem \ref{T13}.
\end{proof}

\begin{proof}[Proof of Theorem \ref{T12}]
We let $\tilde \eps$ and $\tilde \delta$ be as in Theorem \ref{T11}, and define $g$ as in Lemma \ref{L20}. We suppose that $\eps_0$ and $\delta_0$ are small enough so that $g\in \B_{\tilde \eps,\tilde \delta}$. From the statement of Theorem \ref{T10} and the proof of Theorem \ref{T13}, we deduce that
\begin{equation}\label{45e}
\|w(t)\|_{L^{\infty}}\lesssim e^{-(\sigma+1)t}\quad\mbox{for all }t\ge 0,
\end{equation}
and thus,
\[
\lim_{t\uparrow\infty} \int w(t)\, d\mu_{\sigma} = 0.
\]
Consequently, thanks to Theorem \ref{T11} there exists a vector $b\in\R^N$ such that
\begin{equation}
\label{45f}
\|w(t) - e^{-(\sigma+1)t} b\cdot z\|_{C^{0,1}}\lesssim e^{-\lambda t}\quad\mbox{for all }t\ge 0.
\end{equation}
Thus, Lemma \ref{L20a} and \eqref{45e} imply that
\[
v(t,x) - \rho(x-e^{-(\sigma+1)t} b) = w(t,z) -e^{-(\sigma+1)t}b\cdot z +  O(e^{-\lambda t}),
\]
for all $(t,z)\in\P(v)$, and the right-hand side is of order $e^{-\lambda t}$ by \eqref{45f}. Finally, as in the proof of Theorem \ref{T13}, we can easily switch from $\rho$ to $v_* = (\rho)_+$ on the left-hand side, extent the statement to all $x\in \R^N$ and derive estimates on the positivity set of $v$. 
\end{proof}

\begin{proof}[Proof of Theorem \ref{T12a}]
We start recalling the well-known fact that the confined porous medium equation \eqref{1z} preserves zero-center of mass. Indeed, a short computation (performed on the level of \eqref{1a}) reveals that
\[
\frac{d}{dt} \int v^{\sigma+1} x_i\, dx = -(\sigma+1) \int v^{\sigma+1} x_i\, dx,
\]
and thus, the centering condition \eqref{1d} is valid for all times. 
Let $\tilde \eps$ and $\tilde \delta$ be as in Theorem \ref{T17}, and define $g$ as in Lemma \ref{L20}. We choose $\eps_0$ and $\delta_0$ small enough so that $g\in \B_{\tilde \eps,\tilde \delta}$. Let $\tilde \lambda\in (\max\{\sigma+1,1\}2(\sigma+1))$. By Theorem \ref{T11}, there exists a vector $b\in \R^N$ such that
\begin{equation}
\label{45g}
|v(t,x) - v_*(x-e^{-(\sigma+1)t}b)|\lesssim e^{-\tilde \lambda t}\quad\mbox{for all }x\in \R^N,t\ge0.
\end{equation}
We claim that $b=0$. Indeed, by the symmetry of $v_*$, it holds that
\[
e^{-(\sigma+1)t} |b| \sim \left|\int v_*(x)^{\sigma+1} e^{-(\sigma+1)t }b\, dx\right| = \left|\int v_*(x)^{\sigma+1} \left(x+ e^{-(\sigma+1)t }b\right)\, dx\right|.
\]
By a change of variables and because $v$ is centered at zero, we further have that
\[
e^{-(\sigma+1)t} |b| \lesssim \int \left|v_*(x - e^{-(\sigma+1)t} b)^{\sigma+1} - v(t,x)^{\sigma+1}\right| |x|\, dx.
\]
In the case $\sigma\ge0$, an application of \eqref{45g} yields $|b|\lesssim e^{-(\tilde \lambda - (\sigma+1))t} $, and thus, letting $t\uparrow \infty$, we see that $b=0$.
Likewise, if $-\frac12<\sigma<1$, we apply \eqref{45g} and obtain $|b|\lesssim e^{-(\tilde \lambda-1)(\sigma+1)t} $, and again, $b=0$ follows upon passing to the limit in $t$.
Because $b=0$, it holds that
\[
\lim_{t\uparrow\infty} e^{(\sigma+1)t} \int w(t,z)z\, d\mu_{\sigma} = 0,
\]
and thus, Theorem \ref{T17} is applicable. The argument of how to deduce Theorem \ref{T12a} from Theorem \ref{T17} proceeds similarly as in the proof of the previous theorem. We omit the details.
%
\end{proof}

\begin{proof}[Proof of Theorem \ref{T12b}]
The proof is conceptionally very similar to the proof of the previous theorem. We notice first that the second moment condition in \eqref{1e} is preserved under the evolution \eqref{1a} because
\[
\frac{d}{dt} \int v^{\sigma+1}x\cdot Mx\, dx = -2(\sigma+1)\int v^{\sigma+1} x\cdot Mx\, dx
\]
for every trace-free matrix $M\in\R^{N\times N}$. Let $\tilde \eps,\, \tilde \delta$ be as in Theorem \ref{T14} and define $g$ as in Lemma \ref{L20}. We choose $\eps_0$ and $\delta_0$ small enough so that $g\in \B_{\tilde \eps,\tilde \delta}$. Let $\tilde \lambda \in \left(2(\sigma+1), 2(\sigma+1) +N\right)$. Hence Theorem \ref{T12a} implies that
\begin{equation}
\label{45h}
|v(t,x) - v_*(x-e^{-2(\sigma+1)t} Ax)|\lesssim e^{-\tilde \lambda t}\quad\mbox{for all }x\in \R^N,t\ge0,
\end{equation}
for some symmetric and trace-free matrix $A\in\R^{N\times N}$. We claim that $A=0$. Indeed, we can find a trace-free matrix $M$ with $|M|\sim 1$ and
\[
|A|\sim \left|\int v_*(x)^{\sigma+1} x\cdot (MA)x\, dx\right|.
\]
Since $v_*$ satisfies \eqref{1e} by symmetry, we thus have
\begin{eqnarray*}
e^{-2(\sigma+1)t} |A| &\sim &\left|\int v_*(x)^{\sigma+1} (x+e^{-2(\sigma+1)t }Ax)\cdot M(x+e^{-2(\sigma+1)t}Ax)\, dx\right| +  O(e^{-4(\sigma+1)t}) \\
&=& \left| \int v_*(x-e^{-2(\sigma+1)t}Ax)^{\sigma+1}x\cdot Mx\, dx\right|+  O(e^{-4(\sigma+1)t}).
\end{eqnarray*}
Invoking \eqref{1e} for $v$, we further have
\[
e^{-2(\sigma+1)t} |A| \lesssim \| v(t)^{\sigma+1} - v_*(x-e^{-2(\sigma+1)t}Ax)^{\sigma+1}\|_{L^{\infty}} +  O(e^{-4(\sigma+1)t}).
\]
Then, because $\sigma\ge 0$ by assumption, an application of \eqref{45h} yields that $e^{-2(\sigma+1)t}|A|\lesssim e^{-\tilde \lambda t}$ for all $t\ge0$, and thus $\tilde \lambda > 2(\sigma+1)$ implies $A=0$. The remainder of the proof is very similar to the proofs of the previous theorems. We only remark  that
\[
v(t,x) - \frac12(1-|x|^2) - e^{-(2(\sigma+1)+N)t} c(1-\gamma |x|^2) = O(e^{-\lambda t})
\]
follows from 
Theorem \ref{T14}. We then compute
\[
\frac1{2R(t)^{\frac1{\alpha}}}\left(R(t)^2 - |x|^2\right)  = \frac12(1-|x|^2) +c e^{-(2(\sigma+1)+N)t} (1-\gamma|x|^2) +  O(e^{-\lambda t}),
\]
where we have used that $ (1-2\alpha)\gamma =1$. The statement of the theorem now can be deduced as before.
\end{proof}

\newpage

\section{The linear problem}\label{S:5}

In this section, we study the initial value problem for the degenerate parabolic equation
\begin{equation}
\label{7a}
\partial_t w - \rho^{-\sigma}\div\left(\rho^{\sigma+1}\grad w\right)=f
\end{equation}
in $B_1(0)$, where $\rho(z) = \frac12(1-|z|^2)$ and $\sigma = \frac{2-m}{m-1}\in(-1,\infty)$. A big effort is made to obtain suitable regularity estimates that serve as a basis for the well-posedness theory developed later in Section \ref{S:5} for the nonlinear perturbation equation \eqref{13}. Here, the choice of the topology is crucial. To reach the ultimate objective, the construction of invariant manifolds for the nonlinear equation in Section \ref{S7},  a differentiable dependence of solutions on the initial data is necessary. For nonlinearities of the form
\[
f = \rho F - \rho^{-\sigma} \div\left( \rho^{\sigma+1}zF\right)
\]
in which $F$ is a quadratic function of $\grad w$, we expect (and, in fact, will prove) that such a differentiable dependence can be established for solutions with small Lipschitz initial data. In order to obtain suitable regularization estimates, we need to work with space-time norms that deteriorate as $t\downarrow0$. We thus naturally arrive at Carleson measures.

%

The analysis of the present section is inspired by Koch and  Lamm's recent approach to tackle semilinear parabolic equations with rough coefficients \cite{KochLamm12} and, even more, by Kienzler's adaption of this approach to a subelliptic parabolic equation \cite{Kienzler14}, which closely resembles ours. The authors use standard tools from the theory of constant coefficient linear equations such as (Gaussian) decay estimates and Calderon--Zygmund-type estimates to establish bounds on the solutions in certain Carleson measures that reflect the regularity theory for the corresponding linear equation. The underlying idea goes back to Koch and  Tataru's work \cite{KochTataru01}, where a Carleson measure formulation of a {\it BMO} norm turned out to be the crucial ingredient in the study of regularity for the small datum Navier--Stokes equation. Many ingredients from \cite{Kienzler14} appeared earlier in Koch's habilitation thesis \cite{Koch99}, where Koch established a Schauder theory for H\"older continuous initial data.



The main difference between \eqref{7a} (and likewise the equations studied in \cite{Koch99,Kienzler14}) and the equations considered in \cite{KochLamm12} is the failure of strict parabolicity of \eqref{7a} at the boundary of the domain. However, we can easily overcome this problem by the following observation:
Equation \eqref{7a} can be interpreted as a heat flow on a certain so-called {\em weighted manifold}, that is, a Riemannian manifold to which a new volume element is assigned, typically a positive multiple of the one induced by the Riemannian metric. The theories of weighted manifolds and heat flows thereon can be developed parallel to the Riemannian counterparts, see, e.g., \cite{Grigoryan06}. In particular, Gaussian decay estimates on the heat kernel exist and a Calderon--Zygmund theory is available. We do thus expect that suitable Carleson measure estimates can be established if we carry over Koch and  Lamm's approach from the Euclidean to the Riemannian setting. In fact, we will still work on the Euclidean unit ball, but we will alter the metric accordingly. By working with the new metric, which is a  Carnot--Carath\'eodory distance, see, e.g., \cite{BellaicheRisler96}, we restore the parabolicity of the equation. In the context of the porous medium equation, the application of such Carnot--Carath\'eodory distances has been proved useful in the pioneering  works \cite{DaskalopoulosHamilton98,Koch99}, and recently in \cite{Kienzler14}. See also \cite{DenzlerKochMcCann15} and \cite{John15} for a similar perspective on the fast diffusion equation and thin film equation, respectively.

Our goal is the following result:

\begin{theorem}
\label{T6}
Let $p>\max\left\{N+2,\frac1{\sigma+1}\right\}$.
For every function $g\in C^{0,1}$ and $f\in Y(p)$, there exists a unique solution $w$ to \eqref{7a}. Moreover, $w\in X(p)$ and
\[
\|w\|_{X(p)}  + \|w\|_{\Lip}\lesssim \|f\|_{Y(p)} + \|g\|_{\Lip}.
\]
\end{theorem}

Here, we have used the notation $X(p) = \{ w: \|w\|_{X(p)}<\infty\}$ and $Y(p) = \{f:\: \|f\|_{Y(p)}<\infty\}$, where 
\begin{eqnarray*}
\|w\|_{ X(p)}&:=& \sup_{\substack{ z\in \overline{B_1(0)} \\0<r\le\sqrt2}} |Q_{r}^d(z)|^{-\frac1p}\left(r^2 \|\partial_t \grad w\|_{L^p(Q_{r}^d(z))} +  r(r+\sqrt{\rho(z)})   \|\grad^2 w\|_{L^p(Q_{r}^d(z))}\right.\\
&&\mbox{} \quad\quad\left.\phantom{\sqrt{(\rho)}}+  r^2 \|\rho\grad^3 w\|_{L^p(Q_{r}^d(z))}\right)\\
&& \mbox{}+\sup_{T\ge 1} \left(\|\partial_t\grad w\|_{L^p(Q(T))} + \|\grad^2 w\|_{L^p(Q(T))} +\|\rho \grad^3 w\|_{L^p(Q(T))}\right),\\
\|f\|_{Y(p)} &:=& 
 \sup_{\substack{ z\in \overline{B_1(0)} \\0<r\le\sqrt2}} |Q_{ r}^d(z)|^{-\frac1p}\left(\frac{r}{
 r+\sqrt{\rho(z)}}\|f\|_{L^p(Q_{ r}^d( z))} +  r^2 \|\grad f\|_{L^p(Q_{ r}^d( z))}\right)\\
&&\mbox{}+  \sup_{T\ge 1} \left(\|f\|_{L^p(Q(T))} + \|\grad  f\|_{L^p(Q(T))}  \right),
\end{eqnarray*}
and where $Q_r^d(z)$ and $Q(T)$ denotes the (intrinsic) time-space cylinders,
\[
Q_r^d(z)  = \left(\frac{r^2}2,r^2\right)\times B_r^d(z)\quad\mbox{and}\quad Q(T) = (T,T+1)\times B_1(0).
\]
Here, the intrinsic balls are relatively open subsets of the closed Euclidean unit ball, $B_r^d(z)= \left\{z'\in \overline{B_1(z)}: d(z,z')<r\right\}$ and $d$ is the {\em new} metric on $B_1(0)$ which is induced by the heat flow interpretation. In particular, $B_r^d(z) = \overline{B_1(0)}$ for every sufficiently large $r$. Finally, the Lipschitz norms are taken with respect to the {\em spatial} Euclidean topology, that is, $\|w\|_{\Lip} = \|\grad w\|_{L^{\infty}}$. 

By the linearity of the equation, we may and do split the problem into the homogeneous problem, where $f=0$, Proposition \ref{P2a}, and the problem with $g=0$, Proposition \ref{P3}. From these two cases, Theorem \ref{T6} follows by superposition.

Though an abstract theory for weighted manifolds is well developed, cf.\ \cite{Grigoryan06} and references therein, we will present a self-contained theory in the subsequent subsections since, with regard to the nonlinear equation \eqref{13} to which the regularity estimates of this section will be applied, results have to be tailored to our needs.

We conclude this introduction with a precise definition of weighted manifolds and the description of the heat flow interpretation of \eqref{7a}. Suppose $(\M, \g)$ is a Riemannian manifold and $\omega$ a positive function on $\M$. If we denote by $d\mathrm{vol}_{\g}$ the volume element on $(\M,\g)$ and set $d\mu = \omega\, d\mathrm{vol}_{\g}$, then the triple $(\M,\g,\mu)$ constitutes a weighted Riemannian manifold. If the Laplacian is defined via the Dirichlet form on $(\M,\g,\mu)$, its representation in local coordinates reads
\[
\laplace_{\mu} w = \frac1{\omega \sqrt{\det\g}}\sum_{i,j=1}^N \partial_i\left(\g^{ij}\omega\sqrt{\det\g} \, \partial_j w\right),
\]
where $\{\g^{ij}\}_{ij}$ is the inverse matrix of $\g$.

In the situation at hand, we simply choose the conformally flat Riemannian metric $\g = \rho^{-1}(dx)^2$ on the Euclidean unit ball $B_1(0)$, and define the weight $\omega = \rho^{\sigma+N/2}$. With these ingredients, \eqref{7a} becomes the heat equation
\[
\partial_t w - \laplace_{{\mu}}w = f.
\]
Notice that by setting $s = \arcsin(|x|)$, the metric transforms into
\[
\frac12 \g = \frac12 \rho^{-1}(dx)^2 = (ds)^2 + (\tan s)^2 (d\ell_{\S^{N-1}})^2,
\]
with $d\ell_{\S^{N-1}}$ being the length element on the unit sphere. This representation indicates that the manifold degenerates at its boundary where $\lim_{s\to \frac{\pi}2} \tan s =\infty$. Loosely speaking, $\g$ can be considered to be ``half way'' between the Euclidean metric $(dx)^2$ and the metric on the hyperbolic Poincar\'e disk $\rho^{-2}(dx)^2$. We expect a deeper understanding of $\M$ from the study of geodesic curves conducted in Subsection \ref{S:5.1} below.

Our program for this present section is the following: In Subsection \ref{S:5.1}, we compute the geodesic distance on $(\M,\g)$ as a function of coordinates on $B_1(0)$. In Subsection \ref{S:5.2}, we study the homogeneous heat equation and derive Gaussian estimates for the heat kernel. These estimates can be used to establish maximal regularity estimates for the heat semi-group. Finally, Subsection \ref{S:5.3} contains the maximal regularity estimates for the inhomogeneous equation.


\subsection{Intrinsic distance, balls and volumes}\label{S:5.1}
In the following, we study the geodesic distance on the (weighted) manifold introduced in the introduction of the present section. It is convenient to express this distance in terms of the conformally flat coordinates, that is, we study the distance induced by the metric $\g = \rho^{-1}(dx)^2$ on the Euclidean ball $B_1(0)$. In view of the interpretation of \eqref{7a} as a heat equation on $(\M,\g,\mu)$, this distance function on $B_1(0)$ can be thought of as the {\em intrinsic distance} for diffusions of the form \eqref{7a}, in the sense that its second power measures the typical time scale at which heat is exchanged between two points. 

The intrinsic distance for \eqref{7a} is (modulo a factor of $\sqrt2$) defined as the quantity
\[
\tilde d(z_1,z_2)  = \inf\left\{ L(\Gamma): \Gamma\mbox{ joins $z_1$ and $z_2$}\right\},
\]
where
\[
L(\Gamma) = \int_a^b \frac{|\Gamma'(\tau)|}{\sqrt{1-|\Gamma(\tau)|^2}}\,d\tau,
\]
if $\Gamma(a)=z_1$ and $\Gamma(b)=z_2$. It is a standard observation that minimizing $L$ is equivalent to minimizing
\[
E(\Gamma) = \int_a^b \frac{|\Gamma'(\tau)|^2}{1-|\Gamma(\tau)|^2}\,d\tau,
\]
because $L(\Gamma)^2\le (b-a) E(\Gamma)$ by H\"older's inequality with equality precisely if $\Gamma$ is geodesic, that is $\frac{|\Gamma'|^2}{1-|\Gamma|^2}=\const$. The functional $E$ is strictly convex and admits hence a unique minimizer. Parametrizing the minimizer $\Gamma $ by arclength, the geodesic equation reads
\begin{equation}\label{5}
|\Gamma'|^2=1-|\Gamma|^2,
\end{equation}
and the Euler--Lagrange equations are
\begin{equation}\label{6}
\left(\frac{\Gamma'}{1-|\Gamma|^2}\right)' = \frac{\Gamma}{1-|\Gamma|^2}.
\end{equation}
It is not difficult to see that geodesic curves through $z_1$ and $z_2$ are confined to the two-dimensional plane that is spanned by the points $z_1$, $z_2$ and $0$ in $\R^N$. Upon a rotation, we may thus write $\Gamma = (x,y,0,\dots,0)^T\in \R\times\R\times \R^{N-2}$, so that \eqref{5} and \eqref{6} translate into
\[
(x')^2 + (y')^2 = 1- x^2 -y^2,
\]
and
\[
\left(\frac{x'}{1-x^2-y^2}\right)' = \frac{x}{1-x^2-y^2}\quad\mbox{and}\quad \left(\frac{y'}{1-x^2-y^2}\right)' = \frac{y}{1-x^2-y^2}.
\]
These equations are solved by hypocycloids, i.e., traces of fixed points on small circles of radius $r\le 1/2$ that roll along the interior boundary of the unit ball:
\begin{eqnarray*}
x(t ) &=& (1-r) \cos\left(\sqrt{\frac{r}{1-r}} t\right) + r\cos\left(\sqrt{\frac{1-r}r} t\right),\\
y(t ) &=& (1-r) \sin\left(\sqrt{\frac{r}{1-r}} t\right) - r\sin\left(\sqrt{\frac{1-r}r} t\right).
\end{eqnarray*}
Here we have rotated and, if necessary, flipped over the two-dimensional disk in such a way that the geodesics hit the boundary at $(1,0)$ and $(\cos(2\pi r),\sin(2\pi r))$. 

In spite of the good understanding of geodesic curves, except from a few particular cases, it is difficult to explicitly calculate the intrinsic distance between two points on $\M$. On the positive side, the geodesic distance $\tilde d$ on $(\M,\g)$ is equivalent to the semimetric
\[
d(z_1,z_2) = \frac{|z_1-z_2|}{\sqrt{\rho(z_1)} + \sqrt{\rho(z_2)} + \sqrt{|z_1-z_2|}}
\]
defined on $\overline{B_1(0)}$:

\begin{prop}\label{P1}
It holds
\[
\tilde d(z_1,z_2) \sim d(z_1,z_2)
\]
for all $z_1,z_2\in \overline{B_1(0)}$.
\end{prop}

In the remainder of the present section, we will mostly choose $d$ over $\tilde d$. It is thus convenient to speak of $d$ in the following as a distance (or even the {\em intrinsic} distance), even though it lacks a proper triangle inequality.


In our proof of Proposition \ref{P1}, we follow \cite[Chapter 4.3]{Koch99}.

\begin{proof} We start with a study of two special cases.  For two points on a straight line through the origin, we chose $r=1/2$, so that $x(t) = \cos(t)$ and $y(t)=0$. For convenience, we suppose that $z_1=\alpha z_2$ for some $\alpha \in [0,1]$.  Then
\begin{equation}\label{6a}
\tilde d(z_1,z_2)  = \left|\arccos(|z_1|) - \arccos(|z_2|)\right| \sim |\sqrt{\rho(z_1)} - \sqrt{\rho(z_2)}|.
\end{equation}
It remains to observe that the latter expression is equivalent to $d(z_1,z_2)$ if $z_1=\alpha z_2$.

Now consider two points on the boundary, $|z_1|=|z_2|=1$, say $z_1= (1,0)$ and $z_2 = (\cos(2\pi r),\sin (2\pi r))$. If $R(t)$ denotes the distance from the origin parametrized by arc length (cf.\ \eqref{5}), then
\[
R(t)^2 = 1-4r(1-r)\sin^2\left(\frac{t}{2\sqrt{r(1-r)}}\right),
\]
and thus $\tilde d(z_1,z_2) = 2\pi\sqrt{r(1-r)}$. On the other hand, a direct calculation yields $|z_1-z_2| = 2\sin (\pi r)$. Solving both identities for $r$ then gives \begin{equation}\label{6b}
\tilde d(z_1,z_2)\sim \sqrt{|z_1-z_2|}.
\end{equation}

In the general case, we fix two non-parallel, non-zero $z_1,z_2\in\overline{ B_1(0)}$. We start with proving the upper bound
\begin{equation}
\label{7}
\tilde d(z_1,z_2)\lesssim d(z_1,z_2).
\end{equation}
We distinguish two cases. First, if $\sqrt{|z_1-z_2|}\le \sqrt{\rho(z_1)} + \sqrt{\rho(z_2)}$, we have
\[
\left|\sqrt{\rho(z_1)} - \sqrt{\rho(z_2)}\right| \le 2 \frac{\left|\rho(z_1) - \rho(z_2)\right| }{\sqrt{\rho(z_1)|} + \sqrt{\rho(z_2)} + \sqrt{|z_1-z_2|}} \lesssim d(z_1,z_2),
\]
where we have used the definition of $\rho$ in the second inequality. We do note loose generality by assuming that $\rho(z_1)\le \rho(z_2)$. Hence, by triangle inequality
\[
\tilde d(z_1,z_2)\le \tilde d\left( z_1,\frac{|z_2|}{|z_1|}z_1\right) + \tilde d\left(\frac{|z_2|}{|z_1|}z_1,z_2\right),
\]
that is, we replace the geodesics by first moving from $z_1$ along a straight line in direction of the origin until we reach the sphere of radius $|z_2|$. From there, we move along the hypocycloid towards $z_2$. By \eqref{6a} and the previous estimate, the length of the first piece of this curve is controlled by $d(z_1,z_2)$. The length of the curve on the hypocycloid is cruelly estimated against the Euclidean distance, that is
\[
\tilde d\left(\frac{|z_2|}{|z_1|}z_1,z_2\right)\lesssim \frac{\left|\frac{|z_2|}{|z_1|}z_1 - z_2\right|}{\sqrt{\rho(z_2)}}\le \frac{|z_1-z_2|}{\sqrt{\rho(z_2)}}.
\]
The latter is estimated by $d(z_1,z_2)$ because of the assumptions $\rho(z_1)\le\rho(z_2)$ and $\sqrt{|z_1-z_2|}\le \sqrt{\rho(z_1)} + \sqrt{\rho(z_2)}$.

In the second case, if $\sqrt{\rho(z_1)} +\sqrt{\rho(z_2)} \le \sqrt{|z_1-z_2|}$, we use the triangle inequality
\[
\tilde d(z_1,z_2)\le \tilde d\left(z_1,\frac{z_1}{|z_1|}\right) + \tilde d\left(\frac{z_1}{|z_1|},\frac{z_2}{|z_2|}\right) + \tilde d\left(\frac{z_2}{|z_2|},z_2\right),
\]
that means, this time our competitor curve first runs straight to the boundary and then along the hypocycloid to connect the projection points $\frac{z_1}{|z_1|}$ and $\frac{z_2}{|z_2|}$. Hence, applying \eqref{6a} and \eqref{6b}, we estimate $\tilde d(z_1,z_2)\lesssim \sqrt{\rho(z_1)} +\sqrt{\rho(z_2)} +\sqrt{|z_1-z_2|}$. The latter term is controlled by $d(z_1,z_2)$ thanks to the above assumption. This concludes the proof of \eqref{7}.

We finally turn to the proof of the opposite estimate
\begin{equation}
\label{8}
d(z_1,z_2)\lesssim \tilde d(z_1,z_2).
\end{equation}
We suppose that $\rho(z_1)\le \rho(z_2)$. Our first goal is the lower bound
\begin{equation}
\label{8a}
\tilde d(z_1,z_2)\gtrsim   \min_{\psi\ge \sqrt{\rho(z_2)}} \left\{ \psi- \sqrt{\rho(z_1)} + \frac{|z_1-z_2|}{\psi}\right\} .
\end{equation}
Indeed, on the one hand, we have the trivial estimate
\[
\tilde d(z_1,z_2)\gtrsim \frac{|z_1-z_2|}{ \psi_*}
\]
where $ \psi_* = \sup_t \sqrt{\rho(\Gamma(t))}$. On the other hand, via the fundamental theorem and because $\Gamma(t)\in B_1(0)$,
\[
\sup_t \sqrt{1- |\Gamma(t)|^2} \le \sqrt{1 - |z_1|^2} + \int_a^b\frac{|\Gamma'(t)|}{\sqrt{1-|\Gamma(t)|^2}}\, dt,
\]
and thus $\psi_*  - \sqrt{\rho(z_1)}\lesssim\tilde d(z_1,z_2)$. Since $\psi_*\ge \sqrt{\rho(z_2)}$, we have thus proved \eqref{8a}. To deduce \eqref{8}, we again distinguish to cases. If the optimal $\psi$ satisfies the estimate $\psi\ge \sqrt{\rho(z_1)} +\sqrt{\rho(z_2)}+\sqrt{|z_1-z_2|}$, then
\[
d(z_1.z_2)\le \sqrt{|z_1-z_2|} \le \psi  - \sqrt{\rho(z_1)} - \sqrt{\rho(z_2)}\stackrel{\eqref{8a}}{\lesssim}\tilde d(z_1,z_2).
\]
If not, we conclude
\[
d(z_1,z_2) \le \frac{|z_1-z_2|}{\psi}\stackrel{\eqref{8a}}{\lesssim }\tilde d(z_1,z_2).
\]
This proves the lower bound \eqref{8}.
\end{proof}

The knowledge of geodesics and of the intrinsic distance improves our understanding of the Riemannian manifold $\M$. The boundary is at a finite distance from the interior of the ball. The shortest distance is realized by straight lines and is exactly $d(z,\partial B) = \arccos(|z|)\sim \sqrt{\rho(z)}$. By contrast, the length of curves parallel to the boundary diverges. For instance, the length of a circle of radius $R$ is $\frac{2\pi R}{\sqrt{1-R^2}}$.

The study of the intrinsic distance function reveals the diffusion time scales on $\M$.  On a general manifold, diffusion over a distance $d$ happens in a time $t$ of order $t\sim d^2$. In our situation, geodesics hit the boundary orthogonally. Hence, diffusion near the boundary happens primarily in normal direction and distances scale $d\sim \sqrt{d_{\mathit{Eucl}}}$ with the Euclidean distance $d_{\mathit{Eucl}}$. The diffusion time scale close to the boundary is thus $t\sim d_{\mathit{Eucl}}$. On the other hand, in the interior where $d\sim d_{\mathit{Eucl}}$, the diffusion time is as in the Euclidean setting, $t\sim d_{\mathit{Eucl}}^2$. The role of different diffusion time scales with regard to the actual position relative to the boundary will be reflected in our Calderon--Zygmund theory developed in Subsections \ref{S:5.2} and \ref{S:5.3} below. Notice that close to the boundary, the diffusion part $-\rho\laplace$ and the drift part $(\sigma+1) z\cdot \grad$ of our linear operator are comparable. Here, the fact that $\sigma +1>0$ turns out to be of importance: Mass is transported towards the boundary and thus boundary conditions are not required. On the contrary, we will see that solutions satisfy some natural boundary conditions, which were described as ``asymptotic boundary conditions'' in \cite{Seis14}. See also Remarks \ref{R1} and \ref{R1a} in Subsection \ref{S:5.1a}  below.

The relevance of intrinsic distances for the regularity theory of the (linearized) porous medium equation was first exploited by Daskalopulos and  Hamilton \cite{DaskalopoulosHamilton98} and Koch \cite{Koch99}.

We conclude this subsection with a study of intrinsic balls and their volumes. An open intrinsic ball of radius $r$ around $z$ is defined by
\[
B_r^d(z):= \left\{ y\in \overline{B_1(0)}:\: d(z,y)< r\right\}.
\]
The first statement of Lemma \ref{L8} below shows that such a ball is comparable to the Euclidean ball with the same center $z$ but radius $r\left(r +\sqrt{\rho(z)}\right)$.

\begin{lemma}\label{L8}
There exists a constant $C>0$ such that for all $r>0$ and $z\in \overline{B_1(0)}$, it holds that
\[
B_{r\left(r+\sqrt{\rho(z)}\right)}(z)\cap \overline{B_1(0)} \subset B_r^d(z) \subset B_{C r\left(r+\sqrt{\rho(z)}\right)}(z)\cap \overline{B_1(0)}.
\]
\end{lemma}

\begin{proof} It is enough to study the case $r\le 1$. We fix $z, z' \in \overline{B_1(0)}$. Suppose that $d(z,z')<r$. We will first establish the second inclusion by showing that
\begin{equation}
\label{9a}
|z-z'|\lesssim r\left(r+\sqrt{\rho(z)}\right).
\end{equation}
From the definition of $d$ we notice that
\[
|z-z'| < r \left( \sqrt{\rho(z)} + \sqrt{\rho(z')} + \sqrt{|z-z'|}\right),
\]
which with the help of Young's inequality turns into
\[
|z-z'| \lesssim r\left(\sqrt{\rho(z)} + \sqrt{\rho(z')}\right) + r^2.
\]
If $|z'|\ge |z|$, then the claim follows. Otherwise, if $|z'| < |z|$, we set $\tilde z : =\frac{|z'|}{|z|} z$. We recall that the shortest curve connecting $z$ and $\tilde z$ lies on a straight line and this line segment realizes the minimal distance between $z$ and the ball of radius $|\tilde z|$ around the origin. Hence,  using Proposition \ref{P1}, $\tilde d(z,\tilde z)\le \tilde d(z,z')\lesssim r$. On the other hand, because the geodesic curve is given by $\Gamma(t) = \frac{z}{|z|}\sin(t) $ where $t\in [\arcsin(|z'|),\arcsin(|z|)]$, it holds 
\[
\tilde d(z,\tilde z) = \arcsin(|z|) - \arcsin(|z'|)  = \int_{|z|}^{|z'|} \frac1{\sqrt{1-\tau^2}}\,d\tau \sim \sqrt{\rho(z')} - \sqrt{\rho(z)}.
\]
We thus deduce that $\sqrt{\rho(z')} \lesssim \sqrt{\rho(z)} + r$, and therefore \eqref{9a} follows.

To prove the first inclusion, we assume that
$|z-z'|< r\left(r+\sqrt{\rho(z)}\right)$. Then, by the definition of $d$ and monotonicity
\[
d(z,z')< \frac{r\left( r+\sqrt{\rho(z)}\right)}{\sqrt{\rho(z)} + \sqrt{\rho(z')} + \sqrt{r\left(r+\sqrt{\rho(z)}\right)}}\le r.
\]
This concludes the proof of Lemma \ref{L8}.
\end{proof}

Towards the end of this subsection, we will gather some technical results that prove to be helpful in the subsequent analysis. The first result provides two comparison formulas for balls whose centers are either relatively close or relatively far away from the boundary.

\begin{lemma}\label{L8c} There exists $C>0$ such that for all $z\in \overline{B_1(0)}$ and $r>0$ the following holds:
\begin{enumerate}
\item If $z\not=0$ and $\sqrt{\rho(z)}\le r$, then
\[
B_{r^2 } \left(\frac{z}{|z|}\right)\cap\overline{B_1(0)}\subset B_{Cr}^d(z)\quad\mbox{and}\quad B_r^d(z)\subset B_{Cr^2}\left(\frac{z}{|z|}\right)\cap\overline{B_1(0)},
\]
and $\sqrt{\rho(z')}\lesssim r$ for all $z'\in B^d_r(z)$.
\item If $\sqrt{\rho(z)}\ge r$, then
\[
B_{r\sqrt{\rho(z)}}(z)\cap\overline{B_1(0)} \subset B_{r}^d(z)\quad\mbox{and}\quad B_r^d(z)\subset B_{Cr\sqrt{\rho(z)}}(z)\cap\overline{B_1(0)},
\]
and $\rho(z')\lesssim \rho(z)$ for all $z'\in B_r^d(z)$. Moreover, if $\sqrt{\rho(z)} \ge 2Cr$, then also $\rho(z')\gtrsim \rho(z)$ for all $z'\in B_r^d(z)$.
\end{enumerate}
\end{lemma}

\begin{proof}
The inclusions in the first statement are immediate consequences of the triangle inequality for $d$, the fact that $d(z,z/|z|)\sim \sqrt{\rho(z)}$, and Lemma \ref{L8}. For instance, for any $z'\in B_r^d(z)$, it holds
\[
d\left(z',\frac{z}{|z|}\right) \lesssim d(z',z) + \tilde d\left(z,\frac{z}{|z|}\right)\lesssim r+ \sqrt{\rho(z)}\lesssim r,
\]
which implies the second inclusion via Lemma \ref{L8}. The first inclusion can be established analogously. Moreover, $\sqrt{\rho}\lesssim r$ in $B_r^d(z)$ follows from the second inclusion.

The second statement is an application of Lemma \ref{L8} and the assumption. Indeed, the comparison formula for the balls follows from Lemma \ref{L8} by further estimating the radii. Then, by triangle inequality we have for all $z'\in B_r^d(z)\subset B_{Cr\sqrt{\rho(z)}}(z)$ that $|z|\le C \rho(z) + |z'|$ and thus $\rho(z') \lesssim \rho(z)$. Moreover, using $|z'|\le Cr\sqrt{\rho(z)} +|z|$, we obtain
\[
1- |z'| \ge 1- |z| - Cr\sqrt{\rho(z)} \ge \frac12\rho(z) +\frac12\sqrt{\rho(z)}\left(\sqrt{\rho(z)} - 2
Cr\right).
\]
For $\sqrt{\rho(z)} \ge 2Cr$ this implies $\rho(z')\gtrsim \rho(z)$ as desired.
\end{proof}

The next lemma provides estimates on the intrinsic volumes of balls, where ``intrinsic'' refers to the volume form inherited from the weighted manifold discussed in the introduction to this section. To fix a notation reminiscent of $\sigma$, we set $\mu_{\sigma} : = \rho^{\sigma}\,\L^N\mres B_1(0) \left( = \omega\,  \mathrm{vol}_{\g}\right)$, where $\L^N$ denotes the Lebesgue measure, and for any measurable $A\subset \R^N$, we write $|A|_{\sigma} = \mu_{\sigma}(A)$. Notice that for any $\sigma>-1$, $\mu_{\sigma}$ defines an absolutely continuous finite Radon measures, which has precisely the same null sets as $\L^N\mres B_1(0)$.

\begin{lemma}\label{L8b}
For any $r\lesssim 1$ and $z\in \overline{B_1(0)}$, it holds
\[
|B_r^d(z)|_{\sigma} \sim r^N(r+\sqrt{\rho(z)})^{N+2\sigma}.
\]
%
%
\end{lemma}

\begin{proof}
The statement is an immediate consequence of Lemma \ref{L8}.
\end{proof}

We finally have:

\begin{lemma}\label{L8d}
Let $r>0$ and $z,z'\in \overline{B_1(0)}$. Then it holds
\[
\frac{|B_r^d(z)|_{\sigma}}{|B_r^d(z')|_{\sigma}}\lesssim \left(1 +\frac{d(z,z')}r\right)^{\max\{N,2N+2\sigma\}}.
\]
In particular,
\[
\frac{r+\sqrt{\rho(z)}}{r+\sqrt{\rho(z')}} \lesssim \left(1+\frac{d(z,z')}r\right)^{\frac{\max\{N,2N+2\sigma\}}{|N+2\sigma|}}.
\]
\end{lemma}

\begin{proof}
The results are trivial if $r\gtrsim 1$. Otherwise, if $r\ll1$, we invoke the triangle inequality to verify the inclusion $B_r^d(z)\subset B_{C\left(r + d(z,z')\right)}^d(z')$. An application of  Lemma \ref{L8b} then yields the desired estimates.
\end{proof}

\subsection{Existence, uniqueness and energy estimates}\label{S:5.1a}

The objective of this subsection is to give a precise meaning of a solution to \eqref{7a} and the corresponding initial value problem, state existence and uniqueness results and derive first  regularity estimates in the natural Hilbert space setting. 

We start with some considerations from a functional analysis point of view. The (sub-)elliptic operator occurring in \eqref{7a}, $\L:= -\rho\laplace +(\sigma+1)z\cdot \grad$, is nonnegative symmetric with respect to the $L^2_{\sigma} := L^2(\mu_{\sigma})$ inner product, because
\begin{equation}\label{60}
\int \varphi \L\psi\, d\mu_{\sigma} = \int \grad \varphi \cdot \grad \psi\, d\mu_{\sigma+1}\quad\mbox{for all }\varphi,\psi\in C^{\infty}(\overline{B_1(0)}).
\end{equation}
As the space of test functions, $C^{\infty}(\overline{B_1(0)})$, is densely contained in the weighted Sobolev spaces $H^k_{\sigma,\dots,\sigma+k}$, where
\[
H^0_{\sigma} = L^2_{\sigma},\quad H^k_{\sigma,\dots,\sigma+k} = L^2_{\sigma}\cap \dot H^1_{\sigma+1}\cap\dots\cap \dot H^k_{\sigma+k},
\]
and $\dot H^{\ell}_{\sigma+\ell} = \dot H^{\ell}(\mu_{\sigma+\ell})$, cf.\ \cite[Lemma 2]{Seis14}, $\L$ extends to a self-adjoint operator on $L^2_{\sigma}$ (Friedrich's extension). The following notion of weak solutions seems thus to be natural:

\begin{definition}\label{D1}
Let $0<T\le\infty$ and $f\in L^1((0,T); L^2_{\sigma})$ and $g\in L_{\sigma}^2$.
We call $w$ a {\em weak solution} to \eqref{7a}  with initial datum $g$, if
\begin{equation}
\label{61}
-\int_0^T \int w \partial_t\zeta\, d\mu_{\sigma} dt + \int_0^T\int \grad w\cdot \grad \zeta\, d\mu_{\sigma +1} dt
=\int g\zeta\, d\mu_{\sigma} +  \int_0^T\int f\zeta \, d\mu_{\sigma} dt
\end{equation}
for all $\zeta\in C^{\infty}([0,T)\times \overline{B_1(0)})$ with $\spt \zeta \subset [0,T)\times \overline{B_1(0)}$.
\end{definition}

Here and in the following, we use the convention that if domains of integration in the spatial variables are not specified, then we integrate over $B_1(0)$.

\begin{remark}[Asymptotic boundary condition]\label{R1}
Notice that the definition of weak solutions (and the one of $\L$ above) contains some (weak) information about the behavior of solutions near the boundary. Indeed, if $w$ is a smooth solution, then necessarily
\[
\lim_{|z|\uparrow1} \rho(z)^{\sigma+1}z\cdot \grad w(z) = 0.
\]
This condition rules out functions that grow rapidly at the boundary. It is proved in the course of the present section (and announced in Theorem \ref{T6} above), that for sufficiently regular data, solutions are in fact $C^1$ up to the boundary. Such solutions trivially satisfy the asymptotic boundary conditions above. In fact, these boundary conditions are also the natural boundary conditions in the sense of calculus of variations as \eqref{60} can be derived as the Euler--Lagrange equations for the (weighted) Dirichlet energy, cf.\ \cite[Lemma 5]{Seis14}. See also Remark \ref{R1a}.\end{remark}

Our first result concerns the well-posedness of the initial value problem in the Hilbert space setting.

\begin{lemma} \label{L6}
Let $0<T<\infty$ and $f\in L^1((0,T); L_{\sigma}^2 )$ and $g\in L_{\sigma}^2$. There exists a unique weak solution $w$ to \eqref{7a} with initial datum $g$. Moreover, $w\in C([0,T]; L_{\sigma}^2 )$ with $w(0)=g$ and the energy identity
\[
\frac12 \|w(T)\|^2_{L_{\sigma}^2} +  \int_0^{T} \|\grad w\|_{L_{\sigma+1}^2}^2\, dt = \frac12 \|g\|_{L_{\sigma}^2}^2  + \int_0^{T} \int fw\, d\mu_{\sigma}dt
\]
holds.
\end{lemma}

\begin{proof}
The proof of such a result follows a standard procedure. For instance, existence can be obtained via an implicit time-discretization and the elliptic theory derived in the appendix of \cite{Seis14}. The energy identity follows from choosing $w$ as a test function in the weak formulation, which is admissible by the density of smooth functions in $H^1_{\sigma,\sigma+1}$. Uniqueness is a consequence of the energy estimate and the linearity of the equation.
\end{proof}

\begin{remark}[No boundary conditions]\label{R1a}
It is remarkable that the weak notion of solutions in Definition \ref{D1} guarantees uniqueness {\em without} specifying the boundary condition on $\partial B_1(0)$. This observation goes far beyond Remark \ref{R1}: In fact, the parabolic equation \eqref{7a} is well-posed in $L^2_{\sigma}$ only if boundary conditions are {\em not} prescribed. Roughly speaking, this phenomenon stems from the fact that close to the boundary, $\rho\laplace$ and $z\cdot \grad$ are both of the same order and the latter term drives mass {\em towards} the boundary. It seems that phenomena of this type have been systemically studied only very recently --- at least in the unweighted setting, see \cite{FornaroMetafunePallaraPruess07,
FornaroMetafunePallaraSchnaubelt15}.
\end{remark}

We prove maximal regularity in $L^2_{\sigma}$ for the equation with zero initial data.

\begin{lemma}\label{L7}
Let $0<T<\infty$ and $f\in L^2((0,T);L^2_{\sigma}(B))$. Suppose that $w$ is the weak solution to \eqref{7a} with initial value $g\equiv 0$. Then the mapping $t\mapsto \|\grad w(t)\|_{L_{\sigma+1}^2}$ is continuous in $[0,T]$ and $\grad w(0)=0$. Moreover,
\[
\int_0^T \|\partial_t w \|_{L^2_{\sigma} }^2 \, dt + \int_0^T \|\grad w\|^2_{L_{\sigma}^2}\, dt + \int_0^{T} \|\grad^2w\|_{L_{\sigma+2}^2}^2\, dt \lesssim \int_0^{T} \|f\|^2_{L^2_{\sigma}}\, dt.
\]
\end{lemma}

\begin{proof}
Let $0\le t_1<t_2\le T$ be arbitrarily given. Using a regularized version of $\zeta = \chi_{[t_1,t_2]} \partial_t w$ as test function in \eqref{61}, we infer the identity
\[
\int_{t_1}^{t_2} \int(\partial_t w)^2\, d\mu_{\sigma} dt + \frac12 \|\grad w(t_2)\|_{L_{\sigma +1}^2}^2  = \frac12\|\grad w(t_1)\|_{L_{\sigma +1}^2}^2 + \int_{t_1}^{t_2} \int f\partial_t w\, d\mu_{\sigma}dt,
\]
from which we deduce that $\|\partial_t w\|_{L^2(L^2_{\sigma})}\lesssim \|f\|_{L^2(L^2_{\sigma})}$, because $g=0$ and $\grad w(0)=0$. Moreover, further examination of the above identity reveals that $[0,T)\ni t\mapsto \|\grad w(t)\|_{L_{\sigma+1}^2}\in \R$ is a continuous mapping. 
We deduce hence that it is enough to study the elliptic problem
\begin{equation}\label{21a}
-\rho^{-\sigma} \div\left(\rho^{\sigma+1}\grad w\right) = \tilde f:= f- \partial_t w\in L^2((0,T); L_{\sigma}^2)
\end{equation}
pointwise in time, that is, we suppose that $\tilde f\in L^2_{\sigma}$. Then the statement of the lemma is a consequence of
 the maximal regularity estimate
\begin{equation}\label{18}
\| \grad^2 w\|_{L_{\sigma+2}^2}\lesssim \|\tilde f\|_{L_{\sigma}^2}
\end{equation}
and
\begin{equation}\label{20}
\|\grad w\|_{L_{\sigma}^2}\lesssim \|\tilde f\|_{L_{\sigma}^2}.
\end{equation}
Our arguments for \eqref{18} and \eqref{20} will be quite formal and can be made rigorous using standard approximation techniques and the density of smooth functions in $H^2_{\sigma,\sigma+1,\sigma+2}$. We start by proving the auxiliary estimate
\begin{equation}\label{19}
\|\grad w\|_{L_{\sigma+1}^2}\lesssim \|\tilde f\|_{L_{\sigma}^2}.
\end{equation}
For this purpose, we test \eqref{21a} with $\rho^{\sigma} w$ and obtain using the Cauchy--Schwarz inequality and $\int \tilde f\, d\mu_{\sigma}=0$ (notice that spatially constant functions are admissible testfunctions) that
\[
\int |\grad w|^2 \, d\mu_{\sigma+1}\le \left(\int \tilde f^2 \, d\mu_{\sigma} \int (w-c)^2 \, d\mu_{\sigma}\right)^{1/2},
\]
where $c$ is an arbitrary constant. Combined with the Hardy--Poincar\'e inequality
\[
\inf_{c\in \R} \int (w-c)^2\, d\mu_{\sigma} \lesssim \int |\grad w|^2\, d\mu_{\sigma+1},
\]
which is proved in \cite[Lemma 3]{Seis14}, this implies \eqref{19}.

For \eqref{20}, we first let $\eta$ be a smooth cut-off function that is supported outside $B_{\frac14}(0)$ such that $\eta=1$ outside $B_{\frac12}(0)$. Testing \eqref{21a} with $\rho^{\sigma} \eta\frac{z}{|z|^2}\cdot \grad w$ then yields after multiple integration by parts
\[
\frac{\sigma+1}2 \int \eta |\grad w|^2 \, d\mu_{\sigma} \le \int \eta \tilde f \frac{z}{|z|^2} \cdot \grad w \, d\mu_{\sigma} + C\int |\grad w|^2 \, d\mu_{\sigma+1},
\]
and via Young's inequality
\[
\int_{B_1(0)\setminus B_{\frac12}(0)} |\grad w|^2 \, d\mu_{\sigma} \lesssim \int \tilde f^2 \, d\mu_{\sigma} + \int |\grad w|^2\, d\mu_{\sigma+1}\stackrel{\eqref{19}}{\lesssim} \int \tilde f^2 \, d\mu_{\sigma}.
\]
On the other hand, because $\rho\sim 1$ in $B_{\frac12}(0)$, we get from \eqref{19} that
\[
\int_{B_\frac12(0)} |\grad w|^2\, d\mu_{\sigma}\lesssim \int \tilde f^2\, d\mu_{\sigma}.
\]
The last two estimates together yield the desired bound \eqref{20}.

We finally turn to the (formal) proof of \eqref{18}, which is based on a suitable choice of testfunction and a multiple integration by parts. Indeed, testing with $ -\rho^{\sigma +1} \laplace w$, we derive
\begin{eqnarray*}
\lefteqn{\int |\grad^2 w|^2\, d\mu_{\sigma+2}}\\
&=&-\int \tilde f\laplace w\, d\mu_{\sigma+1}  + (\sigma +2)\int z\cdot (\grad^2 w \grad w)\, d\mu_{\sigma +1} - \int \laplace w z\cdot \grad w\, d\mu_{\sigma+1},
\end{eqnarray*}
from which we deduce \eqref{18} with the help of  the Cauchy--Schwarz and Young inequalities and \eqref{20}.
\end{proof}

We conclude this subsection with a symmetry observation for the homogeneous problem, that is $f=0$. 

\begin{lemma}\label{R2}
If $w_1$ and $w_2$ are two solutions to the linear homogeneous equation (that is, \eqref{7a} with $f\equiv0$) with initial data $g_1$ and $g_2$, respectively, then
\[
\int w_1(t)g_2\, d\mu_{\sigma} = \int g_1 w_2(t)\, d\mu_{\sigma}
\]
for all $t\ge 0$.
\end{lemma}

This lemma shows in particular that the semigroup operator $e^{-t\L}$ is self-adjoint on $L^2_{\sigma}$.

\begin{proof}
We define the time inversion operator $\cS:[0,t]\to[0,t]$ by $\cS(\tau) = t- \tau $. Then
\begin{eqnarray*}
\lefteqn{\int w_1(t)g_2\, d\mu_{\sigma} - \int g_1 w_2(t)\, d\mu_{\sigma}}\\
&=& \int w_1(t) (w_2\circ \cS) (t) \, d\mu_{\sigma} - \int w_1(0) (w_2\circ \cS) (0)\, d\mu_{\sigma}\\
&=& \int_0^{t}  \int \left( \partial_t w_1 (w_2\circ \cS) - w_1 (\partial_t  w_2)\circ \cS\right)\, d\mu_{\sigma}dt.
\end{eqnarray*}
Using the equations for $w_1$ and $w_2$ and integrating by parts yields the desired result.
\end{proof}

\subsection{Carleson measure estimates for the homogeneous problem} \label{S:5.2}

In this subsection, we derive regularity estimates for the homogeneous equation
\begin{eqnarray}
\partial_t w - \rho^{-\sigma} \div\left(\rho^{\sigma+1}\grad w\right) &=& 0,\label{17a}\\
w(0,\tacka)&=& g.\label{17b}
\end{eqnarray}
We suppose that the initial datum is a Lipschitz function with respect to the Euclidean topology, that means $\|g\|_{\Lip} = \|\grad g\|_{L^{\infty}}<\infty$. 

In a first step, we derive local $L^2_{\sigma}$ regularity estimates.

\begin{lemma}\label{L8a}
Let $0< \eps<\tilde\eps<1$ and $0<\tilde \delta <\delta <1$. Let $w$ be a solution to the homogeneous equation \eqref{17a}.
If $ \delta $ and $\tilde \delta /\delta$ are sufficiently small, then the following holds for any $k\in \N_0$, $\beta\in \N_0^N$ and $0<r\le\sqrt2$:
\begin{enumerate}
\item If $z_0\in \overline{ B_1(0)}$ is such that $\sqrt{\rho(z_0)}\le \delta r$ and $\tau\ge 0$, then
\[
\int_{\tau + \tilde \eps r^2}^{\tau + r^2} \|\partial_t^k \partial_z^{\beta} w\|^2_{L_{\sigma}^2(B_{\tilde \delta r}^d(z_0))}\, dt \lesssim r^{-4|\beta| - 4k} \int_{\tau+\eps r^2}^{\tau + r^2} \|w\|^2_{L_{\sigma}^2(B_{\delta r}^d(z_0))}\, dt.
\]
\item  If $z_0\in \overline{ B_1(0)}$ is such that $\sqrt{\rho(z_0)}\ge \delta r$ and $\tau \ge 0$, then
\[
\int_{\tau + \tilde \eps r^2}^{\tau + r^2} \|\partial_t^k \partial_z^{\beta} w\|^2_{L_{\sigma}^2(B_{\tilde \delta r}^d(z_0))}\, dt \lesssim   r^{-4k-2|\beta|}\rho(z_0)^{-|\beta|}  \int_{\tau+\eps r^2}^{\tau + r^2} \|w\|^2_{L_{\sigma}^2(B_{\delta r}^d(z_0))}\, dt.
\]
\end{enumerate}
\end{lemma}

Of course, the constants in the inequalities (which are suppressed in our notation) depend on the particular values of $k,\, \beta,\, \eps,\, \tilde \eps,\, \delta$, and $ \tilde \delta$. For our subsequent arguments, only the dependence on $r$, $\tau$, and $z_0$ is of relevance. We will thus neglect the dependency on the other variables for notational convenience.

\begin{proof}Upon a shift in the time variable, we may without loss of generality assume that $\tau=0$. 

We first consider the case $\sqrt{\rho(z_0)}\le \delta r$.
In view of Lemma \ref{L8c}, it suffices to prove the statement for $z_0\in \partial B_1(0)$, provided that $\tilde \delta/ \delta$ is small enough. In the first statement considered here, we may thus replace the intrinsic balls by Euclidean balls with squared radii.

We choose $\eps <\hat \eps<\tilde \eps$ and $\tilde \delta <\hat \delta <\delta$ and let $\eta$ be a smooth cut-off function that is supported in $( \eps r^2,r^2]\times \overline{B_{ \delta^2 r^2}(z_0)\cap  B_1(0)}$, and satisfies  $\eta =1$ in $[ \hat\eps r^2,r^2]\times \overline{B_{\hat \delta^2 r^2}(z_0)\cap B_1(0)}$ and $|\partial_t^k\partial_z^{\beta} \eta|\lesssim r^{-2k-2|\beta|}$. For the arguments that follow it is convenient to first consider the inhomogeneous equation
\[
\partial_t w - \rho^{-\sigma}\div\left(\rho^{\sigma+1}\grad w\right) = f,
\]
for some $f$ that will be specified later. Then $\eta w$ solves
\[
\partial_t (\eta w) - \rho^{-\sigma} \div\left(\rho^{\sigma+1}\grad(\eta w)\right) = \eta f + \left(\partial_t \eta - \rho^{-\sigma}\div\left(\rho^{\sigma+1}\grad \eta\right)\right)w - 2\rho\grad \eta\cdot \grad w.
\]
For further reference, we denote the left-hand side of this equation by $\tilde f$.
Notice that $\eta$ acts as a spatially constant function on the differential equation if $\hat \delta^2 r^2 \ge 2$. In this case, the equation and most of the arguments that follow simplify dramatically. We will thus focus on the case where $\eta$ is not spatially constant.
 Testing with $\rho^{\sigma} \eta w$, using the properties of $\eta$ and observing that $\rho(z)\lesssim r^2$ for $z\in \overline{ B_{\delta^2 r^2 }(z_0)\cap B_1(0)}$, we derive the Caccioppoli estimate
\begin{eqnarray}
\lefteqn{\int_{\hat \eps r^2}^{r^2} \int_{B_{\hat \delta^2 r^2} (z_0)} |\grad w|^2\, d\mu_{\sigma+1} dt}\nonumber\\
&\lesssim& r^2 \int_{\eps r^2}^{r^2} \int_{B_{\delta^2 r^2}(z_0)} f^2\, d\mu_{\sigma} dt+  \frac1{r^2} \int_{\eps r^2}^{r^2} \int_{B_{\delta^2 r^2}(z_0)} w^2\, d\mu_{\sigma} dt.\label{19aa}
\end{eqnarray}
Here we have written $B_{\ell}(z_0)$ for $B_{\ell}(z_0)\cap \overline{B_1(0)}$ and will do so in the following whenever it seems convenient.

The estimate in \eqref{19aa} is not yet the statement we are looking for since the weight on the left-hand side is still too small. We choose a new cut-off function by replacing $(\eps,\delta)$ by $(\hat \eps,\hat \delta)$ and $(\hat \eps,\hat \delta)$ by $(\tilde \eps,\tilde \delta)$ in the above definition. We denote this function again by $\eta$.  Invoking the estimate from Lemma \ref{L7} with $w$ and $f$ replaced by $\eta w$ and $\tilde f$ we obtain
\[
\int_0^T \int \left( (\partial_t (\eta w))^2 + \left|\grad(\eta w)\right|^2 + |\rho \grad^2(\eta w)|^2\right)\, d\mu_{\sigma} dt\lesssim \int_0^T \int \tilde f^2\, d\mu_{\sigma}dt.
\]
By the definition of $\tilde f$, using the properties of $\eta$ and the fact that $\rho(z)\lesssim r^2$ in the support of $\eta$, 
\begin{eqnarray*}
\lefteqn{\int_0^T \int \tilde f^2\, d\mu_{\sigma}dt}\\
& \lesssim& \int_{\hat \eps r^2}^{r^2} \int_{B_{\hat \delta^2 r^2}(z_0)} \left( f^2 +\frac1{r^4} w^2\right)\, d\mu_{\sigma}dt + \frac1{r^2}  \int_{\hat \eps r^2}^{r^2} \int_{B_{\hat \delta^2 r^2}(z_0)}|\grad w|^2\, d\mu_{\sigma +1}dt.
\end{eqnarray*}
Combining the previous two estimates now yields
\begin{eqnarray}
\lefteqn{\int_{\tilde \eps r^2}^{r^2} \int_{B_{\tilde \delta^2 r^2} (z_0)} \left((\partial_t w)^2 +|\grad w|^2 + |\rho \grad^2 w|^2\right)\, d\mu_{\sigma} dt}\nonumber \\
&\lesssim &  \int_{\eps s^2}^{r^2} \int_{B_{\delta^2 r^2}(z_0)} f^2\, d\mu_{\sigma} dt+  \frac1{ r^4} \int_{\eps r^2}^{r^2} \int_{B_{\delta^2 r^2}(z_0)} w^2\, d\mu_{\sigma} dt.\label{19a}
\end{eqnarray}
In particular, choosing $f=0$, we obtain the statement in the cases $(k,|\beta|)=(0,1)$ and $(k,|\beta|) = (1,0)$.

To prove the analogous statement for the second-order spatial derivatives, we have to lower the exponent on weight factor. For this purpose, we first differentiate the equation in direction of the tangent vectors close to $z_0\in \partial B$. Without loss of generality we assume that $z_0=e_1$. We let $R_2,\dots, R_N$ be rotation matrices in the sense that $R_i z\perp z$ for all $i$ and all $z\in \R^N$, and such that $\{z, R_2 z, \dots, R_N z\}$ forms a basis of $\R^N$ for any $z\in B_{\ell}(e_1)$, provided that $\ell $ is chosen sufficiently small. Suppose moreover that the $R_i$'s are such that $R_i e_1=e_i$.\footnote{E.g., in 3D, we let
\[
R_2= \left(
\begin{array}{ccc} 0&-1&0\\ 1&0&-1\\ 0&1&0\end{array}
\right)
\quad\mbox{and}\quad
R_3 =\left(
\begin{array}{ccc} 0&0&-1\\ 0&0&-1\\1&1&0\end{array}
\right),
\]
so that $R_2 e_1 = e_2$ and $R_3 e_1 = e_3$. Then, for all $z\in\R^3$ it holds $R_1 z\perp z$ and $R_2z\perp z$. Moreover, for any $z\in B_{\ell}(e_1)$ with $\ell $ sufficiently small, the set $\{z, R_2 z, R_3 z\}$ forms a basis of $\R^3$.}
At any point $z\in B_{\delta^2 r^2}(z_0)$ we let $\tilde \partial_i  = \tilde z_i \cdot \grad$ where $\tilde z_i = R_i z$ for some fixed $i$. 
Differentiating the parabolic equation yields
\[
\partial_t \tilde \partial_i w - \rho^{-\sigma} \div (\rho^{\sigma +1} \grad\tilde \partial_i w) =-2\rho R_i:\grad^2 w +(\sigma +1) \tilde \partial_i w,
\]
where $A:B := \sum_{i j} A_{ij} B_{ij}$ for any two matrices $A$ and $B$. Applying \eqref{19a} to this equation (with suitable changes between $(\eps,\delta)$, $(\tilde \eps, \tilde \delta)$, and $(\hat \eps,\hat \delta)$) and then once again to the original equation yields
\begin{eqnarray*}
\lefteqn{\int_{\tilde \eps r^2}^{r^2} \int_{B_{\tilde \delta^2 r^2} (z_0)}\left( (\partial_t \tilde \partial_i w)^2 +  |\grad\tilde \partial_i w|^2 +  |\rho\grad^2 \tilde \partial_i w|^2 \right) \, d\mu_{\sigma} dt}\\
&\lesssim &\int_{\hat \eps r^2}^{r^2} \int_{B_{\hat \delta^2 r^2}(z_0)} |\rho \grad^2 w|^2\, d\mu_{\sigma} dt +  \frac1{ r^4} \int_{\hat \eps r^2}^{r^2} \int_{B_{\hat \delta^2 r^2}(z_0)} |\grad w|^2\, d\mu_{\sigma} dt\\
&\stackrel{\eqref{19a}}{\lesssim} & \frac1{r^8} \int_{\eps r^2}^{r^2} \int_{B_{\delta^2 r^2}(z_0)}  w^2\, d\mu_{\sigma} dt.
\end{eqnarray*}
Applying \eqref{19a} again, we can change the order of differentiation on the left-hand side and obtain

\begin{eqnarray*}
\lefteqn{\int_{\tilde \eps r^2}^{r^2} \int_{B_{\tilde \delta^2 r^2} (z_0)} \left( |\tilde \partial_i\partial_t w|^2 +|\tilde \partial_i\grad w|^2 +  |\rho \tilde \partial_i\grad^2 w|^2\right) \, d\mu_{\sigma} dt}\\
&\lesssim&  \frac1{r^8 }  \int_{\eps r^2}^{r^2} \int_{B_{\delta^2 r^2}(z_0)}  w^2\, d\mu_{\sigma} dt,\hspace{8em}
\end{eqnarray*}
and the estimate is independent of the particular value of  $i\in\{2,\dots,N\}$. This procedure can be iterated and yields for any multi-index $\beta\in \N_0^{N-1}$
\begin{eqnarray}
\lefteqn{\int_{\tilde \eps r^2}^{r^2} \int_{B_{\tilde \delta^2 r^2} (z_0)} \left( |\tilde \partial_{z'}^{\beta}\partial_t w|^2 +|\tilde\partial_{z'}^{\beta}\grad w|^2 + |\rho\tilde \partial_{z'}^{\beta}\grad^2 w|^2\right) \, d\mu_{\sigma} dt}\nonumber\\
&\lesssim&  \frac1{r^{4(|\beta|+1)}}  \int_{\eps r^2}^{r^2} \int_{B_{\delta^2 r^2}(z_0)}  w^2\, d\mu_{\sigma} dt.\label{19b}\hspace{8em}
\end{eqnarray}
Here, we have used the notation $z'=(z_2,\dots,z_N)^T$ and defined $\tilde \partial_{z'}^{\beta}$ in the obvious way.

In the following, we will write $\tilde \grad' $ to denote the tangential gradient operator, that means,  $(\tilde \partial_2,\dots,\tilde \partial_N)^T$.

We turn to the estimates of the derivatives in direction normal to the boundary.  Differentiating the equation in normal direction, using $\tilde \partial_1 = z\cdot \grad$ yields
\[
\partial_t \tilde \partial_1 w - \rho^{-\sigma}\div(\rho^{\sigma+1} \grad \tilde \partial_1 w) + \laplace w=0.
\]
Observing that $z\cdot \grad\tilde \partial_1 w = \tilde \partial_1 w +  (z-\hat z)\otimes (z+\hat z):\grad^2 w + \hat z\otimes \hat z:\grad^2 w$, where $\hat z = z/|z|$,
 and $\laplace w - \hat z\otimes \hat z:\grad^2 w  = O(|\tilde \grad' \grad w|)$ because $\{\hat z, R_1 z,\dots R_N z\}$ is a basis of $\R^N$ and $R_i z\perp z$,
\footnote{Since $\{\hat z, R_2z,\dots,R_Nz\}$ is a basis of $\R^N$, we find $\lambda_{ij} \in \R$ such that $e_i = \lambda_{i1} \hat z + \sum_{j=2}^N \lambda_{ij} R_j z$. In particular, for any matrix $A\in\R^{N\times N}$, we have that $A_{ii} = e_i\otimes e_i:A = \left(\lambda_{i1}\hat z + \sum_{j=2}^N \lambda_{ij}R_j z\right)\otimes\left(\lambda_{i1}\hat z + \sum_{j=2}^N \lambda_{ij}R_j z\right):A$. If $A$ is symmetric, an expansion yields
\[
A_{ii} = (\lambda_{1i})^2 \hat z\otimes \hat z: A + \sum_{j,k=2}^N \lambda_{ij}\lambda_{ik} (R_jz)\otimes (R_k z):A + 2\sum_{j,k=2}^N \lambda_{ij}\lambda_{ik} (R_j z)\otimes \hat z:A.
\]
Notice that $e_i\cdot \hat z = \lambda_{i1}$ by orthogonality, and thus $1 = |\hat z|^2 = \sum_{i=1}^N (e_i\cdot \hat z)^2  = \sum_{i=1}^N (\lambda_{i1})^2$. In particular, summing over all $i$ yields
\[
\tr(A)  = \sum_{i=1}^N A_{ii} = \hat z\otimes \hat z:A + \sum_{i=1}^N \sum_{j,k=2}^N \lambda_{ij}\lambda_{ik} (R_jz)\otimes (R_k z):A + 2\sum_{i=1}^N\sum_{j,k=2}^N \lambda_{ij}\lambda_{ik} (R_j z)\otimes \hat z:A.
\]
}
we can rewrite this equation as
\begin{equation}
\label{19bb}
\partial_t\tilde \partial_1 w - \rho^{-(\sigma+1)}\div(\rho^{\sigma+2}\grad\tilde \partial_1 w) =  \tilde \partial_1 w + O\left(|\rho\grad^2 w|\right) +  O\left(|\tilde \grad'\grad w|\right).
\end{equation}
As in the derivation of \eqref{19a}, we can smuggle a cut-off function into the equation. Applying \eqref{19a} with $f$ being the right-hand side of the resulting equation and with $\sigma$ replaced by $\sigma+1$ yields the estimate
\begin{eqnarray*}
\lefteqn{\int_{\tilde \eps r^2}^{ r^2} \int_{B_{\tilde \delta^2 r^2} (z_0)}\left(( \partial_t \tilde \partial_1 w)^2 + |\grad \tilde \partial_1 w|^2 + |\rho \grad^2 \tilde \partial_1 w|^2\right)\, d\mu_{\sigma+1} dt}\\
&\lesssim&
\int_{\eps r^2}^{r^2} \int_{B_{\delta^2 r^2}(z_0)} \left(  |\rho\grad^2w|^2 + |\tilde \grad' \grad w|^2\right)\, d\mu_{\sigma+1}dt\\
&&\mbox{} + 
\frac1{r^2} \int_{ \eps r^2}^{r^2} \int_{B_{\delta^2 r^2}(z_0)}  |\rho\grad \tilde \partial_1 w|^2\, d\mu_{\sigma+1} dt+
\frac1{r^4} \int_{ \eps r^2}^{r^2} \int_{B_{\delta^2 r^2}(z_0)}  |\grad w|^2\, d\mu_{\sigma+1} dt.
\end{eqnarray*}
The terms on the right can be estimated using that $\mu_{\sigma+1}\lesssim \min\{1, r^2\} \mu_{\sigma}$ in the domain of integration. We change the order of differentiation and obtain via \eqref{19b}
\begin{eqnarray*}
\lefteqn{
\int_{\tilde \eps r^2}^{r^2} \int_{B_{\tilde \delta^2 r^2} (z_0)}\left( (\tilde \partial_1 \partial_t w)^2 +|\tilde \partial_1\grad  w|^2 +|\rho\tilde \partial_1 \grad^2 w|^2 \right) \, d\mu_{\sigma+1} dt }\\
&\lesssim&  \frac1{r^6} \int_{\eps r^2}^{r^2} \int_{B_{\delta^2 r^2}(z_0)}  w^2\, d\mu_{\sigma}dt.\hspace{8em}
\end{eqnarray*}
We can argue similarly for the third order derivatives. Differentiating once more in transversal direction yields the estimate
\[
\int_{\tilde \eps r^2}^{r^2} \int_{B_{\tilde \delta^2 r^2} (z_0)}\left( (\tilde \partial_1^2w)^2 + |\tilde \partial_1^2\grad  w|^2 \right) \, d\mu_{\sigma+2} dt \lesssim  \frac1{r^8}  \int_{ \eps r^2}^{r^2} \int_{B_{\delta^2 r^2}(z_0)}  w^2\, d\mu_{\sigma}dt.
\]

We finally use a Hardy-type inequality to control the second order derivatives with respect to the measure $\mu_{\sigma}$: Using the fact that $(1-|z|^2)^{\sigma} \sim -\div(z(1-|z|^2)^{\sigma+1})$ close to the boundary, it can be proved via integration by parts that
\begin{eqnarray*}
\lefteqn{
\int_{\tilde \eps r^2}^{r^2} \int_{B_{\tilde \delta^2 r^2} (z_0)}\left( |\tilde \partial_1 \partial_t w|^2 +|\tilde \partial_1 \grad w|^2\right) \, d\mu_{\sigma}dt}\\
&\lesssim& \int_{\eps r^2}^{r^2} \int_{B_{\delta^2 r^2}(z_0)} \left( |\tilde \partial_1^2 \partial_t w|^2 + |\tilde \partial_1^2 \grad w|^2\right)\, d\mu_{\sigma +2}dt .
\end{eqnarray*}
A combination of the previous estimates now yields
\[
\int_{\tilde \eps r^2}^{r^2} \int_{B_{\tilde \delta^2 r^2} (z_0)} \left( |\tilde \partial_1 \partial_t w|^2+ |\tilde \partial_1 \grad w|^2\right)\, d\mu_{\sigma}dt \lesssim\frac1{r^8} \int_{\eps r^2}^{r^2} \int_{B_{\delta^2 r^2}(z_0)}  w^2\, d\mu_{\sigma}dt,
\]
and thus, together with \eqref{19b} we get the first statement of the lemma for $(k,|\beta|)\in\{(1,1),(0,2)\}$. Control on higher order estimates can be obtained via iteration.

The second statement can be proved in a very similar (but easier) way. Indeed, choosing $\delta$ smaller if necessary,  Lemma \ref{L8c} enables us to derive the statement  in the Euclidean and unweighted setting. (In other words, the equation is strictly parabolic in this situation.) Such a result being standard, we omit further details.
\end{proof}

In the following lemma, we establish uniform bounds on solutions and their derivatives by combining the estimates from the previous lemma with Morrey-type inequalities.

\begin{lemma}\label{L9}
Let $w$ be a solution to the homogeneous equation \eqref{17a}
and let $z_0\in \overline{ B}$ and $\tau \ge 0$. Let $0<\eps< 1$ and $0<\delta < 1$ sufficiently small. Then, for all $k\in \N_0$, $\beta\in \N_0^N$ and $0<r\lesssim 1$, the following holds:
\begin{enumerate}
\item If $\sqrt{\rho(z_0)}\le \delta r$, then
\begin{equation}\label{23}
|\partial_t^k \partial_z^{\beta} w(t,z)|^2 \lesssim \frac{r^{-4k-4|\beta|}}{r^2 |B_r^d(z_0)|_{\sigma}} \int_{\tau}^{\tau+r^2} \int_{B_r^d(z_0)} w^2\, d\mu_{\sigma}dt
\end{equation}
for all $(t,z) \in [\tau + \eps r^2,\tau +r^2]\times B_{\delta r}^d(z_0)$.
\item If $\sqrt{\rho(z_0)}\ge \delta r$, then
\begin{equation}\label{23b}
|\partial_t^k \partial_z^{\beta} w(t,z)|^2 \lesssim \frac{r^{-4k-2|\beta|} \left(\rho(z_0)\right)^{-|\beta|}}{r^2|B_r^d(z_0)|_{\sigma}} \int_{\tau}^{\tau+r^2} \int_{B_r^d(z_0)} w^2\, d\mu_{\sigma}dt
\end{equation}
for all $(t,z) \in [\tau + \eps r^2,\tau +r^2]\times B_{\delta r}^d(z_0)$.
\end{enumerate}
\end{lemma}

\begin{proof}
Upon a shift in the time variable it is enough to consider the case $\tau=0$ only. As in the proof of the previous lemma, we will often write the handier $B_{\ell}(z_0)$ instead of $B_{\ell}(z_0)\cap \overline{B_1(0)}$.

 We will start considering the scenario $\sqrt{\rho(z_0)}\le \delta r$. According to Lemma \ref{L8c}, 
it is enough to prove \eqref{23} in the situation where $z_0$ is located on the boundary of the unit ball. In this case $B_{\delta r}^d(z_0)$ can be replaced by $B_{\delta^2r^2}(z_0)$. The statement in \eqref{23} now follows from the previous lemma via the estimate
\[
| v(t,z)|^2 \lesssim \sum_{\ell = 0}^{1} \sum_{|\gamma|=0}^{ N + p} \frac{ r^{4\ell + 4|\gamma|}}{r^2 |B_{\delta^2 r^2}(z_0)|_{\sigma}} \int_{ \eps r^2}^{r^2} \int_{B_{\delta^2 r^2}(z_0)\cap B}(\partial_t^{\ell}\partial_z^{\gamma} v)^2\, d\mu_{\sigma}dt,
\]
setting $v = \partial_t^k\partial_z^{\beta}w$ and where $p\in \N_0$ is such that $0\le \sigma+1 - p<1$ and $(t, z)\in[\eps r^2, r^2]\times  B_{\delta^2 r^2}(z_0)$.
(We do not attempt to give a sharp estimate here.)
Since $z_0$ is on the boundary, if $\delta $ is sufficiently small, we may replace $B_1(0)$ by the half-space $H=\{z\in \R^N:\: z_N>0\}$ and assume that $z_0=0$. By the equivalence of norms, we may also consider cubes instead of balls, and upon rescaling $t = r^2 \hat t$, $z = r^2\hat z$ and $v=\hat v$, the previous estimate follows from
\[
|\hat v(\hat t,\hat z)|^2 \lesssim \sum_{\ell = 0}^{1} \sum_{|\gamma|=0}^{ N + p}   \int_{  \eps}^{1 } \int_{Q_{\delta^2}(0)\cap H}(\partial_{\hat t}^{\ell}\partial_{\hat z}^{\gamma} \hat v)^2\, \tilde z_N^{\sigma}\, d\tilde zd\tilde t.
\]
Here, $Q_r(0)$ denotes a cube in $\R^N$, centered at $0$ and with side length $2r$.
In view of the elementary inequality $\sup_{[a,b]}|f|\lesssim \|f\|_{L^1(a,b)} + \|f'\|_{L^1(a,b)}$ (with constants dependent on $a$ and $b$), we further reduce the statement we have to prove to
\begin{equation}\label{24}
|\hat v (\hat z_N)|^2\lesssim \sum_{n=0}^{\max\{p,1\}} \int_0^{\delta} (\partial_{\hat z_N}^n \hat v)^2 \tilde z_N^{\sigma}\, d\tilde z_N.
\end{equation}
The latter is established as follows: We drop the hats and the subscript $N$ from $z$. For $z\in [0,\delta]$ and $q = \sigma +1 -p\in[0,1)$, we have
\[
v(z)\lesssim \int_0^{\delta} |v|\, dz +  \int_{0}^{\delta} |v'|\, dz \le \left(\int_0^{\delta} v^2 z^{q}\, dz + \int_0^{\delta} (v')^2 z^{q}\, dz\right)^{1/2} \left(\int_0^1\frac1{z^{q}}\, dz\right)^{1/2},
\]
where we have used the Cauchy--Schwarz inequality. Since $q<1$, the latter integral converges. Notice first that if $p=0$, that is $\sigma<0$, we can lower the weight, $z^q  =z^{\sigma+1}\le z^{\sigma}$, and \eqref{24} follows from the previous estimate. The same holds true if $p=1$, because then $q=\sigma$. From now on we assume that $p\ge 2$. 
Letting $\eta$ denote a cut-off function supported in $[0,2\delta]$ with $\eta=1$ in $[0,\delta]$, we obtain via integration by parts
\[
\int_0^1 \eta f^2 z^{q}\, dz \lesssim  \int_0^1 \eta (f')^2 z^{q+1}\, dz + \int_0^1 (\eta + |\eta'|) (f)^2 z^{q+1}\, dz.
\]
Choosing $f=v$ and $f=v'$, and repeating this procedure $p-1$ times yields
\[
\int_0^{\delta} (v^2 + (v')^2) z^{q}\, dz \lesssim \sum_{n=0}^p \int_0^{2\delta} (\partial_z^n v)^2 z^{q +p-1}\, dz.
\]
By the choice of $p$ and the definition of $q$, it is $q +p-1 = \sigma$, which proves \eqref{24}.

We finally comment on the proof of the second assertion. Thanks to Lemma \ref{L8c}, the assumption $\sqrt{\rho(z_0)}\ge \delta r$ converts the problem into the unweighted Euclidean setting. Thus, a standard Morrey inequality is applicable and yields the desired result. We omit details.
\end{proof}

For two arbitrarily fixed real numbers $a$ and $b$, $b>0$, we define the auxiliary function $\chi_{a,b}: \overline{B_1(0)}\times \overline{B_1(0)} \to [0,\infty)$ by
\[
\chi_{a,b}(z,z_0):= \frac{a\hat d(z,z_0)^2}{\sqrt{b+\hat d(z,z_0)^2}},
\]
where
\[
\hat d(z,z_0)^2 := \frac{|z-z_0|}{\rho(z) + \rho(z_0) + |z-z_0|} \sim d(z,z_0)^2.
\]
A short computation shows that
\begin{equation}\label{25a}
\sqrt{\rho(z)} |\grad_z \chi_{a,b} (z,z_0)|\lesssim |a|,
\end{equation}
which in particular implies that $\chi_{a,b}(\tacka,z_0)$ is Lipschitz continuous with respect to the induced metric $ d$, that is
\begin{equation}\label{26}
|\chi_{a,b}(z,z_0) - \chi_{a,b}(z',z_0)|\lesssim |a| d(z,z').
\end{equation}
Indeed, if we denote the gradient on $(\M,\g)$ by $\grad_{\g}$, then Lipschitz continuity of a mapping $\chi: \M\to \R$ with Lipschitz constant proportional to $|a|$ follows from the gradient estimate $\sqrt{\g(\grad_{\g} \chi, \grad_{\g}\chi)}\lesssim |a|$. Since in our case $\g \sim \rho^{-1}(dz)^2$ and $\grad_g \sim \rho\grad$, the latter is equivalent to \eqref{25a}.

The function $\chi_{a,b}$ will be a key ingredient in the derivation of the Carleson measure estimate (Proposition \ref{P2a}) and the Gaussian estimate (Proposition \ref{P2b}) below. The following result is an intermediate step towards both estimates.

\begin{lemma}\label{L10}Let $w$ be the solution to the initial value problem \eqref{17a}, \eqref{17b}. Then there exists a constant $C>0$ such that for all $z, z_0\in \overline{B_1(0)}$, $0<r\le \sqrt2$ and $t\ge r^2$ and any $k\in\N_0$ and $\beta\in\N_0^N$ it holds 
\[
|\partial_t^k\partial_z^{\beta} w(t,z)|\lesssim \frac{  r^{-2k-|\beta|} (r+ \sqrt{\rho(z)})^{-|\beta|}}{|B_r^d(z)|_{\sigma}^{1/2}} e^{ C a^2 t -\chi_{a,b}(z,z_0)} \|e^{\chi_{a,b}(\tacka,z_0)}g\|_{L_{\sigma}^2}.
\]
\end{lemma}

\begin{proof}For abbreviation, we will write $\chi:= \chi_{a,b}(\tacka,z_0)$. Notice that $e^{\chi} w$ solves the equation
\[
\partial_t (e^{\chi}w) - \rho^{-\sigma}\div\left(\rho^{\sigma+1}\grad(e^{\chi}w)\right) = -\rho^{-\sigma}\div\left(\rho^{\sigma+1}\grad e^{\chi}\right)w - 2\rho\grad w\cdot \grad e^{\chi}.
\]
Hence, testing with $\rho^{\sigma}e^{\chi}w$ and integrating by parts a multiple times yields the identity
\[
\frac{d}{dt} \frac12\int (e^{\chi}w)^2\, d\mu_{\sigma}  + \int|\grad(e^{\chi}w)|^2\, d\mu_{\sigma+1} = \int |\grad e^{\chi}|^2 w^2\, d\mu_{\sigma+1}.
\]
Invoking the gradient estimate \eqref{25a} and a Gronwall argument, this implies
\begin{equation}\label{27}
\int (e^{\chi} w(t))^2\, d\mu_{\sigma} \le e^{2Ca^2 t}\int (e^{\chi} g)^2\, d\mu_{\sigma}
\end{equation}
for some $C>0$. We now combine this estimate with the $L^{\infty}$--$L^2$ bounds from Lemma \ref{L9}. Estimates \eqref{23} and \eqref{23b} together imply the bound
\[
|\partial_t^k \partial_z^{\beta} w(t,z)|^2 \lesssim \frac{ r^{-4k-2|\beta|} (r + \sqrt{\rho(z)})^{-2|\beta|} }{|B_r^d(z)|_{\sigma}} \sup_{\tilde z\in B_r^d(z)} e^{-2\chi(\tilde z)} \sup_{t\le\tau+r^2} \int_{B_r^d(z)} (e^{\chi} w(t))^2\, d\mu_{\sigma}
\]
for all $0<r\lesssim 1$ and $t\in[\tau+\eps r^2,\tau +r^2]$. Hence, applying \eqref{27}, we have
\[
|\partial_t^k \partial_z^{\beta} w(t,z)|\lesssim \frac{ r^{-2k-|\beta|} (r+\sqrt{\rho(z)})^{-|\beta|}}{|B_r^d(z)|^{1/2}_{\sigma}} \sup_{\tilde z\in B_r^d(z)} e^{-\chi(\tilde z)+C a^2 (\tau+r^2)} \|e^{\chi} g\|_{L_{\sigma}^2}.
\]
It remains to  invoke the Lipschitz bound \eqref{26} and we choose $t=\tau+r^2$ for conclusion.
\end{proof}

The above estimate is valuable for estimating the small time behavior of solutions. For large times, we can do much better by taking into account the spectral information on the linear operator obtained by the author in \cite{Seis14}. We will recall the  results of \cite{Seis14} in every detail in Section \ref{S7} below. At this stage, only the following is relevant: A solution $w$ to the initial value problem \eqref{17a}, \eqref{17b} can be written as the semi-flow $w(t) = e^{-t\Ha} g$ on the homogeneous Sobolev space $\dot H^1_{\sigma+1} $. Here $\Ha$ is the self-adjoint (Friedrichs) extension of $-\rho\laplace +(\sigma +1)z\cdot \grad$ to $\dot H^1_{\sigma+1}$. The regularity obtained in Lemma \ref{L10} guarantees that solutions are in the domain of $\Ha$ for positive times. The work \cite{Seis14} contains a computation of the purely discrete spectrum of $\Ha$ which lies in the set $(0,\infty)$. In particular, if we denote by $\lambda_1$ the smallest eigenvalue, we have the spectral gap estimate
\begin{equation}
\label{27a}
\int \grad w\cdot \grad \Ha w\, d\mu_{\sigma+1} \ge \lambda_1 \int |\grad w|^2\, d\mu_{\sigma+1}.
\end{equation}
Notice that the latter is not true for the Hilbert space $L^2_{\sigma}$. Indeed, the latter contains nonzero constants (which are modded out in the diagonalization \cite{Seis14}), and thus, the corresponding operator on $L^2_{\sigma}$ has, in addition, a zero eigenvalue.

With these preparations, we are now in the position to state and prove the following:

\begin{lemma}\label{L10a} Let $w$ be the solution to the initial value problem \eqref{17a}, \eqref{17b}. Let $k\in\N_0$ and $\beta\in\N_0^N$ be such that $k+|\beta|\ge 1$. Then, for all $t\ge \frac12$ and $z\in \overline{B_1(0)}$ it holds
\[
|\partial_t^k\partial_z^{\beta} w(t,z)| \lesssim e^{-\lambda_1 t} \|\grad g\|_{L^2_{\sigma+1}}.
\]
\end{lemma}

\begin{proof}Since $w$ is a solution to the homogeneous equation $\partial_t w + \Ha w=0$, it holds that
\[
\frac{d}{dt} \frac12 \int |\grad w|^2\, d\mu_{\sigma+1}   =  - \int \grad w\cdot \grad \Ha w\, d\mu_{\sigma+1}.
\]
The spectral gap estimate \eqref{27a} and a Gronwall argument then yield
\[
\|\grad w\|_{L^2_{\sigma+1}} \le e^{-\lambda t} \|\grad g\|_{L^2_{\sigma+1}}.
\]
Notice that $c = \int w\, d\mu_{\sigma}$ is a conserved quantity for the homogeneous equation \eqref{17a}. Hence, with the help of Lemma \ref{L9} in which we choose $r=1$ , $\eps=1/2$ and $z_0=0$, and a Hardy--Poincar\'e estimate (e.g., \cite[Lemma 3]{Seis14}) we obtain for every $(t,z)\in [\tau + 1/2,\tau+1]\times \overline{B_1(0)}$ that
\[
|\partial_t^k\partial_z^{\beta} w(t,z)|^2 \lesssim \int_{\tau}^{\tau+ 1}  \int (w-c)^2\, d\mu_{\sigma} \lesssim \int_{\tau}^{\tau+1} \int |\grad w|^2\, d\mu_{\sigma+1} .
\]
We easily infer the statement of the lemma.
\end{proof}

The first main result of this subsection is the control of the Carleson measures $\|w\|_{X(p)}$ in terms of the Lipschitz norm of the initial datum.

\begin{prop}\label{P2a}
Let $w$ be the solution to the initial value problem \eqref{17a}, \eqref{17b}. Then
\[
\| w\|_{X(p)}+ \| w\|_{ \Lip}\lesssim \| g\|_{\Lip}.
\]
\end{prop}

\begin{proof}Upon a shift in the time variable, the exponential decay estimates of Lemma \ref{L10a} imply that
\[
|\partial^k_t\partial_z^{\beta} w(t,z)|\lesssim e^{-\lambda_1 t} \|w\left(1/2\right)\|_{\Lip},
\]
for ant $t\ge 1$, and thus the control of large time terms ($r\gtrsim 1$) in the Carleson measures follows from the control of the small time terms ($r\lesssim 1$).

Let $z_0\in \overline{B_1(0)}$, $0<r\lesssim  1$ and $t\in \left(\frac{r^2}2,r^2\right)$. Observing that $w-g(z_0)$ is a solution to the homogeneous equation with initial datum $g-g(z_0)$ and using that $\chi_{a,b}(z_0,z_0)=0$, we deduce from Lemma \ref{L10} that
\begin{eqnarray*}
\lefteqn{|\partial_t^k \left.\partial_z^{\beta}\right|_{z=z_0} (w(t,z) -g(z_0))| }\\
&\lesssim&\frac{r^{-2k-|\beta|} (r+\sqrt{\rho(z_0)})^{-|\beta|}}{|B_r^d(z_0)|_{\sigma}^{1/2}} e^{ C a^2t} \|e^{\chi_{a,b}(\tacka,z_0)}|\tacka - z_0|\|_{L_{\sigma}^2}\|g\|_{\Lip}
\end{eqnarray*}
for all $k\in\N_0$ and $\beta\in\N_0^N$. We now specify the choice of $a$ and $b$ by setting $a=-\frac1{r}$, $b=r^2$. We claim that
\begin{equation}\label{28}
\|e^{\chi_{-\frac1{r},r^2}(\tacka,z_0)}|\tacka - z_0|\|_{L_{\sigma}^2}\lesssim r(r+\sqrt{\rho(z_0)})|B_r^d(z_0)|_{\sigma}^{1/2}.
\end{equation}
Prior to establishing this estimate, we derive the statement of proposition. Thanks to \eqref{28}, the estimate from Lemma \ref{L10} (now with the above choice of $a$ and $b$) becomes
\[
|\partial_t^k\partial_z^{\beta} w(t,z_0)| \lesssim 
 r^{1-2k-|\beta|} (r+\sqrt{\rho(z_0)})^{1-|\beta|} \|g\|_{\Lip},
\]
for all $z_0\in \overline{B_1(0)}$, provided that $k+|\beta|\ge 1$.
In the case $(k,|\beta|) = (0,1)$, the terms depending on $r$ drop out and we obtain uniform control on $\grad w$ and thus on the Lipschitz norm of $w$.
Otherwise, if $(k,|\beta|)\in\{(1,1),(0,2),(0,3)\}$, integration over $Q_{r}^d(z) $ and rearranging terms yields the desired estimates on $\partial_t \grad w$, $\grad^2 w$, and $\grad^3 w$. Notice that in this argument we additionally use the estimate $r\le r+ \sqrt{\rho(z_0)}\lesssim r+ \sqrt{\rho(z)}$ for all $z_0\in B_r^d(z)$, which is a consequence of Lemma \ref{L8c}.

We finally turn to the proof of \eqref{28}. For $j\in\N$ consider the annuli $A_j := B_{jr}^{\hat d}(z_0)\setminus B_{(j-1)r}^{\hat d}(z_0)$, which cover $\overline{B_1(0)}$ and satisfy $A_j=\emptyset$ if $j\gg\frac1r$. An elementary computation shows that $\chi_{-\frac1{r},r^2}(z,z_0) \le-\frac{j-1}{\sqrt{5}}$ for $z\in A_j$. Hence, for all $j\lesssim \frac1r$,
\[
\int_{A_j} e^{2\chi_{-\frac1{r},r^2}(z,z_0)} |z-z_0|^2\, d\mu_{\sigma}(z)\lesssim e^{-\frac2{\sqrt{5}}j} j^2r^2 (jr +\sqrt{\rho(z_0)})^2 |A_j|_{\sigma}
\]
by the comparison formula for Euclidean and intrinsic balls, cf.\ Lemma \ref{L8}. 
Invoking the volume formula for intrinsic balls from Lemma \ref{L8b} we further estimate
\[
|A_j|_{\sigma}  \lesssim (jr)^N (jr +\sqrt{\rho(z_0)})^{N+2\sigma} \lesssim j^{\kappa} |B_r^d(z_0)|_{\sigma}
\]
for some $\kappa	>0$. Notice that the latter is only valid because $j\lesssim \frac1r$. Hence, combining the previous two estimates and summing over all $j\in\N_0$ yields \eqref{28} as desired. This concludes the proof.
\end{proof}

From Lemma \ref{L10}, we also deduce the existence of a heat kernel, that is, the integral kernel of the heat semi-group operator $e^{-t\L}$ (or $e^{t\laplace_{\mu}}$, 
if the stick to the notation of the introduction to this section). It is well-known that heat kernels are smooth and symmetric functions and inherit the semi-group property from the semi-group operator, cf.\ \cite[Theorem 3.3]{Grigoryan06}. For the sake of a self-contained presentation, we derive the properties that are relevant for subsequent analysis directly from the Lemmas \ref{R2}, \ref{L10},  and \ref{L10a} above. In particular, we establish sharp decay estimates on the kernel, that are called Gaussian estimates in analogy to the Euclidean situation. For general parabolic equation, such estimates we derived by Aronson in \cite{Aronson67}.

\begin{prop}\label{P2b}
There exists a function $G:(0,\infty)\times \overline{B_1(0)}\times \overline{B_1(0)}\to \R$ with the following properties:
\begin{enumerate}
\item For any $g\in L^2_{\sigma}$, the solution $w$ to the initial value problem \eqref{17a}, \eqref{17b} has the representation
\[
\partial_t^k \partial_z^{\beta} w(t,z) = \int \partial_t^k \partial_z^{\beta}G(t,z,z')g(z')\, d\mu_{\sigma}(z')
\]
for all $k\in \N_0$, $\beta\in\N_0^N$, $t>0$ and $z\in \overline{B_1(0)}$.
\item For every $(t,z,z')\in(0,\infty)\times \overline{B_1(0)}\times \overline{B_1(0)}$, it holds
\[
G(t,z,z') = G(t,z',z),
\]
i.e., $G(t,\tacka,\tacka)$ is symmetric for every $t>0$.
\item It holds
\[
\partial_t G - \rho^{-\sigma}\grad_z\!\cdot\left(\rho^{\sigma+1}\grad_z G\right) = 0
\]
and, for all $z\in \overline{B_1(0)}$
\[
\rho^{\sigma} G(t,z,\tacka) \to \delta_z\quad\mbox{as }t\downarrow 0\mbox{ in the sense of distributions}.
\]
In particular, $G$ is a fundamental solution of \eqref{17a}.
\item For all $z,z'\in \overline{B_1(0)}$, $t\in(0,\infty)$ and $k\in\N_0$, $\beta\in\N_0^N$, it holds that
\[
|\partial_t^k\partial_z^{\beta} G(t,z,z')| \lesssim \frac{\sqrt{t}^{-2k-|\beta|}\left(\sqrt{t} + \sqrt{\rho(z)}\right)^{-|\beta|}}{|B_{\sqrt{t}}^d(z)|_{\sigma}^{1/2} |B_{\sqrt{t}}^d(z')|_{\sigma}^{1/2}} e^{-C d(z,z')^2/t}.
\]
\item For all $z,z'\in \overline{B_1(0)}$, $t\in(1,\infty)$ and $k\in\N_0$, $\beta\in\N_0^N$ with $k+|\beta|\ge1$ it holds
\[
|\partial_t^k\partial_z^{\beta} G(t,z,z')|\lesssim e^{-\lambda_1t}.
\]
\end{enumerate}
\end{prop}


%

\begin{proof}
For the existence of the heat kernel, we observe that Lemma \ref{L10} applied with $a=0$, and thus $\chi\equiv0$, shows the boundedness of the linear mapping
\[
L_{\sigma}^2\ni g\mapsto \partial_t^k\partial_z^{\beta} w(t,z)\in \R,
\]
for all  $(t,z)\in (0,\infty)\times \overline{B_1(0)}$.
Hence Riesz' representation theorem yields the existence of a unique function $G_{k,\beta}(t,z,\tacka)\in L_{\sigma}^2$ such that
\[
\partial_t^k \partial_z^{\beta} w(t,z) = \int G_{k,\beta}(t,z,z')g(z')\, d\mu_{\sigma}(z').
\]
By uniqueness, it is clear that $\partial_t^k\partial_z^{\beta} G = G_{k,\beta}$, where $G = G_{0,0}$. Moreover, symmetry follows from the symmetry of the heat semi-group operator in  Lemma \ref{R2}.

We now prove the Gaussian estimate. It is a consequence of the auxiliary estimate
\begin{equation}\label{29}
|\partial_t^k\partial_z^{\beta} G(t,z,z')| \lesssim \frac{ r^{-2k-|\beta|} (r+\sqrt{\rho(z)})^{-|\beta|}}{|B_r^d(z)|_{\sigma}^{1/2} |B_r^d(z')|_{\sigma}^{1/2} } e^{ C a^2 t} e^{\chi_{a,b}(z,z_0) - \chi_{a,b}(z',z_0)},
\end{equation}
which holds for any $z,z',z_0\in \overline{B_1(0)}$, $t\ge r^2/2\ge0$ and $a\in\R$, $b>0$. Indeed, choosing $z'=z_0$, $b=3\hat  d(z,z')^2$ and $a=-2\ell$ with $\ell>0$, we deduce from \eqref{29} that
\[
|\partial_t^k\partial_z^{\beta} G(t,z,z')| \lesssim \frac{ r^{-2k-|\beta|} (r+\sqrt{1-|z|})^{-|\beta|}}{|B_r^d(z)|_{\sigma}^{1/2} |B_r^d(z')|_{\sigma}^{1/2} } e^{ 4 C \ell ^2 t - \ell \hat d(z,z') }.
\]
Optimizing the exponent with respect to $\ell$ yields $\ell t\sim \hat d(z,z')\sim d(z,z')$. The Gaussian estimate follows by choosing $r=\sqrt{t}$.

We turn to the proof of \eqref{29}.
Keeping $a\in \R$ and $b>0$ be arbitrarily fixed, we write $\chi = \chi_{a,b}(\tacka,z_0)$ for abbreviation.
Let $\xi \in L^1_{\sigma}: = L^1(\mu_{\sigma})$ be fixed and such that $g_{\xi}:= e^{\chi} |B_r(\tacka)|_{\sigma}^{1/2} \xi\in L_{\sigma}^2$. Then $w_{\xi}$ given by
\[
w_{\xi}(t,z) = \int  G(t,z,z') g_{\xi} (z')\, d\mu_{\sigma}(z')
\]
is a solution to the linear homogeneous equation with initial datum $g_{\xi}$. Applying Lemma \ref{L10} with $a$ replaced by $-a$, i.e., $\chi$ replaced by $-\chi$, and with initial time $t/2$, we obtain
\begin{equation}
\label{30}
|\partial_t^k\partial_z^{\beta} w_{\xi}(t,z)|\lesssim \frac{r^{-2k-|\beta|} (r + \sqrt{\rho(z)})^{-|\beta|}}{|B_r^d(z)|_{\sigma}^{1/2} } e^{ C a^2t/2 + \chi(z)}\|e^{-\chi} w_{\xi}\left(\frac{t}2\right)\|_{L_{\sigma}^2}.
\end{equation}
In particular, since $g\in L_{\sigma}^2$ precisely if $h= e^{\chi}g\in L_{\sigma}^2$, there exists a well-defined operator
\[
\A: L_{\sigma}^2\ni h\mapsto e^{-\chi}h=g\mapsto w\mapsto w\left(\frac{t}2,\tacka\right) \mapsto |B_r^d(\tacka)|_{\sigma}^{1/2} e^{\chi} w\left(\frac{t}2,\tacka\right) \in L^{\infty}.
\]
This operator is linear and bounded. Indeed, thanks to Lemma \ref{L10} it holds
\[
|B^d_r(z)|_{\sigma}^{1/2} e^{\chi(z)}|w\left(\frac{t}2,z\right)|  \lesssim e^{ C a^2 t/2} \|h\|_{L_{\sigma}^2},
\]
or, equivalently,
\[
\|\A h\|_{L^{\infty}} \lesssim  e^{ C a^2 t/2} \|h\|_{L_{\sigma}^2}.
\]
By the definition of $\A h$ and the symmetry of the heat semi-group operator, Lemma \ref{R2}, we have the identity
\[
\int \A h\xi\, d\mu_{\sigma}  =  \int e^{-\chi} w_{\xi}\left(\frac{t}2,\tacka\right) h\, d\mu_{\sigma}.
\]
Therefore, the action of the dual operator $\A^*: (L^{\infty})^* \to L_{\sigma}^2$ on functions $\xi\in L^1_{\sigma}$ with $g_{\xi}\in L^2_{\sigma}$ is $\A^*\xi = 
e^{-\chi}w_{\xi} \left(\frac{t}2,\tacka\right)$. In particular, because $\|\A\| = \|\A^*\|$, we must have
\begin{equation}
\label{33}
\|e^{-\chi} w_{\xi}\left(\frac{t}2,\tacka\right)\|_{L_{\sigma}^2} \lesssim e^{ C a^2 t/2}\|\xi\|_{L^1_{\sigma}}.
\end{equation}
%
%
We now conclude
\begin{eqnarray*}
\lefteqn{\int \partial_t^k\partial_z^{\beta} G(t,z,\tacka) g_{\xi} \, d\mu_{\sigma}}\\
&=& \partial_t^k\partial_z^{\beta} w_{\xi}(t, z)\\
&\stackrel{\eqref{30}}{\lesssim}& \frac{r^{-2k-|\beta|} (r+\sqrt{\rho(z)})^{-|\beta|}}{|B_r^d(z)|_{\sigma}^{1/2}} e^{C a^2 t/2 + \chi(z)} \|e^{-\chi} w_{\xi}\left(\frac{t}2\right)\|_{L_{\sigma}^2}\\
&\stackrel{\eqref{33}}{\lesssim}& \frac{r^{-2k-|\beta|} (r+\sqrt{\rho(z)})^{-|\beta|}}{|B_r^d(z)|_{\sigma}^{1/2}} e^{Ca^2 t/2 + \chi(z)} \|\xi\|_{L^1_{\sigma}}.
\end{eqnarray*}
The last expression in the above chain of inequalities is bounded for any $\xi \in L^1_{\sigma}$ and thus, by approximation, the estimate holds for any such $\xi$. On the other hand, thanks to the duality $L^{\infty}  \cong (L^1_{\sigma})^*$, we have that
\begin{eqnarray*}
\lefteqn{e^{\chi(z')} |B_r^d(z')|_{\sigma}^{1/2} |\partial_t^k \partial_z^{\beta} G(t,z,z')|}\\
&\le & \sup_{\|\xi\|_{L^{1}_{\sigma}}\le 1} \iint  e^{\chi} |B_r^d(\tacka)|_{\sigma}^{1/2} \partial_t^k\partial_z^{\beta} G(t,z,\tacka)\xi\, d\mu_{\sigma},
\end{eqnarray*}
which in combination with the previous estimate yields \eqref{29}.

Since $G$ is a smooth kernel, it satisfies \eqref{17a} itself, and $\rho^{\sigma}G(t,z,\tacka) \to \delta_{z}$ in the sense of distributions.

Finally, since $G(\tacka,\tacka,z')$ is a solution to \eqref{17a}, an application of Lemma \ref{L10a} and the Gaussian estimate yields
\[
|\partial_t^k\partial_z^{\beta} G(t,z,z')| \lesssim e^{-\lambda_1 t} \|\grad G\left(1/2,\tacka,z'\right)\|_{L^2_{\sigma+1}} \lesssim e^{-\lambda_1t}.
\]
This concludes the proof of Proposition \ref{P2b}.
\end{proof}

\subsection{Carleson measure estimates for the inhomogeneous problem}\label{S:5.3}

This subsection is devoted to the study of the initial value problem for the inhomogeneous parabolic equation
\begin{eqnarray}
\partial_t w -  \rho^{-\sigma}\div\left(\rho^{\sigma+1}\grad w\right) &=& f,\label{32a}\\
w(0,\tacka) &=& 0\label{32b}.
\end{eqnarray}
Thanks to the existence of a smooth solution kernel $G(t,z,z')$ for the homogeneous initial value problem, cf.\ Proposition \ref{P2b} above, solutions to \eqref{32a}, \eqref{32b} have an integral representation. This is the content of the following

\begin{lemma}[Duhamel's principle]\label{R4}
For $(t,t',z,z')\in (0,\infty)\times (0,\infty)\times\overline{B_1(0)}\times \overline{B_1(0)}$ with $t'<t$ define $G(t,t'z,z'): = G(t-t',z,z')$. If $f\in L^2((0,\infty); L^2_{\sigma})$ and $w$ is the solution to the initial value problem \eqref{32a}, \eqref{32b}, then
\[
w(t,z) = \int_0^t \int G(t,t',z,z') f(t',z')\, d\mu_{\sigma}(z') dt',
\]
for all $(t,z)\in(0,\infty)\times \overline{B_1(0)}$.
\end{lemma}

\begin{proof}
The statement follows from the fact that $G$ is a fundamental solution to \eqref{17a}, see Proposition \ref{P2b}.
\end{proof}

In Lemma \ref{L7}, we established $L^2(L^2_{\sigma})$ estimates on $\partial_t w$, $\grad w$ and $\rho \grad w$
for solutions to the inhomogeneous problem \eqref{32a}, \eqref{32b}. As a consequence of the Gaussian estimates in Proposition \ref{P2b}, it turns out that these estimates carry over to any $L^p(L^p_{\sigma})$ space, where $L^p_{\sigma}:= L^p(\mu_{\sigma})$, provided that $p\in (1,\infty)$. 
We will thus call $\ell$, $k$, and $\beta$ in $\rho^{\ell}\partial_t^k\partial_z^{\beta}w $ {\em Calderon--Zygmund exponents} whenever
\[
(\ell,k,|\beta|) \in \left\{(0,1,0),(0,0,1), (1,0,2)\right\}.
\]
For these exponents, the kernels $\rho^{\ell} \partial_t^k\partial_z^{\beta} G$ satisfy so-called Calderon--Zygmund cancellation conditions, which in turn yield the $L^p(L^p_{\sigma})$ estimates. Going one step further, Muckenhoupt theory allows for dropping the weights if $p<\infty$ is chosen sufficiently large. We thus deduce maximal regularity estimates in regular $L^p = L^p(L^p)$ spaces.

We refrain from establishing the cancellation conditions in this paper, which would be a straightforward but tedious calculation based on the Gaussian estimates of Proposition \ref{P2b}. We refer the interested reader to Kienzler's work \cite{Kienzler14}, specifically Corollary 3.18, therein.

\begin{lemma}\label{L12}
Let $w$ be the solution to the initial value problem \ref{32a}, \eqref{32b}. Then for any $p>\max\left\{1,\frac1{\sigma+1}\right\}$ it holds
\[
\| \partial_t  w\|_{L^p} + \|\grad w\|_{L^p} + \|\rho \grad^2 w\|_{L^p}\lesssim  \| f\|_{L^p}
\]
and
\[
\| \partial_t \grad w\|_{L^p} + \|\grad^2 w\|_{L^p} + \|\rho \grad^3 w\|_{L^p}\lesssim  \|\grad f\|_{L^p}.
\]
\end{lemma}

\begin{proof}
We first notice that Lemma \ref{L7} ensures that the mappings 
\[
L^2(L_{\sigma}^2)\ni f\mapsto \rho^{\ell} \partial_t^k\partial_z^{\beta}w \in L^2(L_{\sigma}^2)
\]
are continuous if $\ell$, $k$, and $\beta$ are Calderon--Zygmund exponents. The corresponding kernels,  $\rho^{\ell}\partial_t^k\partial_z^{\beta} G$, satisfy the Calderon--Zygmund cancellation condition, cf.\ \cite[Corollary 3.18]{Kienzler14}, and thus, maximal regularity theory applies and yields
\[
\|\rho^{\ell} \partial_t^k\partial_z^{\beta} w\|_{L^p(L^p_{\sigma})}\lesssim \|f\|_{L^p(L^p_{\sigma})}.
\]
Furthermore, the Muckenhoupt theory allows for dropping the weights. More precisely, $\rho^{\theta - \sigma}$ is a $p$-Muckenhoupt weight with respect to $\mu_{\sigma}$ if and only if $-1<\theta< p(\sigma +1) - 1$. We can thus translate the maximal regularity theory from $L^p(L^p_{\sigma})$ to $L^p(L^p(\rho^{\theta-\sigma}\, \mu_{\sigma})) = L^p(L^p_{\theta})$. In particular, for $\theta=0$, it holds
\begin{equation}\label{34}
\|\rho^{\ell} \partial_t^k\partial_z^{\beta} w\|_{L^p}\lesssim \|f\|_{L^p}
\end{equation}
provided that $p> \max\left\{1, \frac1{\sigma+1}\right\}$.

The second estimate is obtained from \eqref{34} by differentiating the equation and iteration. Thanks to strict parabolicity, this amounts to a standard exercise in the interior of the unit ball. Close to the boundary, however, we have to argue more carefully. In fact, after a suitable localization, the techniques developed in the proof of Lemma \ref{L8a} are applicable. We claim that
\[
\|\rho^{\ell} \partial_t^k \partial_z^{\beta}\grad w\|_{L^p} \lesssim \|f\|_{L^p(J\times B)} + \|\grad f\|_{L^p},
\]
and leave the details to the reader. 
Let now $h(t) := |B_1(0)|^{-1}\int f(t,z)\, dz$ and $H(t):= \int_0^t h(\tau)\, d\tau$. Obviously, $w-H$ is a solution to the equation with inhomogeneity $f-h$, and the previous estimate applies. Invoking Poincar\'e's inequality $\|f-h\|_{L^p}\lesssim \|\grad f\|_{L^p}$ and observing that $\partial_t^k\partial_z^{\beta}\grad (w-H) = \partial_t^k\partial_z^{\beta} \grad w$ for any choice of $k$ and $\beta$ yields the second statement of the lemma.
\end{proof}

In the proof of Proposition \ref{P3} it will be convenient to decompose $f = \eta f + (1-\eta) f$ where $\eta$ denotes a cut-off function which is constantly $1$ in the cylinder $Q_r^d(z_0)$ and vanishes outside the larger cylinder
\[
\widehat Q_{r}^d (z_0) :=\left(\frac{r^2}4,r^2\right)\times B_{2r}^d(z_0).
\]
First, we provide a ``diagonal''  $L^p$ estimate.

\begin{lemma}\label{L13}
Let $w$ be a solution to the initial value problem \eqref{32a}, \eqref{32b}. Suppose that $\spt f\subset \widehat Q_r^d(z_0)$ for some $z_0\in \overline{B_1(0)}$ and $0<r\le \sqrt2$. Then for all Calderon--Zygmund exponents $\ell$, $k$ and $\beta$ and $p>\max\left\{1,\frac1{\sigma+1}\right\}$ it holds
\[
r^2 |Q_{r}^d(z_0)|^{-\frac1p} \|\rho^{\ell} \partial_t^k\partial_z^{\beta} \grad w\|_{L^p(Q_{r}^d(z_0))}
\lesssim \| f\|_{Y(p)}.
\]\end{lemma}

\begin{proof}
We deduce from Lemma \ref{L12} the estimate
\[
\|\rho^{\ell}\partial_t^k \partial_z^{\beta} \grad  w\|_{L^p(Q_{r}^d(z_0))} \lesssim \| \grad f\|_{L^p(\widehat Q_{r}^d(z_0))}.
\]
Let $\{Q_{r_i}^d(z_i)\}_{i\in I}$ be a finite cover of $\widehat Q_r^d(z_0)$ such that $r_i\sim r$ and $\sum_{i\in I} |Q_{r_i}^d(z_i)|\lesssim |\widehat Q_r^d(z_0)|$. Then
\begin{eqnarray*}
\|\grad f\|_{L^p(\widehat Q_{r}^d(z_0))} &\le& \sum_{i\in I} \| \grad f\|_{L^p(Q_{r_i}^d (z_{i}))}\\
&\lesssim &\frac1{r^2} \sum_{i\in I}|Q_{r_i}^d (z_{i})|^{\frac1p}
r_i^2|Q_{r_i}^d (z_{i})|^{-\frac1p}
\| \grad f\|_{L^p(Q_{r_i}^d (z_{i}))}\\
&\lesssim& \frac1{r^2} |\widehat Q_{r}^d(z_0)|^{\frac1p}  \|  f\|_{Y(p)}.
\end{eqnarray*}
Notice that $|\widehat Q_{r}^d (z_0)| =\frac34r^2|B_{2r}^d(z_0)| \lesssim r^2 |B_{r}^d(z_0)| = |Q_{r}^d(z_0)|$ by Lemma \ref{L8b}. This proves the lemma.
\end{proof}

In the case $(\ell,k,|\beta|) = (0,0,1)$, if $\sqrt{\rho(z_0)}\gg r$, the estimate of Lemma \ref{L13} is not strong enough to obtain a Carleson measure bound. The stronger estimate will be derived in the following lemma. Simultaneously, we establish a pointwise estimate on $\grad w$.

\begin{lemma}\label{L11a}
Let $w$ be a solution to the initial value problem \eqref{32a}, \eqref{32b}. Suppose that $\spt f\subset \widehat{Q}_{r}^d(z_0)$ for some $z_0\in \overline{B_1(0)}$ and $0<r\le \sqrt2$. Let $p> \max\left\{N+1,\frac1{\sigma+1}\right\}$. Then, for any $t\in [0,r^2]$ it holds
\[
|\grad w(t,z_0)|\lesssim \|f\|_{Y(p)}.
\]
Moreover, if $\sqrt{\rho(z_0)}\gg r$ and $p> \max\left\{N+2,\frac1{\sigma+1}\right\}$, it holds
\[
r\left(r+\sqrt{\rho(z_0)}\right)|Q_r^d(z_0)|^{-\frac1p}\|\grad^2 w\|_{L^p(Q_r^d(z_0))}\lesssim \|f\|_{Y(p)}.
\]
\end{lemma}

\begin{proof}
It is convenient to deduce the two statements from the following estimates on $w$ and $\grad w$, respectively: Let $\tilde f\in L^p$. Suppose that $\tilde w$ is a solution to \eqref{32a}, \eqref{32b} with $f$ replaced by $\tilde f$. If $p> N+1$, then
\begin{equation}
\label{35b}
|\tilde w(t,z)|\lesssim r^2 |Q_r^d(z)|^{-\frac1p}\|\tilde f\|_{L^p}
\end{equation}
for all $(t,z)\in\left[0,r^2\right]\times \overline{B_1(0)}$, and if $p> N+2$ and $\sqrt{\rho(z)}\gtrsim r$ then
\begin{equation}
\label{35c}
r\sqrt{\rho(z)}|\grad\tilde w(t,z)|\lesssim r^2 |Q_r^d(z)|^{-\frac1p}\|\tilde f\|_{L^p}
\end{equation}
for all $(t,z)\in[0,r^2]\times\overline{B_1(0)} $.

Indeed, differentiation of \eqref{32a} with respect to the $i$-th coordinate gives the equation
\[
\partial_t\partial_i w - \rho^{-\sigma} \div\left(\rho^{\sigma+1} \grad\partial_i w\right) = \partial_i f + z_i\laplace w + (\sigma +1)\left(\partial_i w + z\cdot \grad\partial_i w\right).
\]
Let $\tilde f$ be the right-hand side of this identity. Applying \eqref{35b} to $\tilde w = \partial_i w$ and summing over $i$ yields
\[
|\grad w(t,z)|\lesssim r^2 |Q_{r}^d(z)|^{-\frac1p} \left(\|\grad f\|_{L^p} + \|\grad w\|_{L^p}+ \|\grad^2 w\|_{L^p}\right),
\]
for all $(t,z)\in\left[0,r^2\right]\times \overline{B_1(0)}$. With the help of Lemma \ref{L12}, using that $f$ is supported in $\widehat Q^d_r(z_0)$ and choosing $z=z_0$, the latter turns into
\[
|\grad w(t,z_0)|\lesssim r^2 |Q_{r}^d(z_0)|^{-\frac1p} \| f\|_{W^{1,p}(\widehat Q_r^d(z_0))}.
\]
Arguing as in the proof of the previous lemma, the right-hand side is controlled by $\|f\|_{Y(p)}$, which shows the first statement of the lemma.

In a very similar way, we deduce from \eqref{35c} the estimate
\[
r\sqrt{\rho(z)} |\grad^2 w(t,z)|\lesssim |Q_r^d(z)|^{-\frac1p}|Q_r^d(z_0)|^{\frac1p}\|f\|_{Y(p)},
\]
for all $(t,z)\in [0,r^2]\times B_r^d(z_0)$. In view of Lemma \ref{L8c}, it is $\rho(z)\sim \rho(z_0)$ for all $z\in B_r^d(z_0)$ provided that $\sqrt{\rho(z_0)}\gg r$. Hence $|Q_r^d(z)|\sim |Q_r^d(z_0)|$ and integrating the previous estimate over $Q_r^d(z_0)$ yields the second statement of the lemma.

It remains thus to prove \eqref{35b} and \eqref{35c}. 
Both estimates can be established simultaneously. We use the heat kernel to represent the solution, cf.\ Lemma \ref{R4}, and estimate with H\"older's inequality
\[
|\partial_z^{\beta}\tilde w(t,z)|\le \| \partial_z^{\beta} G(t,\tacka,z,\tacka)\|_{L^{q}\left((0,r^2);L^{q}_{q\sigma}\right)} \|\tilde f\|_{L^p(\widehat Q_r^d(z_0))},
\]
for all $(t,z)\in \left[0,r^2\right]\times \overline{B_1(0)}$, and where $\frac1p+\frac1q=1$ as usual and $\beta\in\N_0^N$ such that $|\beta|\in\{0,1\}$. The estimates \eqref{35b} and \eqref{35c} follow thus from appropriate $L^q$ estimates on the heat kernel. First, recall the Gaussian estimate from Proposition \ref{P2b}
\begin{equation}
\label{33a}
|\partial_z^{\beta}G(t,t',z,z')|\lesssim \sqrt{\tau}^{-|\beta|} \left(\sqrt{\tau} +\sqrt{\rho(z)}\right)^{-|\beta|}|B_{\sqrt{\tau}}^d(z)|_{\sigma}^{-1} e^{-Cd(z,z')^2/\tau},
\end{equation}
where we have set $\tau=t-t'$ and we have used Lemma \ref{L8d} to replace the ball centered at $z'$ by the ball centered at $z$.
We let $\left\{B_{j\sqrt{\tau}}^d(z)\right\}_{1\le j\le J}$ be a cover finite cover of $\overline{B_1(0)}$ and estimate
\begin{eqnarray*}
\int e^{-qCd(z,z')^2/\tau}\, d\mu_{q\sigma}(z') &\lesssim&\sum_{j=1}^J\int_{B_{j\sqrt{\tau}}^d(z)\setminus B_{(j-1)\sqrt{\tau}}^d(z)} e^{-qCd(z,z')^2/\tau}\, d\mu_{q\sigma}(z')\\
&\lesssim &\sum_{j=1}^J e^{-qC(j-1)^2} |B_{j\sqrt{\tau}}^d(z)|_{q\sigma}.
\end{eqnarray*}
Invoking Lemma \ref{L8b}, we compute $|B_{j\sqrt{\tau}}^d(z)|_{q\sigma} \lesssim j^{\kappa}|B_{\sqrt{\tau}}^d(z)|_{q\sigma}\sim j^{\kappa}|B_{\sqrt{\tau}}^d(z)|^{1-q} |B_{\sqrt{\tau}}^d(z)|_{\sigma}^{q}   $  for some $\kappa>0$. Hence
\[
\int e^{-qCd(z,z')^2/\tau}\, d\mu_{q\sigma}(z')\lesssim \sum_{j\in\N} j^{\kappa} e^{-qC(j-1)^2}  |B_{\sqrt{\tau}}^d(z)|^{1-q} |B_{\sqrt{\tau}}^d(z)|_{\sigma}^{q} .
\]
Because the series is convergent, an integration of \eqref{33a} over $(0,r^2)\times B_1(0)$ yields
\[
\int_0^{r^2} \int |\partial_z^{\beta} G(t,t',z,z')|^q\, d\mu_{q\sigma}(z')dt'\lesssim \int_0^{r^2}  \sqrt{\tau}^{-q|\beta|}\left(\sqrt{\tau} + \sqrt{\rho(z)}\right)^{-q|\beta|}|B_{\sqrt{\tau}}^d(z)|^{1-q}\, d\tau.
\]
Once again, we distinguish the cases where $z$ is relatively close to the boundary and the case where it is not. In the first case, if $\sqrt{\rho(z)}\lesssim r$, we choose $\beta=0$ and estimate with the help of Lemma \ref{L8}
\[
\int_0^{r^2}  |B_{\sqrt{\tau}}^d(z)|^{1-q}\, d\tau 
\lesssim \int_0^{r^2}\tau^{(1-q)N}\,d\tau
\lesssim  r^{2(1-q)N +2} 
\lesssim  r^2 |B_{r}^d(z)|^{1-q},
\]
provided that $p>N+1$. In the second case, if $\sqrt{\rho(z)}\gtrsim r$, we have 
\begin{eqnarray*}
\lefteqn{\int_0^{r^2} \sqrt{\tau}^{-q|\beta|}\left(\sqrt{\tau} + \sqrt{\rho(z)}\right)^{-q|\beta|} |B_{\sqrt{\tau}}^d(z)|^{1-q}\, d\tau}\\
 &\lesssim& \sqrt{\rho(z)}^{-q|\beta| +(1-q)N}\int_0^{r^2}\sqrt{\tau}^{-q|\beta + (1-q)N}\,d\tau\\
&\lesssim & \sqrt{\rho(z)}^{-q|\beta| + (1-q)N} r^{-q|\beta| + (1-q)N +2} \\
&\lesssim & \left(r \sqrt{\rho(z)}\right)^{-|\beta|} r^2 |B_{r}^d(z)|^{1-q},
\end{eqnarray*}
provided that $(2-|\beta|)p >N+2$.
Therefore, in either case, if $p>N+1$ it holds
\[
\| G(t,\tacka,z,\tacka)\|_{L^q\left(\left(0,r^2\right); L^q_{q\sigma}\right)} \lesssim r^{\frac2q} |B_{r}^d(z)|^{\frac1q-1} \sim r^2 |Q_{r}^d(z)|^{-\frac1p},
\]
which implies \eqref{35b}. The estimate \eqref{35c} follows analogously with $|\beta|=1$ and $p>N+2$.
\end{proof}

We come now to the ``off-diagonal'' estimates. We define
\[
\widetilde Q_r^d(z_0) : = \left( \frac34r^2,r^2\right)\times B_{\frac{r}2}^d(z_0),
\]
 so that $\widetilde Q_r^d(z_0) \subset Q_r^d(z_0)\subset \widehat Q_r^d(z_0)$.

\begin{lemma}\label{L15}
Let $w$ be the  solution to the initial value problem \eqref{32a}, \eqref{32b}. Suppose that $\spt f\subset \left[0,r^2\right]\times \overline{B_1(0)}\setminus Q_{r}^d(z_0)$ for some $z_0\in \overline{B_1(0)}$ and $0<r\le \sqrt2$. Then for all $(t,z)\in \widetilde Q_r^d(z_0)$ and $p>\max\left\{1,\frac1{\sigma+1}\right\}$ it holds
\begin{eqnarray*}
\lefteqn{|\grad w(t,z)|+ r^2|\partial_t \grad w(t,z)| + r\left(r+\sqrt{\rho(z)}\right)|\grad^2 w(t,z)| + r^2 \rho(z)|\grad^3 w(t,z)|}\\
&\lesssim    &\|f\|_{Y(p)}.\hspace{25em}
\end{eqnarray*}
\end{lemma}

\begin{proof}
We start the proof with an elementary auxiliary estimate whose proof we postpone until later: We claim that there exists a constant $\tilde C>0$ such that
\begin{equation}\label{33b}
\sqrt{t-t'}^{-\theta} e^{- Cd(z,z')/(t-t')} \lesssim r^{-\theta} e^{-\tilde Cd(z,z')/r}
\end{equation}
for all $(t,z)\in \widetilde Q_r^d(z_0)$ and $(t',z')\in \left[0,r^2\right]\times\overline{B_1(0)}\setminus Q_r^d(z_0)$. In the following $\tilde C$ will denote a universal constant whose value may change from line to line.

For all $(t,z)\in \widetilde Q_r^d(z_0)$ it is $f(t,z)=0$ and thus
\begin{eqnarray*}
\lefteqn{\left|\partial_t^k\partial_z^{\beta} w(t,z) \right|}\\
&\lesssim&  \int_0^{r^2} \int \frac{\sqrt{\tau}^{-2k-|\beta|}\left(\sqrt{\tau}+\sqrt{\rho(z)}\right)^{-|\beta|}e^{-Cd(z,z')^2/\tau}}{|B_{\sqrt{\tau}}^d(z)|^{1/2}_{\sigma}|B_{\sqrt{\tau}}^d(z')|_{\sigma}^{1/2}} |f(t-\tau,z')|\, d\mu_{\sigma}(z')d \tau
 \end{eqnarray*}
by Duhamel's principle (Lemma \ref{R4}) and the Gaussian estimates  (Proposition \ref{P2b}). The estimate \eqref{33b} and the monotonicity of the function $s\mapsto \frac{s}{s+c}$ show that we can replace $\sqrt{\tau}$ by $r$ in the above estimate. Thanks to Lemma \ref{L8d}, we can moreover substitute balls centered at $z'$ by balls centered at $z$ and vice versa. We thus have
\begin{eqnarray*}
\lefteqn{ r^{2k+|\beta|-1} \left(r+\sqrt{\rho(z)}\right)^{|\beta|-1}\left|\partial_t^k\partial_z^{\beta} w(t,z) \right|}\\
&\lesssim&
 \int_0^{r^2} \int \frac{e^{-\tilde C d(z,z')/r}}{r|B_{r}^d(z)|_{\sigma}}\left(r+\sqrt{\rho(z')}\right)^{-1}|f(t',z')|\, d\mu_{\sigma}(z')d t'.
 \end{eqnarray*}
Let $\left\{B_r^d(z_i)\right\}_{1\le i\le n}$ be a family of balls covering $\overline{B_1(0)}$. Observing that  $d(z,z_i)\lesssim d(z,z') + r$ for any $z'\in B_r^d(z_i)$, we further estimate the expression on the right-hand side by
\[
\left(\sum_{i=1}^n e^{-\tilde C d(z,z_i)/r}\right)
 \sup_{\tilde z\in B} \frac{e^{-\tilde C d(z,\tilde z)/r}}{r|B_{r}^d(z)|_{\sigma}}
 \int_0^{r^2} \int_{B_r^d(\tilde z)} \left(r+\sqrt{\rho}\right)^{-1} |f|
\, d\mu_{\sigma}(z') d  t'.
\]
The sum  is bounded uniformly in $r$ and $z$. Moreover, thanks to Lemma \ref{L8d}, we can replace $z$ by $\tilde z$ in the multiplicative factor in front of the integral. Hence, to prove the statement of the lemma, we have to show that
\begin{equation}
\label{36a}
\frac{1}{r|B_{r}^d(\tilde z)|_{\sigma}}\|\rho^{\sigma}\left(r+\sqrt{\rho}\right)^{-1}f\|_{L^1\left(\left(0,r^2\right)\times B_r^d(\tilde z)\right)}\lesssim \|f\|_{Y(p)}.
\end{equation}
For that purpose, we cover the domain of integration  with a countable number of intrinsic cylinders. More precisely, with $r_{j}  = \sqrt{3/2}^{\, -j} r$ for $j\in\N_0$, we find a finite family of balls $\left\{B_{r_j}^d(z_{ij})\right\}_{1\le j\le n_j}$ with $z_{ij}\in B_r^d(\tilde z)$ such that
\begin{equation}
\label{37}
\sum_{i=1}^{n_j} |B_{r_j}^d(z_{ij})|_{\sigma} \lesssim |B_r^d(\tilde z)|_{\sigma}
\end{equation}
uniformly in $j$ and $r$. (The latter follows from a non-Euclidean version of Vitali's covering lemma, cf.\ \cite[Lemma 2.2.2]{Koch99}). Then $\left(0,r^2\right)\times B_r^d(\tilde z)\subset \cup_{j\in\N_0} \cup_{1\le i\le n_j} Q_{r_j}^d(z_{ij})$ and thus, via H\"older's inequality
\[
\|\rho^{\sigma}\left(r+\sqrt{\rho}\right)^{-1}f\|_{L^1\left(\left(0,r^2\right)\times B_r^d(\tilde z)\right)}
 \lesssim
 \sum_{j\in\N_0} \sum_{i=1}^{n_j}\|\rho^{\sigma}\left(r_j + \sqrt{\rho}\right)^{-1}\|_{L^q(Q_{r_j}^d(z_{ij}))} \|f\|_{L^p(Q^d_{r_j}(z_{ij}))},
\]
where $\frac1p +\frac1q=1$. It follows from Lemma \ref{L8c} that $(r_j+\sqrt{\rho})^{-1}\lesssim (r_j + \sqrt{\rho(z_{ij})})^{-1}$ in $B_{r_j}^d(z_{ij})$. Therefore, 
\begin{eqnarray*}
\|\rho^{\sigma}\left(r_j + \sqrt{\rho}\right)^{-1}\|_{L^q(Q_{r_j}^d(z_{ij}))}
&\lesssim&
 \left(r_{j} +\sqrt{\rho(z_{ij})}\right)^{-1} |Q_{r_j}^d(z_{ij})|_{q\sigma}^{\frac1q}\\
 & \sim& 
  \frac{r^2_j}{r_{j} +\sqrt{\rho(z_{ij})}} |B_{r_j}^d(z_{ij})|_{\sigma} |Q_{r_j}^d(z_{ij})|^{-\frac1p}.
\end{eqnarray*}
Notice that the measure $\mu_{q\sigma} $ is finite only  if $q\sigma>-1$, or equivalently, $p>\frac1{\sigma+1}$. Combining the previous two estimates and using \eqref{37} and the convergence of the geometric series yields \eqref{36a}.

We finally turn to the proof of \eqref{33b}. we first consider the case where $z'\in B_{r}^d(z_0)$. We then must have $t'\not\in \left( \frac{r^2}2,r^2\right)$, and thus $\tau =t-t'\ge \frac14r^2$. Then
\[
\sqrt{\tau}^{-\theta} e^{-Cd(z,z')^2/\tau} \le \sqrt{\tau}^{-\theta} \lesssim r^{-\theta} \lesssim r^{-\theta} e^{-Cd(z,z')/r},
\]
because $d(z,z') \lesssim d(z,z_0) + d(z_0,z') \le \frac32r$. In the complementary case, where $z'\not\in B_r^d(z_0)$, it holds that $\frac{r}2\le d(z,z')$ because $z\in B_{\frac{r}2}^d(z_0)$\footnote{Here we assume that the triangle inequality holds for $d$ for notational convenience. In fact, to guarantee a lower bound $d(z,z')\gtrsim r$ in this case, we would have to consider intrinsic balls of smaller radius in the definition of $\widetilde Q_r^d(z_0)$.}  Now, a small computation shows that the function $\tau\mapsto e^{-C d(z,z')^2 /\tau} \sqrt{\tau}^{-\theta}$ is increasing as long as $0<\tau \lesssim d(z,z')^2$. Since $\tau\le r^2\lesssim d(z,z')^2$, it thus holds that
\[
\sqrt{\tau}^{-\theta} e^{-Cd(z,z')^2/\tau} \lesssim r^{-\theta} e^{-Cd(z,z')^2/r^2} \le r^{-\theta} e^{-Cd(z,z')/r}.
\]
This proves \eqref{33b}.

\end{proof}

We are finally in the position to derive the Caleson-measure estimate for the initial value problem \eqref{32a}, \eqref{32b}.

\begin{prop}\label{P3}
Let $w$ be a solution to the initial value problem \eqref{32a}, \eqref{32b}. Suppose that $p>\max\left\{N+2,\frac1{\sigma+1}\right\}$. Then
\[
\|w\|_{X(p)}+ \|w\|_{\Lip} +   \lesssim \|f\|_{Y(p)} .
\]
\end{prop}
 
\begin{proof}
We start with the estimates for small times, i.e., we assume that $r\le \sqrt{2}$. There is no loss of generality to assume that $f(t,\tacka) = 0$  for $t\ge r^2$. 
We fix $(t,z_0)\in (0,r^2)\times \overline{B_1(0)}$ let $\eta $ denote a cut-off function such that $\eta=1$ in $Q_r^d(z_0)$ and $\eta=0$ outside $\widehat Q_{r}^d(z_0)$. We decompose $f=f_1+f_2$ with $f_1=\eta f$ and $f_2=(1-\eta)f$, so that $\spt f_1\subset \widehat  Q_r^d(z_0)$ and $\spt f_2\subset[0,r^2]\times B\setminus Q_r^d(z_0)$. Let $w_1$ and $w_2$ denote the solutions to the inhomogeneities $f_1$ and $f_2$, respectively. By the linearity of the equation, it holds $w=w_1+w_2$.
We deduce the control of $w_1$ from Lemmas \ref{L13} and \ref{L11a} and of $w_2 $ from Lemma \ref{L15}. For instance, for all $t\in\left[0,r^2\right]$,
\[
|\grad w(t,z_0)|\le |\grad w_1(t,z_0)| + |\grad w_2(t,z_0)|\lesssim \|f_1\|_{Y(p)} + \|f_2\|_{Y(p)}.
\]
The latter is controlled by $\|f\|_{Y(p)}$ because
\[
|\grad \eta|\lesssim r^{-1}\left(r+\sqrt{\rho(z_0)}\right)^{-1}.
\]
We thus have $|\grad w(t,z_0)|\lesssim \|f\|_{Y(p)}$ for all $(t,z_0)\in \left[0,r^2\right]\times \overline{B_1(0)}$. The right-hand side being independent of $r$ and $z_0$, this proves the desired Lipschitz bound for $w$ for small times $t\in[0,2]$. The Carleson measure estimates in this time interval are derived similarly.

It remains to derive the estimates for large times, $T\ge 1$. By superposition, in view of the above estimates and Proposition \ref{P2a} we can thus assume that $f\equiv0$ in $[0,1)\times \overline{B_1(0)}$. We let $\chi=\chi_{(T,T+1)}$ and decompose $f = f_1+f_2$ with $f_1 = \chi f$ and $f_2 = (1-\chi)f$. Let $w_1$ and $w_2$ denote the solutions to the inhomogeneities $f_1$ and $f_2$, respectively.

From Lemma \ref{L12}, we deduce that
\begin{equation}\label{40}
\|\partial_t \grad w_1\|_{L^p(Q(T))} + \|\grad^2 w_1\|_{L^p(Q(T))}  + \|\rho\grad^3 w_1\|_{L^p(Q(T))}
\lesssim 
\|\grad f_1\|_{L^p(Q(T))}.
\end{equation}
To gain control on the Lipschitz norm of $w_1$, we invoke the  Sobolev embedding theorem, which, if $p>N$ reads
\[
\|\zeta\|_{L^{\infty}(Q(T))} \lesssim  \|\zeta\|_{L^p(Q(T))} +  \|\partial_t \zeta\|_{L^p(Q(T))} + \|\grad \zeta\|_{L^p(Q(T))}.
\]
Choosing $\zeta = \grad w_1$ and using a Poincar\'e estimate in time with $\grad w_1(T)=0$, \eqref{40} yields
\[
\|\grad w_1\|_{L^{\infty}(Q(T))}\lesssim \|\grad f_1\|_{L^p(Q(T))}\le  \|f_1\|_{Y(p)}.
\]
In order to derive the analogous estimates for $w_2$, we use the exponential decay of the Gaussian in Proposition \ref{P2b} to estimate
\[
|\partial_t^k\partial_z^{\beta} w_2(t,z)|\lesssim \int_1^{T} \int e^{-\lambda_1(t-t')} |f_2(t',z')|\, d\mu_{\sigma}(z')dt',
\]
for all $(t,z)\in \left(T+\frac12,T\right)\times B_1(0)$. We employ H\"older's inequality and use that $\mu_{q\sigma}$ is a finite measure if its dual $p>\frac1{\sigma+1}$. Then
\[
|\partial_t^k\partial_z^{\beta} w_2(t,z)|\lesssim e^{-\lambda_1 t}\sum_{n=1}^{[T]} \left(\int_{n}^{n+1} e^{\lambda q t'}\, dt'\right)^{1/q}\|f_2\|_{L^p(Q(k))}.
\]
The sum is estimated by $e^{\lambda_1 t}$, and we obtain thus control in the  Lipschitz norm and the $L^p$ norm in the cylinder $\left(T+\frac12,T\right)\times B_1(0)$. Arguing as before, we obtain control of $w=w_1+w_2$ in that cylinder because $\|f_i\|_{Y(p)}\le \|f\|_{Y(p)}$, and the bound is independent of the decomposition. The desired statement easily follows. 
\end{proof}

%

\newpage

\section{The nonlinear problem}\label{S:6}

In this section, we turn to the proof of Theorem \ref{T1}, that is, we will show well-posedness of the 
 nonlinear perturbation equation \eqref{13} in a neighborhood of the constant $w_*\equiv 0$ and establish analytic dependence of solutions on the initial data in that neighborhood. For convenience, we rewrite \eqref{13} as:
\begin{equation}\label{42a}
\partial_t w -\rho^{-\sigma}\div\left(\rho^{\sigma+1} \grad w\right) = \beta \rho F(w,\grad w)-\rho^{-\sigma }\div\left(\rho^{\sigma+1} z F(w,\grad w)\right)   ,
\end{equation}
where $\beta= N+2\sigma +1$ and
\[
F(q,p)  =  \frac{|p|^2}{1 + q + z\cdot p}
\]
for $(q,p)\in \R\times \R^N$. For abbreviation, we will sometimes denote the right-hand side of \eqref{42a} by $f(w)$. Well-posedness for this equation will be derived by combining the linear theory of the previous section with a fixed point argument, that relies on the observation  that the nonlinearity is essentially quadratic in a neighborhood of the constant $w_*\equiv0$. However, a direct application of the linear theory presents us with a technical problem: The (semi-)norm on the solution space $X(p)$ in Section \ref{S:5} contains only the homogeneous part of the Lipschitz norm, whereas the nonlinearity $F$ above demands control of the full $C^{0,1}$ norm to prevent the denominator from degenerating. We will face a second problem later in Section \ref{S7}: Theorem \ref{T1} shows that the nonlinear problem \eqref{42a} generates a $C^{0,1}$ semi-flow  {\em locally} in a neighborhood of the constant $w_*\equiv0$. However, in order to fit into the dynamical systems framework needed to establish  invariant manifold theorems, it is necessary to extend the local flow into a global one.

We will tackle both problems simultaneously by considering a truncated version of \eqref{42a} first. 
 More precisely, we introduce a smooth function that cuts the nonlinear terms off at points where the solution or its gradient is too large. The resulting equation is linear at such points. If, however, a solution is globally sufficiently small in $C^{0,1}$, then the cut-off is inactive and the truncated equation coincides with the original one. The latter is guaranteed by a smallness condition on the initial datum thanks to the a priori estimates from Theorem \ref{T6} and a (weak) comparison principle, cf.\ Lemmas \ref{L15bis} and \ref{L15a} below.
 
With regard to the invariant manifold theorem that will be derived in Section \ref{S7}, working with the truncated equation instead of \eqref{42a} has one more advantage: The truncated equation generates not only a semi-flow in $C^{0,1}$, but also in $L^2_{\sigma}$. We prefer to exploit the latter as it closely connects to the spectral analysis conducted in \cite{Seis14}. Moreover, many of the concepts used in the derivation of the theorem are easier explained and formulated in the Hilbert space setting. We will eventually transfer the results from $L^2_{\sigma}$ to $C^{0,1}$ with the help of smoothing estimates. 
 
This section is organized as follows: In Subsection \ref{S:6.1}, we introduce and study the truncated problem. We prove well-posedness in $L^2_{\sigma}$ and establish a priori estimates in smoother spaces. Subsection \ref{S:6.2} contains the proof of well-posedness of the full nonlinear perturbation equation, Theorem \ref{T1}.

\subsection{The truncated problem}\label{S:6.1}

We let $\hat \eta: [0,\infty)\to[0,1] $ be a smooth cut-off function that is supported on $[0,2)$ and constantly one on $[0,1]$. Let $0<\eps,\,\delta <1$ be fixed such that $\sqrt2\left(\eps+\delta \right)<1$. We then set for $(q,p)\in \R\times \R^N$
\[
\eta_{\eps,\delta}(q,p)  =  \hat \eta\left( \frac{q^2}{\delta^2}\right)\hat \eta\left(\frac{|p|^2}{\eps^2}\right),
\]
and with that $F_{\eps,\delta} = \eta_{\eps,\delta} F$. The {\em truncated equation} is
\begin{equation}
\label{43}
\partial_t w -\rho^{-\sigma}\div\left(\rho^{\sigma+1} \grad w\right) =  \beta\rho F_{\eps,\delta}(w,\grad w)-\rho^{-\sigma }\div\left(\rho^{\sigma+1} z F_{\eps,\delta}(w,\grad w)\right)   .
\end{equation}
For abbreviation, we denote the right-hand side by $f_{\eps,\delta}(w)$. It is clear that this equation is equivalent to \eqref{42a} as long as $ |w|\le\delta$ and $ |\grad w| \le  \eps$.

Notice that in \eqref{43}, the cut-off $\eta_{\eps,\delta}$ acts on both $w$ and $\grad w$, so that the nonlinearity $F_{\eps,\delta}(w,\grad w)$ is globally defined. Conveniently, the cut-off not only allows for a global $C^{0,1}$ semi-flow as a consequence of the estimates derived in Theorem \ref{T6}. It also has a regularizing effect as now a Hilbert space theory becomes applicable. Via a fixed point argument, it yields the existence of a global $L^2_{\sigma}$ semi-flow. It seems hence natural to first construct invariant manifolds for the flow for the truncated equation in the Hilbert space setting and then carry the results over to the original equation for the $C^{0,1}$ semi-flow. This will be done in Section \ref{S7} below.

We finally remark that for the derivation of Theorem \ref{T1}, the study of the truncated equation is redundant and Theorem \ref{T1} can be established directly using Carleson measure estimates and a fixed point argument exploiting the quadratic behavior of $F(q,p)$ in $p$ near the origin. For this, however, in addition to the Carleson measures in Theorem \ref{T6}, which are formulated for the gradient of $w$, we would have to establish the analogous bounds for $w$ itself to gain control over the full $C^{0,1}$ norm in terms of the initial datum. In this regard, the situation considered in this paper is slightly more complex than the one studied in 
 \cite{Kienzler14}.

We now state and prove

\begin{theorem}\label{T4}
There exists $\eps_0>0$ such that for every $0<\eps\le \eps_0$ and $0<\delta<1$ with $\sqrt2\left(\eps+\delta\right)<1$ the following holds:
For every $g\in L^2_{\sigma}$, there exists a unique solution $w$ to the truncated equation \eqref{43} with initial datum $g$. Moreover, $w\in L^{\infty}((0,\infty);L^2_{\sigma})\cap L^2((0,\infty);\dot H^1_{\sigma+1})$.

\end{theorem}

\begin{proof}
For every $\tilde F\in L^2(L^2_{\sigma+1})$, $g\in L^2_{\sigma}$, and $T>0$, the initial value problem for the linear equation
\begin{equation}\label{43c}
\partial_t \tilde w -  \rho^{-\sigma} \div\left(\rho^{\sigma+1}\grad \tilde w\right) = \beta \rho \tilde F - \rho^{-\sigma}\div\left(\rho^{\sigma+1} z\tilde F\right)
\end{equation}
has a unique solution in the time interval $(0,T)$. It satisfies the  a priori estimate
\begin{equation}
\label{43b}
\|\tilde w\|_{L^{\infty}(L^2_{\sigma})} + \| \grad\tilde  w\|_{L^2(L^2_{\sigma +1})} \le C_T\left(\| \tilde F\|_{L^2(L^2_{\sigma +1})}  + \|g\|_{L^2_{\sigma}}\right) ,
\end{equation}
where $C_T$ is a generic constant that depends on time $T$.  (Here and in the following discussion, the norm dependences  on the time interval $(0,T)$ are suppressed.)
Existence follows from standard Hilbert space methods and \eqref{43b} can be derived by testing the equation with $\rho^{\sigma }\tilde w$ --- we omit the details.

To prove well-posedness for the nonlinear equation \eqref{43} for any initial datum in $L^2_{\sigma}$, we invoke a   fixed point argument. Given a function $w\in L^{\infty}(L^2_{\sigma})\cap L^2(\dot H^1_{\sigma+1})$, we denote by $\tilde w = \tilde w(w,g)$ the solution to \eqref{43c} with $\tilde F=  F_{\eps,\delta}(w,\grad w)$  and with initial value $g$. Because $|F_{\eps,\delta}(q,p)|\lesssim \eps |p|$, we have $\| \tilde F\|_{L^2(L^2_{\sigma +1})}  \lesssim \eps \|\grad w\|_{L^2(L^2_{\sigma+1})}$, and thus by \eqref{43b}
\[
\|\tilde w\|_{L^{\infty}(L^2_{\sigma})} + \|\grad \tilde w\|_{L^2(L^2_{\sigma +1})} \le C_T\eps \|\grad w\|_{L^2(L^2_{\sigma+1})} + C_T\|g\|_{L^2_{\sigma}}
\]
for some universal (but time-dependent) constant $C_T>0$. Moreover, because $|F_{\eps,\delta}(q_1,p_1)- F_{\eps,\delta} (q_2,p_2)| \lesssim \eps\left(|q_1-q_2| + |p_1-p_2|\right)$ --- as the reader may straightforwardly compute --- applying \eqref{43b} again, we obtain for two solutions $\tilde w_1 = \tilde w(w_1,g)$ and $\tilde w_2 = \tilde w(w_2,g)$ that
\begin{eqnarray*}
\lefteqn{\|\tilde w_2-\tilde w_1\|_{L^{\infty}(L^2_{\sigma})} + \|\grad \tilde w_2-\grad\tilde w_1 \|_{L^2(L^2_{\sigma +1})} }\\
&\le& C_T \eps \left( \| w_2- w_1\|_{L^{\infty}(L^2_{\sigma})} + \|\grad  w_2-\grad w_1 \|_{L^2(L^2_{\sigma +1})}\right).
\end{eqnarray*}
Hence, for $\eps $ sufficiently small, the solution map $w\mapsto \tilde w(w,g)$  is a contraction in $L^{\infty}(L^2_{\sigma})\cap L^2(\dot H^1_{\sigma+1})$, and thus, for such $\eps$, there exists a unique local solution in $L^{\infty}(L^2_{\sigma})\cap L^2(\dot H^1_{\sigma+1})$ to  \eqref{43} with initial datum $g\in L^2_{\sigma}$.

To show the existence of a global solution, it suffices to derive a bound on solutions of \eqref{43} that is uniform in $T$. For this purpose, we test the equation with $w-c$ where $c\ge \sqrt2\delta$ is a fixed constant. Then
\begin{eqnarray*}
\lefteqn{
\frac{d}{dt}\frac12 \int (w-c)^2\, d\mu_{\sigma} + \int |\grad w|^2\, d\mu_{\sigma+1} }\\
&=&  \beta \int (w-c) F_{\eps,\delta}(w,\grad w) \, d\mu_{\sigma+1} + \int ( z\cdot \grad w )F_{\eps,\delta}(w,\grad w)\, d\mu_{\sigma +1}.
\end{eqnarray*}
Since $0\le F_{\eps,\delta}(w,\grad w)\lesssim \eps |\grad w|$, the second term on the right can be easily absorbed into the left-hand side provided that $\eps$ is chosen sufficiently small. Moreover, because  $|w|\le \sqrt2\delta$ in the support of $F_{\eps,\delta}(w,\grad w)$, the first term on the right-hand side is non-positive:
\[
\int (w-c) F_{\eps,\delta}(w,\grad w) \, d\mu_{\sigma+1} \le (\delta - c) \int F_{\eps,\delta}(w,\grad w) \, d\mu_{\sigma+1}\le 0.
\]
Thus
\[
\frac{d}{dt} \int (w-c)^2\, d\mu_{\sigma} + \frac1C \int |\grad w|^2\, d\mu_{\sigma+1}\le 0,
\]
for some universal constant $C>0$. Integration yields the desired uniform bound. This concludes the proof of Theorem \ref{T4}.

\end{proof}

Our first higher regularity estimate is for Lipschitz initial data.

\begin{lemma}\label{L15bis}
There exists $\eps_0>0$ such that for every $0<\eps\le \eps_0$ and $0<\delta<1$ with $\sqrt2\left(\eps+\delta\right)<1$ the following holds:
If $w$ is the solution to \eqref{43} with initial datum $g\in C^{0,1}$, then $w$ is smooth and, for any $p\in[1,\infty)$ it holds
\[
\|w\|_{X(p)} + \|w\|_{\Lip}\lesssim \|g\|_{\Lip}.
\]
\end{lemma}

\begin{proof}
It is enough to prove the estimate for $p>\max\left\{N+2,\frac1{\sigma+1}\right\}$. The extension to the full range $p\in[1,\infty)$ then follows by H\"older's inequality. If $\eps_0$, $\eps$, and $\delta$ are as in the assumption of Theorem \ref{T6}, then we have the estimate
\begin{equation}
\label{43f}
\|w\|_{X(p)} +\|w\|_{\Lip}\lesssim \|f_{\eps,\delta}(w)\|_{Y(p)} + \|g\|_{\Lip}.
\end{equation}
Therefore, the statement of the lemma follows provided that
\begin{equation}
\label{43g}
\|f_{\eps,\delta}(w)\|_{Y(p)}\lesssim \eps\left(  \|w\|_{X(p)} + \|w\|_{\Lip}\right),
\end{equation}
and upon choosing $\eps_0$ (and thus $\eps$) smaller if necessary.

We start noticing that $|f_{\eps,\delta}|\lesssim \rho| \grad F|+ \rho|F||\grad w| + \eps^{-1}\rho|F| |\grad^2 w|+ |F|$ and $|\grad f_{\eps,\delta}|\lesssim |\rho \grad^2 F| + \eps^{-1}\rho|\grad F| + |\grad F| + \eps^{-2}\rho|F| + \eps^{-1}|F|$ in the support of the cut-off function. (The computations leading to these and the following estimates are straightforward but tedious. We omit the details.) Hence, by chain rule and the definition of $F$,
\[
|f_{\eps,\delta} |\lesssim \eps \rho |\grad^2 w| +  \eps |\grad w|
\]
and 
\[
|\grad f_{\eps,\delta} |\lesssim \eps |\grad w| + \eps |\grad^2 w| + \rho|\grad^2 w|^2 + \eps \rho|\grad^3 w|
\]
in the support of $\eta_{\eps,\delta}$. In the estimates of $f_{\eps,\delta}$ and $\grad f_{\eps,\delta}$, the terms that deserve a special consideration are those which have second order derivatives, i.e., $\eps\rho|\grad^2 w|$ and $\eps|\grad^2 w|^2$. We restrict our attention to the estimates for the time interval $(0,1)$. First, an application of Lemma \ref{L8c} yields
\[
\frac{r}{r+\sqrt{\rho(z)}} \|\rho \grad^2 w\|_{L^p(Q_r^d(z))} \lesssim r\left(r+\sqrt{\rho(z)}\right)\|\grad^2 w\|_{L^p(Q_r^d(z))}.
\]
For the quadratic term, we use the interpolation inequality $\|\grad\zeta\|_{L^{2p}_p}^2\lesssim \|\zeta\|_{L^{\infty}}\|\grad^2 \zeta\|_{L^p_p}$, cf.\ \cite[Proposition 2.18]{Kienzler13}, and obtain
\[
\|\rho |\grad^2 w|^2\|_{L^p(Q_r^d(z)\cap \spt \eta_{\eps,\delta})} \lesssim \|\grad w\|_{L^{\infty}(\spt \eta_{\eps,\delta} )} \|\rho \grad^3w\|_{L^p(Q_r^d(z)\cap \spt \eta_{\eps,\delta})}.
\]
Then \eqref{43g} follows because $\|\grad w\|_{L^{\infty}}\lesssim\eps$ in the support of $\eta_{\eps,\delta}$.

Since $w$ is Lipschitz, \eqref{43} can be seen as the parabolic equation
\[
\partial_t w - \rho^{-\sigma} \div\left(\rho^{\sigma+1} (\I - A)\grad w\right) = \beta \rho F_{\eps,\delta}(w,\grad w),
\]
where $A$ given by 
\[
A = \eta_{\eps,\delta}(w,\grad w)\frac{z\otimes \grad w} {1+ w + z\cdot \grad w}
\]
is bounded and $\I- A$ is elliptic. Hence,  higher regularity can be obtained from regularity theory for parabolic equations with rough coefficients, see, for instance, \cite[Theorem 5.6.1]{Koch99} for a theory in the degenerate setting.
\end{proof}

We finally establish the following weak comparison principle for the nonlinear equation \eqref{43}.
\begin{lemma}\label{L15a}
Let $\eps_0,\, \eps$ and $\delta$ be as in Lemma \ref{L15bis}.
Let $w$ be the solution to  \eqref{43} with initial datum $g\in C^{0,1}$. Suppose that $a\le g(z)\le b$ for some $a,b\in \R$ and all $z\in B_1(0)$. Then
\[
a\le w(t,z)\le b\quad\mbox{for all }(t,z)\in(0,\infty)\times  B_1(0).
\]
In particular,
\[
\|w\|_{L^{\infty}}\le \|g\|_{L^{\infty}}.
\]
\end{lemma}

\begin{proof}By Lemma \ref{L15bis}, $w$ is smooth, so that we can give a pointwise argument.
We only prove the lower bound, the upper bound is established analogously. Let $\gamma>0$ be arbitrarily given. We will show that $a-\gamma <w(t,z)$ for all $t>0$ and $z\in B_1(0)$. Since $\gamma>0$ was arbitrary, the result follows. We argue by contradiction. Suppose that there exists $t_*>0$ and $z_*\in\overline{ B_1(0)}$ such that $a-\gamma =w(t_*,z_*)$ and $a-\gamma <w(t,z)$ for all $t<t_*$ and $z\in\overline{ B_1(0)}$. In the case where $z_*$ lies in the interior of the domain, $z_*\in B_1(0)$, then $\partial_t w(t_*,z_*) < 0$, $\grad w(t_*,z_*)=0$ and $\grad^2 w(t_*,z_*)$ is positive definite. It immediately follows that $\left(\partial_t w - \rho^{-\sigma }\div\left(\rho^{\sigma +1}\grad w\right)\right) (t_*,z_*)<0$ while $f_{\eps,\delta}(w,)(t_*,z_*)=0$ --- a contradiction. Otherwise, if $z_*\in \partial B_1(0)$, then $\partial_t w(t_*,z_*)<0$ and $z_* \cdot \grad w(t_*,z_*)\le 0$. We deduce that 
\[
\left(\partial_t w -\rho^{-\sigma}\div\left(\rho^{\sigma+1} \grad w\right) \right)(t_*,z_*)  = \partial_t w(t_*,z_*) + (\sigma+1) z_* \cdot \grad w(t_*,z_*) <0,
\]
which contradicts $f(w)(t_*,z_*) = (\sigma +1)F(w,\grad w)(t_*,z_*)\ge 0$.
\end{proof}

\subsection{The full problem. Proof of Theorem \ref{T1}}\label{S:6.2}

It is convenient to make the following definition: For $\eps,\, \delta>0$, we consider
\[
\B_{\eps,\delta} : = \left\{w\in C^{0,1}\!:\: \|w\|_{L^{\infty}}\le \delta,\, \|w\|_{\Lip}\le\eps\right\}.
\]

We are finally turn to the

\begin{proof}[Proof of Theorem \ref{T1}] 
Existence of a Lipschitz solution for sufficiently small initial data follows immediately from the theory for the truncated equation \eqref{43}: Theorem \ref{T4} provides us with a unique solution for the truncated equation, which is also a solution to the original equation thanks to Lemmas \ref{L15bis} and \ref{L15a}. It is clear that this solution is unique by the uniqueness for \eqref{43}.

We shall show that solutions depend analytically on the initial datum. For this purpose, we restrict our attention to finite time intervals $(0,T)$, but suppress the dependence on $T$ in the following. We first notice that both
\[
\B_{\eps,\delta} \ni w\mapsto F(w,\grad w) \in L^2(L^2_{\sigma+1}),
\]
and
\[
X(p)\cap \B_{\eps,\delta} \ni w\mapsto f(w,)\in Y(p)
\]
are analytic mappings.  Moreover, because of the elementary estimates $\|\zeta\|_{L^{\infty}} \lesssim \|\zeta\|_{L^2_{\sigma}} + \|\zeta\|_{\Lip}$ and $\|\zeta\|_{L^2_{\sigma}}\lesssim \|\zeta\|_{L^{\infty}}$, we deduce from Lemma \ref{L6} and Theorem \ref{T6} that
\[
\|w\|_{C^{0,1}}\lesssim \|F\|_{L^2(L^2_{\sigma+1})} + \|f\|_{Y(p)} + \|g\|_{C^{0,1}}.
\]
Therefore, if $\tilde w = \tilde w(w,g)$ denotes the solution to the linear perturbation equation with $f = f(w)$ for some $w\in X(p)\cap \B_{\eps,\delta}$ and initial datum $g\in \B_{\eps,\delta}$, then $\tilde w: X(p)\cap \B_{\eps,\delta} \times \B_{\eps,\delta} \to X(p)\cap C^{0,1}$ is analytic. Hence, also $\Xi: X(p)\cap \B_{\eps,\delta} \times \B_{\eps,\delta} \to X(p)\cap C^{0,1}$ with $\Xi(w,g) = w-\tilde w(w,g)$ is analytic. It moreover satisfies $\Xi(0,0) = 0$ and $D_w\Xi(0,0) = \I$, and thus, the analytic implicit function theorem shows (upon decreasing $\delta$ and $\eps$ if necessary) that the unique solution $w=w(g)$ of $\Xi(w,g)=0$ depends analytically on $g$. However, by the definition of $\Xi$, $w$ is the unique solution to the nonlinear perturbation equation \eqref{42a} with initial datum $g$.

The proof \eqref{7az} proceeds along the same lines as Kienzler's proof of the analogous statement in \cite[Theorem 4.1]{Kienzler14}. Only instead of distinguishing between ``horizontal'' and ``vertical'' directions, we should argue as in the proof of Lemma \ref{L8a}, that is, we rather distinguish between ``angular'' and ``radial'' directions.
\end{proof}

\newpage

\section{Invariant manifolds}\label{S7}

In this section, we turn to the proof of Theorem \ref{T5}.

The operator obtained from linearizing the truncated equation in the Hilbert space $L^2_{\sigma}$ around the constant function $w_*\equiv0$ differs from the linear operator studied in \cite{Seis14} only in the underlying Hilbert space (and a multiplicative factor, that we neglect in the following discussion): As in \cite{Seis14}, the action of the linear operator $\L$ on a smooth function $w$ is given by 
\begin{equation}
\label{43d}
\int \grad\varphi\cdot \grad w\, d\mu_{\sigma+1} = \int \varphi\L w\, d\mu_{\sigma}\quad\mbox{for all }\varphi\in C^{\infty}(\overline{B_1(0)}),
\end{equation}
i.e., $\L w = -\rho\laplace w +(\sigma+1)z\cdot \grad w$ for smooth $w$. The role of boundary conditions for this elliptic problem is discussed in Remarks \ref{R1} and \ref{R1a} in Section \ref{S:5} above.

In the present paper, we choose the Lebesgue space $L^2_{\sigma}$ as the underlying Hilbert space, while in \cite{Seis14}, the author diagonalized the operator in the Sobolev space $\dot H^1_{\sigma +1}/ \R$.  To distinguish these two operators, we will in the following denote by $\Ha$ the operator on $\dot H^1_{\sigma+1}/\R$. 

We shall be more specific with regard to the underlying functional analytical concepts. The operator $\left.\L\right|_{C^{\infty}(\overline{B_1(0)})}$ defined via \eqref{43d} on $C^{\infty}(\overline{B_1(0)})$ is a densely defined nonnegative symmetric operator both on $L^2_{\sigma}$ and on $\dot H^1_{\sigma+1}/\R$ (cf.\ \cite[Appendix]{Seis14}). Hence, $\left.\L\right|_{C^{\infty}(\overline{B_1(0)})}$ is closable in both Hilbert spaces and we denote its closure in $L^2_{\sigma}$ by $\L$, and its closure in $\dot H_{\sigma+1}^1/\R$ by $\Ha$. Since both operators are defined through the closure of a quadratic form, $\L$ and $\Ha$ are Friedrichs self-adjoint extensions (cf.\ \cite[Chapter X.3]{ReedSimon2}) of $\left.\L\right|_{C^{\infty}(\overline{B_1(0)})}$ in $L^2_{\sigma}$  and $\dot H_{\sigma+1}^1/\R$, respectively.

Comparing functions in $\dot H^1_{\sigma+1}/\R$ and $L^2_{\sigma}$ reveals some information on the spectrum. 
In either case, eigenfunction are smooth and lie in any Hilbert space $ H^k_{\sigma, \dots, \sigma+k}$. Hence, even though $\L$ is set in a larger space, we expect that $0$ is the only new eigenvalue in the discrete spectrum, i.e.,
\[
\Sigma(\L) = \Sigma(\Ha)\cup\{0\}.
\]
The emergence of the new zero eigenvalue stems from the fact that $L^2_{\sigma}$ contains the constant (eigen-)functions, which were modded out in the quotient space $\dot H^1_{\sigma+1}/\R$. This observation is the content of the proposition that follows. The computational results are taken from \cite{Seis14} and are appropriately rescaled.

\begin{prop}[The spectrum of the linearized operator]\label{P2}
The spectrum of $\L$ consists entirely of isolated eigenvalues with finite multiplicity. They are given by
\[
\lambda_{\ell k} = (\sigma+1)(\ell + 2k) + k(2k+2\ell+N-2),
\]
where $(\ell,k)\in \N_0\times\N_0$ if $N\ge 2$ and $(\ell,k)\in \{0,1\}\times\N_0$ if $N=1$. The corresponding eigenfunctions are polynomials of the form
\begin{eqnarray*}
\psi_{010}(z)&=&1,\\
\psi_{\ell nk}(z) &= &F\left(-k,\sigma+\ell +\frac{N}2+k; \ell+\frac{N}2; |z|^2\right)Y_{\ell n}\left(\frac{z}{|z|}\right)|z|^{\ell},
\end{eqnarray*}
where $n\in \{1,\dots,N_{\ell}\}$ with $N_{\ell}=1$ if $\ell=0$ or $\ell=N=1$ and $N_{\ell}=\frac{(N+\ell-3)!(N+2\ell-2)}{\ell!(N-2)!}$ else. Here, $F(a,b;c;z)$ is a hypergeometric function and in the case $N\ge 2$, $Y_{\ell n}$ is a spherical harmonic corresponding to the eigenvalue $\ell(\ell+N-2)$ of $\laplace_{\S^{N-1}}$ with multiplicity $N_{\ell}$. Otherwise, if $N=1$, it is $Y_{\ell 1}(\pm 1)=(\pm 1)^{\ell}$.
\end{prop}

The literature on hypergeometric functions and spherical harmonics is vast, see, e.g., \cite{Rainville1,Rainville2,BealsWong10,Groemer96}.

Hypergeometric functions $F(a,b;c;z)$ can be defined as the power series
\[
F(a,b;c;x) = 1 + \sum_{j=1}^{\infty} \frac{(a)_j(b)_j}{(c)_j j!} z^j,
\]
for any $a,\, b,\, c,\,  x\in \R$ and $c$ is not a non-positive integer. The definition involves the extended factorials or Pochhammer symbols
\[
(s)_j = s(s+1)\dots (s+j-1),\quad\mbox{for }j\in \N,\, s\in\R.
\]
The power series is convergent if $|x|<1$. In the situation a hand, hypergeometric functions come into play in the study of the radial part of the eigenvalue equation. The latter is a second order Fuchsian ODE with three regular singular points that can be transformed into the hypergeometric differential equation. In \cite{Seis14}, the author solved this equation and selected those solutions as eigenfunctions that lie in the domain of the linear operator.

Plugging the value $a=-k$  into the definition of the hypergeometric functions, we see that in Proposition \ref{P2} the hypergeometric functions reduce to polynomials of degree $k$ in $x=|z|^2$. Recalling that spherical harmonics $Y_{\ell n}$ are homogeneous harmonic polynomials of degree $\ell$ (defined on the unit sphere), we thus observe that the eigenfunction $\psi_{\ell nk}$ is  a polynomial of degree $\ell +2k$.

At this point, we shall briefly compare our spectrum with that found by Denzler and  McCann in \cite{DenzlerMcCann05} for the linearized fast diffusion equation. Since the solutions to the fast diffusion equation are positive instantaneously, the linearized operator is defined on all of $\R^N$, and thus, due to spatially algebraic decay of solutions, the operator has non-compact resolvents. It follows that  the spectrum is non-discrete. To be more specific, Denzler and  McCann found the same eigenvalues $\lambda_{\ell k}$ for values of $m= \frac{\sigma+2}{\sigma+1}<1$ close to one. These eigenvalues dissolve into continuous spectrum if $m $ becomes small. At the threshold $m=1$ between fast diffusion and porous medium dynamics, after rescaling with $\sigma+1$, we recover the spectrum of the Ornstein--Uhlenbeck operator $-\laplace + z\cdot \grad$.

\begin{proof}[Proof of Proposition \ref{P2}] The discreteness of the spectrum follows from the fact that $\L$ has a compact resolvent, see e.g., \cite[Ch.\ III, \textsection 6.8]{Kato}. The latter is a consequence of standard Hilbert space methods, cf.\ \cite[Appendix]{Seis14}: For every $f\in L^2_{\sigma}$, there exists a unique $w$ satisfying $\L w + w=f$. Via the energy estimate
$\|\grad w\|_{L^2_{\sigma+1}} + \|w\|_{L^2_{\sigma}} \lesssim \|f\|_{L^2_{\sigma}}$ and the compact embedding $H^1_{\sigma, \sigma+1}\subset L^2_{\sigma}$, we deduce that $(\L +1)^{-1}$ is well-defined and compact.

The eigenfunctions of $\Ha$ computed in \cite{Seis14} are regular and thus $\Sigma(\Ha)\subset \Sigma(\L)$. On the other hand, by the energy identity for the eigenvalue equation $\|\grad w\|_{L^2_{\sigma+1}} = \lambda \|w\|_{L^2_{\sigma}}$, every nonconstant eigenfunction of $\L$ is also an eigenfunction of $\Ha$. The statement of the proposition is thus proved.
\end{proof}

We can deduce the long-time limit of solutions to the perturbation equation \eqref{42a} from the known limiting behavior of  solutions in the original variables. Since changes of mass are allowed after the change of variables, we have to allow for perturbations of arbitrary mass. Solutions to the self-similar pressure equation $v(t,x)$ converge to a Barenblatt profile of some radius $r$ which is determined by the initial condition,
\[
v(t,x)\approx \frac12(r-|x|^2)_+\quad\mbox{as }t\gg1,
\]
cf.\ \eqref{1b}. By the change of coordinates introduced in Section \ref{S:main} above and made rigorous in Section \ref{S:8} below, it is $v(t,x) = \frac12(1-|x|)^2 + w(t,z) +\frac12 w(t,z)^2$ in the positivity set of $v$. Hence, considering the long-time limit and using that $\|w(t)\|_{L^{\infty}}\le \delta\le 1$, we deduce that
\[
w(t,z) \to \sqrt{r} -1\quad\mbox{as }t\gg1.
\]

In the following, we will study the higher order asymptotics of this limit via invariant manifolds. As outlined in the introduction to the previous section, we will first construct invariant manifolds for the truncated equation \eqref{43} and then carry the results over to the original equation with the help of smoothing estimates.

The formulation of the invariant manifold theorem requires some preparations. Theorem \ref{T4} guarantees the existence of a semi-flow
\[
S^t_{\eps,\delta} : L^2_{\sigma} \to L^2_{\sigma}
\]
defined by $S_{\eps,\delta}^t(g) = w(t)$,  if $w$ denotes the solution to the truncated equation \eqref{43} with initial datum $g$. The map $[0,\infty)\times L^2_{\sigma} \ni (t,g) \mapsto S_{\eps,\delta}^t (g)\in L^2_{\sigma}$ is continuous. 
Following Koch's (among others) approach to invariant manifolds based on discrete semi-flows, see \cite{Koch97}, it is enough to study the corresponding time-one maps. That is, we consider the mapping $S = S^1_{\eps,\delta} :L^2_{\sigma}\to L^2_{\sigma}$. By the construction in Theorem \ref{T4}, it is clear that $S$ is Lipschitz,
\begin{equation}
\label{43e}
\| S(g_1) - S(g_2)\|_{L^2_{\sigma}}\lesssim \|g_1-g_2\|_{L^2_{\sigma}},
\end{equation}
for any two $g_1,g_2\in L^2_{\sigma}$. Moreover, the following smoothing property holds:

\begin{lemma}\label{L15b}
There exists $\eps_0>0$ such that for every $0<\eps\le \eps_0$ and $0<\delta<1$ with $\sqrt2\left(\eps+\delta\right)<1$ the following holds: For any $g_1,\, g_2\in L^2_{\sigma}\cap \B_{\eps_0,\delta}$, it holds that
\[
\|S(g_1) - S(g_2)\|_{C^{0,1}}\lesssim \|g_1-g_2\|_{L^2_{\sigma}}.
\]
\end{lemma}

\begin{proof}Let us start by choosing $\eps_0,\, \eps$ and $\delta$ as in Lemma \ref{L15bis}. 
We will prove the stronger statement
\begin{equation}
\label{43ea}
|\partial_t^k\partial_z^{\beta} (w_1(t,z)-w_2(t,z))|\lesssim \|g_1-g_2\|_{L^2_{\sigma}},
\end{equation}
for all $(t,z)\in \left(\frac12,1\right)\times B_1(0)$ and $k\in\N_0,\, \beta\in\N_0^N$, where $w_1$ and $w_2$ denote the solutions to the nonlinear equation \eqref{42a} with initial values $g_1$ and $g_2$, respectively. We argue similarly as in the proofs of Lemmas \ref{L8a} and \ref{L9}. That is, we deduce \eqref{43ea} from the estimate
\begin{equation}
\label{43eb}
\|\partial_t^k\partial_z^{\beta} (w_1-w_2)\|_{L^2\left(\left(\frac12,1\right);L^2_{\sigma}\right)}\lesssim \|g_1-g_2\|_{L^2_{\sigma}}.
\end{equation}
For this, consider a smooth temporal cut-off function $\eta$ such that $\eta=1$ in $\left[\frac12,1\right]$ and $\eta=0$ in $\left[0,\frac14\right]$. Then $\eta (w_1 -w_2)$ satisfies the equation
\[
\partial_t \left(\eta (w_1-w_2)\right) +\L\left(\eta(w_1-w_2)\right) = \eta (f_1-f_2) + \partial_t \eta (w_1-w_2),
\]
where $f_i = f(w_i,\grad w_i)$ for $i=1,2$. By Lemma \ref{L7}, it holds that
\begin{eqnarray*}
\lefteqn{ \|\grad(\eta(w_1-w_2))\|_{L^2(L^2_{\sigma})} +  \|\rho \grad^2(\eta(w_1-w_2))\|_{L^2(L^2_{\sigma})} }\\
&\lesssim& \|\eta(f_1-f_2)\|_{L^2(L^2_{\sigma})} + \|w_1-w_2\|_{L^2(L^2_{\sigma})}.
\end{eqnarray*}
Thanks to the decay estimate \eqref{7az} in Theorem \ref{T1}, it holds that
\begin{equation}
\label{43ec}
\|\partial_z^{\beta} \partial_q^{\ell}\partial_p^{\gamma} F(w,\grad w)\|_{L^{\infty}\left(\left(\frac14,1\right)\times B_1(0)\right)} \lesssim \eps,
\end{equation}
for any $\ell\in\N_0$, $\beta,\, \gamma\in \N_0^N$ and any solution $w\in \B_{\eps,\delta}$. We thus have
\begin{eqnarray*}
\lefteqn{\|\eta(f_1-f_2)\|_{L^2(L^2_{\sigma})}}\\
& \lesssim & \eps  \|\grad(\eta(w_1-w_2))\|_{L^2(L^2_{\sigma})} +  \eps\|\rho \grad^2(\eta(w_1-w_2))\|_{L^2(L^2_{\sigma})} + \|w_1-w_2\|_{L^2(L^2_{\sigma})}.
\end{eqnarray*}
Because $\|w_1-w_2\|_{L^{\infty}(L^2_{\sigma})} \lesssim \|g_1-g_2\|_{L^2_{\sigma}}$ holds as a by-product of the proof of Theorem \ref{T4}, the latter implies \eqref{43eb} for $(k,|\beta|)=(0,1)$, 
provided that $\eps$ is sufficiently small.  Moreover, since
\[
\|\partial_z^{\beta} \left(\eta(f_1-f_2)\right)\|_{L^2(L^2_{\sigma})} \lesssim \eps \sum_{\gamma\le\beta} \|\partial_z^{\gamma} \left(\eta(w_1-w_2)\right)\|_{L^2(L^2_{\sigma})},
\]
we can argue as in Lemma \ref{L8a} (distinguishing ``angular'' and ``radial'' derivatives) and derive \eqref{43eb} for $k=0$. The control over the temporal derivatives is obtained by using the equation. Finally, \eqref{43ea} follows from \eqref{43eb} by the use of Morrey-type estimates. This concludes the proof.
\end{proof}

The following decomposition simplifies the discussion. If we denote by $L = e^{-\L}: L^2_{\sigma}\to L^2_{\sigma}$ the time-one map corresponding to the heat semi-group, then we write
\[
S = L+R.
\]
It is clear that $R: L^2_{\sigma}\to L^2_{\sigma}$ is Lipschitz because both $S$ and $L$ are.
Furthermore, $R$ is a contraction on $L^2_{\sigma}$ provided that $\eps$ is chosen sufficiently small and differentiable as a map from $\dot C^{0,1}$ to $L^2_{\sigma}$ and from $C^{0,1}$ to $C^{0,1}$:

\begin{lemma}\label{L17} There exists $\eps_0>0$ such that for all $0<\eps\le \eps_0$ and $\delta>0$ with $\sqrt2\left(\eps+\delta\right)<1$ the following holds: For any $g_1,\, g_2\in L^2_{\sigma}$, 
\[
\|  R(g_1) -  R (g_2)\|_{L^2_{\sigma}} \lesssim \eps \|  g_1 -  g_2\|_{L^2_{\sigma}}.
\]
Moreover,
\[
\|R (g)\|_{L^2_{\sigma}} \lesssim \|g\|_{\Lip}^2\quad\mbox{and}\quad \|R(g)\|_{C^{0,1}} \lesssim \|g\|_{C^{0,1}}^2
\]
for all $g\in C^{0,1}$. 
\end{lemma}

The contraction property will be a crucial ingredient in the construction of the invariant manifolds below. We will use the quadratic bound on $R$ to show that the center manifold meets the eigenspaces tangentially at the origin.

\begin{proof}
In the beginning, we choose $\eps_0,\, \eps$ and $\delta$ as in Lemma \ref{L15bis}. 
We first prove the Lipschitz estimate. Let $\tilde w$ denote the solution to the equation
\[
\partial_t \tilde w - \rho^{-\sigma}\div\left(\rho^{\sigma +1}\grad \tilde w\right)  = f_{\eps,\delta} (w_2) - f_{\eps,\delta}(w_1)
\]
with zero initial datum. Here $w_i(t) = S^t_{\eps,\delta} (g_i)$ and therefore $\tilde w(1) = R (g_2)-R(g_1)$. Using the elementary estimate $|F_{\eps,\delta}(q_1,p_1) - F_{\eps,\delta}(q_2,p_2)| \lesssim \eps \left(|q_1-q_2| + |p_1-p_2|\right)$, we compute
\begin{eqnarray*}
 \int \tilde w(1)^2\, d\mu_{\sigma}  
 &\lesssim&  \int_0^1 \int (F_{\eps,\delta}(w_1,\grad w_1) - F_{\eps,\delta}(w_2,\grad w_2))^2\, d\mu_{\sigma+1}dt\\
 &\lesssim &  \eps^2 \int_0^1 \int (w_2-w_1)^2\, d\mu_{\sigma}dt+ \eps^2 \int_0^1 \int |\grad w_2-\grad w_1|^2\, d\mu_{\sigma+1} dt.
\end{eqnarray*}
The fixed point argument in the construction of solutions to the truncated equation, Theorem \ref{T4}, yields that the latter is bounded by $\eps^2 \|g_1-g_2\|_{L^2_{\sigma}}^2$. This proves the first statement.

If, however, $g_1=g$ and $g_2=0$, then the previous argument also shows that
\[
\int \tilde w(1)^2\, d\mu_{\sigma} \lesssim \int_0^1 \int F(w,\grad w)^2\, d\mu_{\sigma+1}dt,
\]
with $w(t) = S_{\eps,\delta}^t(g)$. Since $|F_{\eps,\delta}(q,p)|\lesssim |p|^2$, we deduce the first quadratic bound on $R$ via Lemma \ref{L15bis}.

Moreover, because $\tilde w(0)=0$, Theorem \ref{T6} yields the estimate
\[
\|\tilde w\|_{\Lip} \lesssim \|f_{\eps,\delta}(w)\|_{Y(p)},
\]
where $p$ as in the assumption of the theorem. By Lemma \ref{L15bis}, the solution $w(t) = S^t_{\eps,\delta}(g)$ remains unchanged if we lower $\eps$ (if necessary) such that $\eps\lesssim \|g\|_{\Lip}$. Now, applying estimates \eqref{43f} and \eqref{43g} from the proof of Lemma \ref{L15bis}, we further have
\[
\|f_{\eps,\delta}(w)\|_{Y(p)} \lesssim \eps \|g\|_{\Lip} \lesssim \|g\|_{\Lip}^2,
\]
provided that $\eps$ is sufficiently small. We thus have
\[
\|R(g)\|_{\Lip}\lesssim \|g\|_{\Lip}^2.
\]
Thanks to the elementary estimate $\|\zeta\|_{L^{\infty}}\lesssim \|\zeta\|_{L^2_{\sigma}} + \|\zeta\|_{\Lip}$, the second quadratic bound follows.

\end{proof}

We are finally in the position to begin our dynamical systems argument.

The operators $\L$ and $L$ share the same eigenfunctions and the spectrum of $L$ is given by the eigenvalues $e^{-\lambda_{\ell k}}$, with the $\lambda_{\ell k}$'s given in Proposition \ref{P2} above.
We relabel these eigenvalues by $\{\lambda_{j}\}_{j\in\N_0}$, where $\lambda_j < \lambda_{j+1}$, that is, eigenvalues will not be repeated. The eigenfunctions of $\L$ (and thus of $L$) can be chosen in such a way that they form an orthonormal basis of $L^2_{\sigma}$. By construction, they are orthogonal in $\dot H^1_{\sigma+1}$.

From now on, we keep $K\in \N_0$ arbitrarily fixed. We denote by $P_c$ the spectral projection onto the (finite dimensional) subspace of $L^2_{\sigma}$ spanned by the eigenfunctions corresponding to the eigenvalues $\{\lambda_0,\dots ,\lambda_{K}\}$. Furthermore, we set $P_s := I - P_c$ and define $E_c = P_c L^2_{\sigma}$ and $E_s = P_s L^2_{\sigma}$, so that $L^2_{\sigma} = E_c\oplus E_s$. We decompose $L$ accordingly by setting $L_c = P_c L P_c$ and $L_s = P_s L P_s$. The subscripts ``{\it c}'' and ``{\it s}'' stand for ``{\em center}'' and ``{\em stable}'', respectively. It is clear that $L_c$ has a bounded inverse and
\[
\| L_c^{-\ell}\| \le e^{\ell\lambda_K}\quad\mbox{for all } \ell\in\N,
\]
where $\|\tacka\|$ denotes the operator norm. Indeed, if $\Pi_k$ projects onto the eigenspace spanned by the $k$th eigenfunction $\psi_k$, then $\Pi_k w(t) = e^{-\lambda_k t} \Pi_k g$. Hence, $\|L_c^{-1} w\|_{L^2_{\sigma}}^2 = \|P_c g\|_{L^2_{\sigma}}^2 = \sum_{k\le K} \la \Pi_k g, \psi_k\ra^2_{L^2_{\sigma}}  = \sum_{k\le K} e^{2\lambda_k} \la \Pi_k w,\psi_k\ra^2_{L^2_{\sigma}}  \le e^{2\lambda_K} \|w\|_{L^2_{\sigma}}^2$. The above bound follows by iteration. Likewise,
\[
\|L_s^{\ell} \| \le e^{-\ell\lambda_{K+1}}\quad\mbox{for all } \ell\in\N.
\]
The latter follows from the energy decay rate for the linear equation and a spectral gap estimate on $E_s$. Indeed, on the one hand $\partial_t  w + \L  w =0$ implies that $\frac{d}{dt} \frac12\|P_s w\|^2_{L^2_{\sigma}}  + \|\grad P_s  w\|^2_{L^2_{\sigma+1}} = 0$. On the other hand, $\lambda_{K+1} \| \hat w\|^2_{L^2_{\sigma}}\le \|\grad \hat w\|_{L^2_{\sigma+1}}^2$ for all $\hat w\in E_s$. Combining both inequalities and iteration yields the desired bound.

Following the notation introduced in \cite{Koch97}, we finally define
\[
\Lambda_s = e^{-\lambda_{K+1}},\quad \Lambda_c  = e^{-\lambda_K},\quad\mbox{and}\quad \Lambda_{\max} = 1.
\]
With these estimates, we have
\begin{equation}
\label{44az}
\|L_c^{-1}\|\le \Lambda_c^{-1},\quad \|L_s\| \le \Lambda_s,\quad \|L\|\le \Lambda_{\max}.
\end{equation}
Furthermore, if $\epsgap\in(0,1)$ is chosen small enough, we can find $\Lambda_-<\Lambda_+$ such that
\[
\begin{array}{rcccl}
 \Lambda_s + \epsgap & < &  \Lambda_- &<& \Lambda_c  - \epsgap,\\
 \Lambda_{\max} + \epsgap &<&\Lambda_+ .&&
 \end{array}
\]
For any function $w\in L_{\sigma}^2$, we define $\|w\| = \max\left\{\|P_c w\|_{L^2_{\sigma}}, \|P_s w\|_{L^2_{\sigma}}\right\}$.
Moreover, for any sequence $\{w_k\}_{k\in\Z}$ we set
\[
\|\{w_k\}_{k\in\Z}\|_{\Lambda_-,\Lambda_+} : = \sup_{k\in \N_0} \max\left\{\Lambda_+^{-k} \|w_k\|, \Lambda_-^k\|w_{-k}\|\right\},
\]
and for any sequence $\{w_k\}_{k\in\N_0}$,
\[
\|\{w_k\}_{k\in\N_0}\|_{\Lambda_-,+} : = \sup_{k\in \N_0} \Lambda_-^{-k} \|w_k\|.
\]
The corresponding Banach spaces will be denotes by $\ell_{\Lambda_-,\Lambda_+}$ and $\ell_{\Lambda_-,+}$, respectively. 

In a first step, we construct the center manifold.

\begin{prop}[Center Manifold]
\label{P3b}
Let $\eps$ be small enough so that $\Lip(R) \le \epsgap$. There exists a map $\theta: E_c\to E_s$ with $\theta(0)=0$ and $D\theta(0)=0$ such that the submanifold
\[
W_c = \left\{ w_c + \theta(w_c):\: w_c\in E_c\right\}
\]
has the following properties:
\begin{enumerate}
\item ($L^2_{\sigma}$ regularity) The map $\theta$ is Lipschitz with $\Lip(\theta)\lesssim \epsgap$. Moreover,
\begin{equation}
\label{44a}
\|\theta(g_c)\|_{L^2_{\sigma}}\lesssim \|g_c\|^2_{L^2_{\sigma}}
\end{equation}
for all $g_c\in \B_{\eps,\delta}\cap E_c$.
\item ($C^{0,1}$ regularity) The map $\theta$ from $C^{0,1}$ to $C^{0,1}$ is Lipschitz on $\B_{\eps,\delta}$. Moreover,
\begin{equation}
\label{44c}
\|\theta(g_c)\|_{C^{0,1}}\lesssim \|g_c\|^2_{C^{0,1}}
\end{equation}
for all $g_c\in \B_{\eps,\delta}\cap E_c$. 
\item (Invariance) The restriction on $W_c$ of the semi-flow $\{ S_{\eps,\delta}^t \}_{t\ge 0}$ can be extended to a Lipschitz flow on $W_c$. In particular, $S_{\eps,\delta}^t( W_c) = W_c$ for all $t\ge 0$, and for any $g\in W_c$ there exists a negative semi-orbit $\{w(t)\}_{t\le 0}$ in $W_c$ such that $w(0)=g$.
\item (Lyapunov exponent) 
A point $g$ belongs to $W_c$ if and only if there exists a flow $\{w(t)\}_{t\in\R}$ with $w(0)=g$ and
\[
\|w(t)\|_{L^2_{\sigma}}\lesssim  \left\{
\begin{array}{ll}\Lambda_+^t &\mbox{for all }t\ge 0,\\
\Lambda_-^t &\mbox{for all }t\le 0.
\end{array}\right.
\]
\item (Optimal Lyapunov exponent) A point $g$ belongs to $W_c$ if and only if there exists a unique flow $\{w(t)\}_{t\in\R}$ with $w(0)=g$ and
\[
\|w(t)\|_{L^2_{\sigma}}\lesssim  \left\{
\begin{array}{ll}\left(\Lambda_{\max}+\epsgap\right)^t &\mbox{for all }t\ge 0,\\
\left(\Lambda_c -\epsgap\right)^t &\mbox{for all }t\le 0.
\end{array}\right.
\]
\end{enumerate}
\end{prop}

In the statement of the proposition, $D\theta$ denotes the Fr\'echet derivative of $\theta$. 

The manifold $W_c$ is not unique but depends on the cut-off function $\eta_{\eps,\delta}$.

With slightly more effort, we could prove \eqref{44a} for all $g\in E_c$.

\begin{proof}Most parts of the theorem are an immediate consequence of \cite[Theorem 2.3]{Koch97}. We thus only sketch the proof. The results which are new or adapted to our particular study, however, will receive a detailed presentation.

For any $g_c\in E_c$ and $\{w_{\ell}\}_{\ell\in\Z}$ we define $\{J_k(g_c,\{w_{\ell}\}_{\ell\in\Z})\}_{k\in\Z}$ by
\[
J_k\left(g_c, \{w_{\ell}\}_{\ell\in\Z}\right) : = \left\{\begin{array}{ll} P_s S(w_{k-1}) + L_c^{-1} P_c(w_{k+1} - R(w_k)) & \mbox{for }k\le -1,\\
P_s S(w_{-1}) +g_c&\mbox{for }k=0,\\
S(w_{k-1})&\mbox{for }k\ge 1.
\end{array}\right.
\]
It is easily checked that $J = \{J_k\}_{k\in\Z}$ maps $E_c\times \ell_{\Lambda_-,\Lambda_+}$ to $\ell_{\Lambda_-,\Lambda_+}$, and for every fixed $g_c\in E_c$, $J(g_c,\tacka)$ is a contraction on $\ell_{\Lambda_-,\Lambda_+}$. Thus, by Banach's fixed point theorem, for every $g_c\in E_c$, $J(g_c,\tacka)$ has a unique fixed point $\{w_k\}_{k\in\Z}$ in $\ell_{\Lambda_-,\Lambda_+}$.
It is a solution to the discrete semi-flow with $P_c w_0=g_c$ and extends to a global flow, that is
\[
w_k = S^k(w_0)\quad\mbox{for all }k\in\Z\quad\mbox{and}\quad P_c w_0 = g_c.
\]
We then set $\hat \theta (g_c)  = \{w_k\}_{k\in\Z}$ and $\theta (g_c) = P_s \hat \theta_0 (g_c) = P_s w_0 $, that is the projection of first coordinate of the sequence $\{w_k\}_{k\in\Z}$ onto $E_s$. Thus $\hat \theta: E_c\to \ell_{\Lambda_-,\Lambda_+}$ and $\theta: E_c\to E_s$, and both maps are Lipschitz as a consequence of the fixed point argument.
Moreover, it can be shown that $\Lip(\theta) \lesssim \epsgap$. Indeed, using \eqref{44az} and Lemma \ref{L17}, via iteration, we observe that
\[
\|\theta(g_c)-\theta(\tilde g_c)\|_{L^2_{\sigma}} \le\left( \left(\frac{\Lambda_s}{\Lambda_-}\right)^k + \epsgap \sum_{\ell=1}^k\left(\frac{\Lambda_s}{\Lambda_-}\right)^{\ell-1} \right) \|\hat \theta(g_c) - \hat \theta(\tilde g_c)\|_{\Lambda_-,\Lambda_+}.
\]
Then the statement follows from using the Lipschitz estimate for $\hat\theta$ and letting $k\uparrow \infty$.

The argument for \eqref{44a} proceeds similar. Notice that $\| R(w_{-k})\|_{L^2_{\sigma}} \lesssim \| S (w_{-k-1})\|_{\Lip}^2 \lesssim \|w_{-k-1}\|_{L^2_{\sigma}}^2$ via Lemmas \ref{L17} and \ref{L15b}. Hence, via iteration
\[
\|\theta(g_c)\|_{L^2_{\sigma}} \le  \left(\frac{\Lambda_s}{\Lambda_-}\right)^k \|\hat \theta(g_c)\|_{\Lambda_-,\Lambda_+} + C\sum_{\ell=1}^k \left(\frac{\Lambda_s}{\Lambda_-^2} \right)^{\ell-1} \|\hat \theta(g_c)\|_{\Lambda_-,\Lambda_+}^2,
\]
and the statement follows from letting $k\uparrow\infty$ via the Lipschitz estimate for $\hat \theta$ because $\hat \theta(0)=0$. Notice that \eqref{44a} entails that $D\theta(0)=0$, where $D\theta$ is the Fr\'echet derivative with respect to $L^2_{\sigma}$.

That $\theta$ is also a contraction as a mapping from $C^{0,1}$ to $C^{0,1}$ can be seen as follows: By the definition of $\theta$ and the regularity estimate from Lemma \ref{L15b}, for any $g_c$ and $\tilde g_c$ in $E_c$, it holds that
\[
\|\theta(g_c) - \theta(\tilde g_c)\|_{C^{0,1}} \lesssim \|S(w_{-1}) - S(\tilde w_{-1})\|_{C^{0,1}} \lesssim \| w_{-1} - \tilde w_{-1}\|_{L^2_{\sigma}},
\]
where $\{w_{k}\}_{k\in\Z} = \hat \theta(g_c)$ and $\{\tilde w_k\}_{k\in\Z} = \hat \theta(\tilde g_c)$. Hence
\[
\|\theta (g_c)-\theta(\tilde g_c)\|_{C^{0,1}}\lesssim \|\hat \theta(g_c) - \hat \theta(\tilde g_c)\|_{\Lambda_-,\Lambda_+} \lesssim \|g_c-\tilde g_c\|_{L^2_{\sigma}},
\]
and the last inequality is due to the Lipschitz dependence of $\hat \theta$ on $g_c$. The $L^2_{\sigma}$ norm being trivially estimated by the $C^{0,1}$ norm, we conclude the Lipschitz regularity with respect to $C^{0,1}$. The quadratic bound \eqref{44c} is established as follows: Via triangle inequality it holds $\|\theta(g_c)\|_{C^{0,1}} \lesssim \|L_s w_{-1}\|_{C^{0,1}} + \|R(w_{-1})\|_{C^{0,1}}$. We use the quadratic estimate for $R$ in Lemma \ref{L17} and the smoothing estimates for $L$ and $S$ in Lemmas \ref{L9} and \ref{L15b} to deduce $\|\theta(g_c)\|_{C^{0,1}} \lesssim \|P_s w_{-2}\|_{L^2_{\sigma}} + \|w_{-2}\|_{L^2_{\sigma}}^2$, and the statement follows via iteration as above.

The remaining parts follow by construction, see \cite{Koch97}.
\end{proof}

\begin{prop}[Stable Manifold]\label{P4}
Let $\eps$ be small enough so that $\Lip(R)\le \epsgap$. 
There is a continuous map $\nu : E_s\times L^2_{\sigma} \to E_c$ such that for each $g\in W_c$ , $\nu ( -P_s g, g) = - P_c g$ and the manifold
\[
M_{g}  = g + \left\{ \nu(w_s,g) + w_s:\: w_s\in E_s\right\}
\]
has the following properties:
\begin{enumerate}
\item ($L^2_{\sigma}$ regularity) For every fixed $g\in L^2_{\sigma}$, the map $\nu(\tacka,g): E_s\to E_c$ is Lipschitz continuous with $\Lip(\nu(\tacka,g))\lesssim \epsgap$.
\item ($C^{0,1}$ regularity) The map $\nu$ is continuous as a mapping from $C^{0,1}\times C^{0,1}$ to $C^{0,1}$. Moreover, for every fixed $g\in C^{0,1}$, the map $\nu(\tacka,g): C^{0,1}\cap E_s\to C^{0,1}$ is Lipschitz with $\Lip(\nu(\tacka,g))\lesssim \epsgap$.
\item (Invariant Foliation) It holds $S_{\eps,\delta}^t(M_{g})\subset M_{S_{\eps,\delta}^t(g)}$ for all $t\ge 0$ and
\[
M_g = \left\{ \tilde g\in L^2_{\sigma}:\: \sup_{k\in\N_0} \Lambda_-^{-k} \| S^k_{\eps,\delta} (g) - S^k_{\eps,\delta}(\tilde g)\|_{L^2_{\sigma}} \le (1-\kappa)^{-1}\|P_s(g-\tilde g)\|_{L^2_{\sigma}}\right\},
\]
where
\[
\kappa = \max\left\{\frac{\Lambda_s+ \epsgap}{\Lambda_-}, \frac{\Lambda_- +\epsgap}{\Lambda_c}\right\}.
\]
\item (Completeness) For every $g\in L^2_{\sigma}$ the set $M_g\cap W_c$ is exactly a single point. In particular, $\{M_g\}_{g\in L^2_{\sigma}}$ is a foliation of $L^2_{\sigma}$  over $W_c$. 
\end{enumerate}
\end{prop}

\begin{proof}
Up to the proof of the second item and the precise decay rate in the third item, the statement can be found, e.g., in \cite[Theorem 2.4]{Koch97}. For the new results (which are peculiar to the situation at hand), we have to review the construction of $\nu$. For this purpose, consider the mapping $ I = \{I_k\}_{k\in\N_0}: E_s\times \ell_{\Lambda_-,+} \to \ell_{\Lambda_-,+}$ defined by
\[
I_k(g_s,\{w_{\ell}\}_{\ell\in\N_0}) =\left\{\begin{array}{ll}
g_s + L_c^{-1} P_c (w_1 - R(w_0)) &\mbox{for }k=0,\\
P_s S(w_{k-1}) + L_c^{-1} P_c(w_{k+1} - R(w_k)) &\mbox{for }k\ge 1.
\end{array}\right.
\]
Let $g_s$ be fixed for a moment. Then $I(g_s,\tacka)$ is a contraction on $\ell_{\Lambda_-,+}$ with contraction constant
\[
\kappa = \max\left\{\frac{\Lambda_s+ \epsgap}{\Lambda_-}, \frac{\Lambda_- +\epsgap}{\Lambda_c}\right\}.
\]
We let $R \ge (1-\kappa)^{-1} \|g_s\|_{L^2_{\sigma}}$. Since $I_k(g_s + P_sg, \{S^{\ell}(g)\}_{\ell\in\N_0}) - S^{k}(g) = g_s\delta_{k0}$ for every $g\in L_{\sigma}^2$, we notice that $I(g_s+P_sg,\tacka)$ maps the $\ell_{\Lambda_-,+}$-ball of radius $R$ around $\{S^k(g)\}_{k\in\N_0}$ onto itself. Hence, Banach's fixed point theorem yields the existence of a unique sequence $\{w_k\}_{k\in\N_0}$ such that $\{w_k\}_{k\in\N_0} = I(g_s + P_sg,\{w_{\ell}\}_{\ell\in\N_0})$ and satisfying
\[
\|\{w_k\}_{k\in\N_0} - \{S^k(g)\}_{k\in\N_0}\|_{\Lambda_-,+} \le R,
\]
By construction, $\{w_k\}_{k\in\N_0}$ is a discrete semi-flow with $P_s w_0 = g_s+P_sg$. We now define $\hat \nu(g_s,g) = \{w_k\}_{k\in\N_0}$ and $\nu(g_s,g):= P_c(w_0-g)$. Via iteration we obtain
\[
\|\nu(g_s,g) - \nu(\tilde g_s,g)\|_{L^2_{\sigma}} \le
\left( \left(\frac{\Lambda_-}{\Lambda_c }\right)^k + \frac{\epsgap}{\Lambda_c } \sum_{\ell=0}^{k-1}\left(\frac{\Lambda_-}{\Lambda_c }\right)^{\ell} \right) \|\hat\nu(g_s,g) - \hat\nu(\tilde g_s,g)\|_{\Lambda_-,+},
\]
and since $\hat \nu(\tacka,g)$ is Lipschitz by construction, we see that $\Lip(\nu(\tacka,g) )\lesssim\epsgap$ by letting $k\uparrow\infty$.

Then, for any two $g_s$ and $\tilde g_s$ in $L^2_{\sigma}$ and $\{w_k\}_{k\in\N_0} = \hat\nu(g_s,g)$ and $\{\tilde w_k\}_{k\in\N_0} = \hat \nu(\tilde g_s,g)$, it holds that
\[
\|\nu(g_s,g) - \nu(\tilde g_s,g)\|_{C^{0,1}}  = \|P_c(w_0-\tilde w_0)\|_{C^{0,1}} \lesssim \|L_cP_c(w_0-\tilde w_0)\|_{C^{0,1}},
\]
by the invertibility of $L_c$ in $C^{0,1}$. We now apply Lemma \ref{L9} to deduce
\[
\|\nu(g_s,g) - \nu(\tilde g_s,g)\|_{C^{0,1}}\lesssim \|P_c(w_0-\tilde w_0)\|_{L^2_{\sigma}} = \|\nu(g_s,g) - \nu(\tilde g_s,g)\|_{L^2_{\sigma}}.
\]
Invoking the Lipschitz bound on $\nu(\tacka,g)$ in $L^2_{\sigma}$ and the trivial embedding of $C^{0,1}$ into $L^2_{\sigma}$, $\nu(\tacka,g)$ is Lipschitz with constant $\lesssim \epsgap$.
A very similar argument shows that $\nu$ is also $C^{0,1}$ Lipschitz in the second argument. 

Based on the construction of $\{w_k\}_{k\in\N_0}$, most of the properties in the statement of the proposition are readily verified. For instance, invariance follows from the construction by uniqueness. Moreover, if $\tilde g= g+\nu(g_s,g) + g_s\in M_g$ for some $g_s\in E_s$, then $\tilde g  =w_0$ where $\{w_k\}_{k\in\N_0} = \hat \nu(g_s,g)$. In particular $w_k = S^k(\tilde g)$ for all  $k\in \N_0$ and thus $\|\{S^k(g)\}_{k\in\N_0} - \{S^k(\tilde g)\}_{k\in\N_0}\|_{\Lambda_-,+}\le R$, with $R = (1-\kappa)^{-1}\|P_s(g-\tilde g)\|_{L^2_{\sigma}}$. On the other hand, if the latter holds for some $\tilde g\in L^2_{\sigma}$, then a similar argument shows that $\tilde g\in M_g$.

Finally, completeness can be seen as follows: Let $\chi: E_s\times L^2_{\sigma}\to E_s$ be given by
\[
\chi(g_s,g) = \theta(\nu(g_s-P_s g,g) + P_cg).
\]
Since $\theta$ and $\nu(\tacka,g)$ are both contractions with respect to the $L^2_{\sigma}$ topology, so is $\chi(\tacka,g)$. Hence, for every $g\in L^2_{\sigma}$, there exists a unique fixed point $g_s = \chi(g_s,g)$. We now set
\[
g_c: =\nu(g_s-P_sg,g) + P_c g,\quad\mbox{and}\quad g_s = \theta(g_c).
\]
In particular, the definitions imply that $g_c + g_s\in W_c\cap M_g$. However, every intersection point is a fixed point of $\chi(\tacka,g)$. Hence, by the uniqueness of the fixed point, $W_c\cap M_g = \{g_c+g_s\}$.

\end{proof}

In the previous propositions, we have derived an invariant manifold theorem for the flow of the truncated equation \eqref{43}. This theorem can be carried over to the original problem.

\begin{proof}[Proof of Theorem \ref{T5}]
That $W_c^{\loc} $ is locally flow invariant is clear from Theorem \ref{T1} and Proposition \ref{P3b}. For the same reason, $W_c^{\loc}$ contains all the negative semi-orbits that satisfy the constraint.

Let $\tilde g$ be the unique point in $M_g\cap W_c$ found in Proposition \ref{P4}. We claim that
\begin{equation}\label{44b}
S_{\eps,\delta}^t(\tilde g) = S^t(\tilde g) \quad\mbox{for all }t\ge 0,
\end{equation}
if $0<\eps<\eps_*$ for some $\eps_*$ sufficiently small if $\| g\|_{\Lip}\le \eps_0$ for some $0< \eps_0\le \eps$ sufficiently small. From this, the characterization of the stable manifold $M_g$ in Proposition \ref{P4} yields that
\[
\|S^t(g) - S^t(\tilde g)\|_{L^2_{\sigma}}\lesssim \max\left\{\frac{\Lambda_-}{\Lambda_- - \Lambda_s -\epsgap}, \frac{\Lambda_c}{\Lambda_c - \Lambda_- - \epsgap}\right\} \Lambda_-^t\quad\mbox{for all }t\ge 0.
\]
We now invoke the smoothing estimate from Lemma \ref{L15b} to deduce the statement of the theorem.

It remains to verify \eqref{44b}. The construction of $W_c$ and $M_g$ in Propositions \ref{P3b} and \ref{P4} yields the existence of points $g_c\in E_c$ and $g_s\in E_s$ such that $\tilde g = g_c+g_s$ and $g_c = \nu(g_s - P_sg,g) + P_c g$ and $g_s=\theta(g_c)$. Hence $g_c$ solves the fixed point equation
\begin{equation}
\label{45}
g_c = \nu(\theta(g_c)-P_sg,g)+P_cg = \nu(\theta(g_c) - P_s g,g) - \nu(0,g).
\end{equation}
With respect to the $C^{0,1}$ topology, $\theta$ is Lipschitz in $\B_{\eps,\delta}$ by Proposition \ref{P3b} and $\nu(\tacka, g)$ is a contraction by Proposition \ref{P4}, provided that $\eps$ is chosen sufficiently small. It follows that the mapping $g_c\to \nu(\theta(g_c)-P_sg,g)+ P_cg$ is a contraction with respect to $C^{0,1}$ in $\B_{\eps,\delta}$, and thus, in $\B_{\eps,\delta}$, \eqref{45} has a unique solution $g_c = g_c(g)$ which depends continuously on $g$ by the $C^{0,1}$ continuity of $\nu$. Now, since $\theta(0)=0$, $\nu(0,0)=0$ and $g_c(0)=0$, by continuity, there exist $\tilde \delta>0$ and $\tilde \eps>0$ such that $\tilde g\in \B_{\eps,\delta}$ if $g\in \B_{\tilde \eps,\tilde \delta}$. This proves \eqref{44b}.
\end{proof}

\newpage

\section{Applications of the invariant manifold theorem}\label{S:8}

In this final section, we give four applications of the invariant manifold theorem. We prove the long-time asymptotics for solutions to the perturbation equation with small initial data: solutions to \eqref{42a} converge to (small) constants, see Theorem \ref{T10} above.  We refine the estimate that is deduced from Theorem \ref{T5} and compute the precise rate of convergence. On the level of the porous medium equation, this rate is saturated by spatial translations of the Barenblatt solution. This first result is obtained by choosing $K=0$ in the notation of the previous section, that is, we investigate the asymptotics towards the one-dimensional manifold spanned by the constant eigenfunctions corresponding to the eigenvalue $\lambda_0=0$. In the second example, we fix the long-time limit of solutions by letting the mean vanish asymptotically. Hence, the long-time limit is the zero function. According to the invariant manifold theorem, solutions converge to the $1+N$ dimensional center manifold corresponding to the eigenvalues $\lambda_	0$ and $\lambda_1$ (thus $K=1$) with an improved rate $\lambda_2-\eps$, see Theorem \ref{T11}. On the level of the porous medium equation, we subtract out the spatial translations. Climbing the eigenvalue ladder one rung up ($K=2$), in Theorem \ref{T17}, we additionally asymptotically suppress the center of mass of solutions, so that the second order corrections (affine transformations) in the asymptotic expansion becomes accessible. Finally, in  Theorem \ref{T14}, we consider the dynamics on the rate $\lambda_4$ (hence $K=3$) in the case where $N<\sigma+1$ , which, on the porous medium level, is governed by dilations. 

In principle, we can access all eigenmodes through the invariant manifolds of Theorem \ref{T5}. Since the mathematical arguments behind the computation of even higher modes do not change significantly, we will finish this study with the third order corrections.

The proofs of Theorems \ref{T10} to \ref{T14} are in order.

\begin{proof}[Proof of Theorem \ref{T10}]
With regard to the notation introduced in the previous section, we suppose that $K=0$, and thus $E_c = \Span\{\psi_{010}\}\cong \R$. We choose $\lambda$ between $\lambda_1$ and $\lambda_0$ and find $\epsgap>0$ such that $\Lambda_- = e^{-\lambda}\in (e^{- (\sigma+1)} + \epsgap, 1 - \epsgap)$. Let $\eps,\, \delta,\, \tilde \eps$ and $\tilde \delta$ be given as in the statement of Theorem \ref{T5}.
Our first claim is
\begin{equation}
\label{46}
W_c^{\loc} \cong \left\{a\in \R:\: |a|\le \delta\right\}.
\end{equation}
Indeed, on the one hand, if $w\in W_c^{\loc}$, then $w\equiv a$ for some $|a|\le \delta$, which shows the inclusion ``$\subset$''. On the other hand, if $|a|\le \delta$ and $w(t)=\delta$ for all $t\le 0$, then $\{w(t)\}_{t\le0}$ is a negative semi-orbit meeting the norm constraints. Hence, by Theorem \ref{T5}, $w(t)\in W_c^{\loc}$ for all $t\le 0$. This implies the inclusion ``$\supset$''.

In particular, by Theorem \ref{T5}, there exists a unique semi-flow $\{\tilde w(t)\}_{t\ge 0}$ in $W_c^{\loc} $ such that
\begin{equation}
\label{46azz}
\|w(t) - \tilde w(t)\|_{C^{0,1}}\lesssim  \Lambda_-^t\quad\mbox{for all }t\ge0.
\end{equation}
However, \eqref{46} implies that $\tilde w$ is constant with modulus less than $\delta$, say $\tilde w \equiv \tilde a$.

It remains to optimize the rate of convergence to $e^{-(\sigma+1)t}$. Consider now $a(t)  =|B_1(0)|^{-1}_{\sigma}\int w(t)\, d\mu_{\sigma}$. A short computation reveals that
\[
\frac{d}{dt} a(t) = \beta \int F(w,\grad w)\, d\mu_{\sigma+1},
\]
and thus, taking into account $F(q,p)\lesssim |p|^2$ and the above decay estimate \eqref{46azz}, we deduce that $0< \frac{d}{dt} a(t) \lesssim e^{-2\lambda t}$. Hence, there exists a constant $a\in\R$ such that $|a(t) - a|\lesssim e^{-2\lambda t}$. On the other hand, $|a(t) - \tilde a|\lesssim e^{-\lambda t}$ by the definition of $a(t)$ and \eqref{46azz}, and thus $\tilde a = a$.

The above computation shows that $P_s w(t) = w(t) -a(t) \in E_s$  satisfies
\begin{equation}
\label{46azza}
\|P_s w(t)\|_{C^{0,1}}\lesssim e^{-\lambda t}\quad\mbox{for all }t\ge 0.
\end{equation}
(Notice that $E_s$ is the space of all mean-zero $L^2_{\sigma}$ functions.) We shall go over to the discrete semi-flow for a moment, that is, we set $w_k = w(k)$ for $k\in\N_0$. Using the decomposition $S = L+R$ introduced in the previous section, we write
\[
w_k  = L^k  w_0 + \sum_{j=1}^k L^{k-j} R(w_{j-1}).
\]
From a variant of Lemma \ref{L17} we deduce the estimate $\|R(g)\|_{L^2_{\sigma}}\lesssim \eps^{1/2} \|P_s g\|_{\Lip}^{1/2} \|P_s g\|_{L^2_{\sigma}}$, which applied to the above identity, yields
\begin{equation}
\label{46azzc}
\|P_s w_k\|_{L^2_{\sigma}} \le \Lambda_s^k \|P_s w_0\|_{L^2_{\sigma}} +C \eps^{1/2}\sum_{j=1}^k \Lambda_s^{k-j} \|P_s w_{j-1}\|_{\Lip}^{1/2} \|P_s w_{j-1}\|_{L^2_{\sigma}}
\end{equation}
for some $C>0$. We claim that for $\eps$ small (but independent of $k$) it holds that
\begin{equation}
\label{46azzb}
\|P_s w_k\|_{L^2_{\sigma}}\le 2C\Lambda_s^k \|w_0\|_{L^2_{\sigma}}\quad\mbox{for all }k\in\N_0.
\end{equation}
We argue by induction. The statement if trivial for $k=0,\,1$. Assuming that \eqref{46azzb} holds true for $\ell=1,\dots,k-1$, we deduce from \eqref{46azzc}
\[
\|P_s w_k\|_{L^2_{\sigma}} \le \left(1 + 2 C^{3/2} \eps^{1/2} \Lambda_s^{-1} \sum_{j=1}^k \|P_s w_{j-1}\|_{\Lip}^{1/2}\right)\Lambda_s^k \|w_0\|_{L^2_{\sigma}} .
\]
Because $ \|P_s w_{j-1}\|_{\Lip}$ is exponentially decaying by \eqref{46azza}, the statement  \eqref{46azzc} is true for any $k$ provided that $\eps$ is sufficiently small. From \eqref{46azzb} we infer the statement for the continuous flow with the help of smoothing estimates, cf.\ \eqref{43ea} in the proof Lemma \ref{L15b}.

(The derivation  of \eqref{46azzb} was inspired by a similar argument in the proof of \cite[Theorem 1.1]{DenzlerKochMcCann15}.)
\end{proof}

\begin{proof}[Proof of Theorem \ref{T11}]
Suppose that $K=1$ and thus $E_c \cong \Span\left\{1,z_1,\dots,z_N\right\}$. Moreover, $\Lambda_s = e^{-2(\sigma+1)}$ and $\Lambda_c = e^{-(\sigma+1)}$. We choose $\lambda\in (\sigma+1,2(\sigma+1))$ and set $\Lambda_- = e^{-\lambda}$. There exists a small $\epsgap>0$ such that $\Lambda_-\in \left(\Lambda_s + \epsgap, \Lambda_c-\epsgap\right)$. Let $0<\tilde \eps<\eps$, $0<\tilde \delta<\delta$, and $\tilde g$ be given as in Theorem \ref{T5}. Then, setting $\tilde w(t) = S^t(\tilde g)$, it holds
\begin{equation}
\label{46az}
\|w(t) - \tilde w(t)\|_{C^{0,1}}\lesssim e^{-\lambda t}\quad\mbox{for all }t\ge 0.
\end{equation}
Because $\tilde w(t)\in W_c$, we find $\tilde a(t)\in \R$ and $\tilde b(t)\in \R^N$ such that
\[
\tilde w(t) =\tilde a(t) + \tilde b(t)\cdot z + \theta(\tilde a(t),\tilde b(t),z).
\]
The statement now follows from the estimates
\begin{eqnarray}
|\tilde a(t)|&\lesssim& e^{-2(\sigma+1)t}\quad\mbox{for all }t\ge0,\label{47}\\
|\tilde b(t) - e^{-(\sigma+1)t}  b|&\lesssim & e^{-2(\sigma+1)t}\quad\mbox{for all }t\ge0,
\label{48}
\end{eqnarray}
and the quadratic bound for $\theta$ in \eqref{44c}.

The proof of \eqref{47} proceeds in the same way as we obtained  the estimate for $a(t) $ in the proof of Theorem \ref{T10}. Notice that
\[
\tilde a(t) = \frac1{|B_1(0)|_{\sigma}}\int \tilde w(t,z)\, d\mu_{\sigma}(z)\to 0\quad\mbox{as }t\uparrow \infty
\]
by the virtue of Theorem \ref{T10} and the assumed long-time limit of $w(t)$. Now, since $\tilde w$ solves the perturbation equation, it holds that
\[
\frac{d}{dt} \tilde a(t) = \frac{\beta}{|B_1(0)|_{\sigma}} \int F(\tilde  w,\grad\tilde w)(t) \, d\mu_{\sigma+1}\lesssim \int |\grad \tilde w(t)|^2\, d\mu_{\sigma+1},
\]
and thus, using the decay rate for $\grad \tilde w$ in Theorem \ref{T10}, we deduce \eqref{47}. 

The second estimate \eqref{48} is established in a similar way. We notice first that
\[
\tilde b_i(t)  = c_{\sigma,N}  \int \tilde w z_i\, d\mu_{\sigma},
\]
for all $i\in\{1,\dots,N\}$ and some constant $c_{\sigma,N}$, and thus, differentiation and integration by parts yields
\[
\frac{d}{dt}\tilde  b_i(t)  = -(\sigma +1) \tilde b_i(t) +c_{\sigma,N} (\beta+1 \int F(\tilde w,\grad \tilde w) z_i\, d\mu_{\sigma+1}.
\]
With the help of Theorem \ref{T10}, the latter leads to the estimate
\[
\left|\frac{d}{dt}\left(e^{(\sigma+1)t} \tilde b_i(t)\right)\right| \lesssim e^{-(\sigma+1)t},
\]
and we deduce \eqref{48} with $ b = \lim_{t\uparrow\infty} e^{(\sigma+1)t} \tilde b(t)$. This concludes the proof.
\end{proof}

\begin{proof}[Proof of Theorem \ref{T17}]
As for the notation of the previous section, we are now in the situation $K=2$. Compared to the previous theorem, we augment the finite dimensional projection space $E_c$ by the eigenspace associated to the new eigenvalue $\lambda_{20} = 2(\sigma+1)$. This eigenvalue corresponds to affine transformations of the attracting ground state and the eigenfunctions are harmonic homogeneous polynomials of degree two, which can be expressed in the form $\psi(z) = z\cdot Az$ for some symmetric, trace-free matrix $A$. We will further split $A = B+C$, where $B$ is off-diagonal and symmetric and $C$ is diagonal and trace-free. 
Therefore, if $\tilde w$ is as in Theorem \ref{T5}, it holds
\[
\tilde w (t,z)  = \tilde a(t) + \tilde b(t)\cdot z + z\cdot \tilde B(t) z  + z\cdot \tilde C(t)z+ \theta(\tilde a(t),\tilde b(t),\tilde B(t),\tilde C(t),z),
\]
for some $\tilde a(t)\in \R$, $\tilde b(t)\in \R^N$, $\tilde B(t)\in \R^{N\times N}$ off-diagonal, symmetric and $\tilde C(t)\in R^{N\times N}$ diagonal, trace-free. Besides showing convergence for $\tilde B(t)$ and $\tilde C(t)$, we need new estimates for $\tilde a(t)$ and $\tilde b(t)$. Repeating the estimates form Theorem \ref{T11} and using the assumptions, we first notice  that $|\tilde a(t)|,\, |\tilde b(t)|\lesssim e^{-2(\sigma+1)t}$ since $b=0$, cf.\ \eqref{47}, \eqref{48}, and thus $\|\tilde w(t)\|_{C^{0,1}} \lesssim e^{-(2(\sigma+1)-\eps)t}$ from the statement of Theorem \ref{T11}.

Let $\{E^n\}$ and $\{F^n\}$ be orthogonal bases for the spaces of off-diagonal symmetric matrices and trace-free diagonal matrices, respectively. Then $\tilde B(t) = \sum_n \tilde \beta_n(t) E^n$ and $\tilde C(t) = \sum_n\tilde \gamma_n(t) F^n$ for some $\tilde \beta_n(t), \,\tilde \gamma_n(t)\in \R$. It is clear that
\[
\tilde \beta_n(t) = c_{\sigma, N}\sum_{i,j=1}^N E_{ij}^n \int z_iz_j\tilde w(t)\, d\mu_{\sigma},\quad	 \tilde \gamma_n(t) =  c_{\sigma, N}\sum_{i,j=1}^N F_{ij}^n \int z_iz_j\tilde w(t)\, d\mu_{\sigma},
\]
where $c_{\sigma,N}$ may be different in both formulas. 
 Using the perturbation equation and integrating by parts a couple of times, we compute
\begin{eqnarray}
\lefteqn{\frac{d}{dt} \int z_iz_j\tilde w\, d\mu_{\sigma}}\nonumber \\
& =& 2\delta_{ij} \int \tilde w\, d\mu_{\sigma+1} - 2(\sigma+1) \int z_iz_j \tilde w\, d\mu_{\sigma} + (\beta+2) \int z_iz_j \tilde F\, d\mu_{\sigma+1},\label{48a}
\end{eqnarray}
where $\tilde F = F(\tilde w,\grad\tilde w)\lesssim |\grad\tilde w|^2$ is bounded by $e^{-(4(\sigma+1) -2\eps)t}$. First, if $i\not=j$, then the above implies that $|B - e^{2(\sigma+1)t} \tilde B(t)| \lesssim e^{-2(\sigma+1-\eps)t}$ for some off-diagonal symmetric matrix $B\in \R^{N\times N}$.
Similarly, because $\sum_{i,j} F^n_{ij}\delta_{ij} = \tr(F^n)=0$, it holds $|C - e^{2(\sigma+1)t} \tilde C(t)| \lesssim e^{-2(\sigma+1-\eps)t}$ for some diagonal trace-free matrix $C\in \R^{N\times N}$. 
 The above bounds on $\tilde B(t)$ and $\tilde C(t)$  reveal that $\tilde A(t) = \tilde B(t) + \tilde C(t)$  is bounded by $ e^{-2(\sigma+1)t}$. Using this new estimate in the expansion of $\tilde w$ and applying the quadratic bound on $\theta$ in \eqref{44c} yields that $\|\tilde w(t)\|_{C^{0,1}}\lesssim e^{-2(\sigma+1)t}$. As a consequence, we can choose $\eps=0$ in the previous computations. We have thus shown that
\[
\left|\tilde A(t) - e^{-2(\sigma+1)t}A\right| \lesssim e^{-4(\sigma+1)t}
\]
where $A=B+C$ and $|\tilde a(t)|,\, |\tilde b(t)| \lesssim e^{-4(\sigma+1)t} $. In particular, the both, the affine term and $\theta$ in the expansion of $\tilde w$ are of higher order. This proves the statement of the theorem.
\end{proof}

\begin{proof}[Proof of Theorem \ref{T14}]
We consider the case $K=3$, that is, compared to Theorem \ref{T17}, the new eigenvalue $\lambda_{01}$ has multiplicity one and the corresponding eigenfunction is $\psi_{011}(z) =1-\gamma |z|^2$, where $\gamma = N^{-1}\left(2(\sigma +1) +N\right)$. Consequently, 
\[
E_c  =\Span\left\{1,z_1,\dots,z_N,1-\gamma|z|^2\right\}\cup \Span\{z\cdot Az:\: A=A^T,\tr(A)=0\},
\]
 and thus, according to Theorem \ref{T5}, there exists a solution $\tilde w$ to the perturbation equation which is of the form
\[
\tilde w(t,z) = \tilde a(t) + \tilde b(t)\cdot z +z\cdot \tilde A(t)z + \tilde  c(t)\left(1-\gamma |z|^2\right) + \theta\left(\tilde a(t),\tilde b(t),\tilde A(t), \tilde c(t),z\right),
\]
and which satisfies
\begin{equation}\label{50}
\|w(t) - \tilde w(t)\|_{C^{0,1}}\lesssim e^{-\lambda t}\quad\mbox{for all }t\ge0.
\end{equation}
By the orthogonality of the eigenbasis, $\tilde c(t)$ has the representation
\[
\tilde c(t) = c_{N,\sigma} \int \tilde w(t)\, d\mu_{\sigma} - c_{N,\sigma} \gamma \int \tilde w(t)|z|^2\, d\mu_{\sigma},
\]
for some constant $c_{N,\sigma}>0$. It is clear that $|\tilde a(t)|, \, |\tilde b(t)|,\, |\tilde A(t)|\lesssim e^{-4(\sigma+1)t}$ by the proof of the previous theorem. Hence, recalling the formula for $\tilde a(t)$, it holds that $\tilde c(t) =  - c_{N,\sigma} \gamma \int \tilde w(t)|z|^2\, d\mu_{\sigma} +O(e^{-4(\sigma+1)t})$. To determine the decay behavior of $\tilde c(t)$, we use \eqref{48a} in the form of
\[
\frac{d}{dt} \int |z|^2\tilde w\, d\mu_{\sigma} = -(2(\sigma+1) +N) \int |z|^2\tilde w\, d\mu_{\sigma} +  N \int \tilde w\, d\mu_{\sigma}  + (\beta+2)\int|z|^2 \tilde F\,d\mu_{\sigma},
\]
where $\tilde F = F(\tilde w,\grad\tilde w)$. Since $A=0$ in the statement of  Theorem \ref{T17}, it holds that $\|\tilde w\|_{C^{0,1}}\lesssim e^{-(3(\sigma+1) -\eps)t}$. In particular, since $\tilde F$ is essentially quadratic in $|\grad \tilde w|$, the third term on the right-hand side of the above identity is of higher order. Recalling the decay behavior of $\tilde a(t)$, we thus have $0 <\frac{d}{dt}\left(e^{(2(\sigma+1)+N)t} \tilde c(t)\right) \lesssim e^{N-2(\sigma+1)t}$. Because $N<2(\sigma+1)$, there exists thus a constant $c\in\R$ such that
\[
\left|\tilde c(t) - e^{-(2(\sigma+1) +N)t} c\right| \lesssim e^{-4(\sigma+1)t}.
\]
Hence, by the virtue of \eqref{44c} and the estimates on $\tilde a(t)$, $\tilde b(t)$, $\tilde A(t)$, the statement of the theorem follows from \eqref{50} and the eigenvalue expansion of $\tilde w$.
\end{proof}



\section*{Acknowledgement}
It is a great pleasure to acknowledge fruitful discussions with Clemens Kienzler and Herbert Koch.

\bibliography{pme_lit}
\bibliographystyle{acm}
\end{document}